\documentclass[12pt]{amsart}
\usepackage{amscd}     
\usepackage{amssymb}
\usepackage{amsmath, amsthm, graphics, dsfont}

\allowdisplaybreaks[3]
\usepackage{amsfonts}
\usepackage[normalem]{ulem}
\usepackage{hyperref}
\usepackage{tikz}
\usepackage{mathtools}
\usepackage{verbatim} 
\usepackage{MnSymbol}
\usepackage{tikz-cd}
\usepackage{subcaption}

\usepackage[margin=4.2cm]{geometry}
\usepackage{latexsym}
\usepackage{epsfig}
\usepackage{enumerate}
\usepackage{times}
\usepackage{xcolor}

\newtheorem{thm}{Theorem}[section]
\newtheorem{prop}[thm]{Proposition}
\newtheorem{lem}[thm]{Lemma}
\newtheorem{cor}[thm]{Corollary}

\newtheorem*{ThmA}{Theorem A}
\newtheorem*{ThmB}{Theorem B}
\newtheorem*{ThmC}{Theorem C}
\newtheorem*{ThmD}{Theorem D}
\newtheorem*{ThmE}{Theorem E}
\newtheorem*{ThmF}{Theorem F}
\newtheorem*{ThmG}{Theorem G}
\newtheorem*{ThmH}{Theorem H}
\newtheorem*{ThmK}{Theorem K}

\theoremstyle{remark}
\newtheorem{rem}[thm]{Remark}
\newtheorem{ex}[thm]{Example}
\newtheorem{defn}[thm]{Definition}

\newcommand{\V}{\mathds{V}}
\newcommand{\W}{\mathds{W}}
\newcommand{\II}{\mathcal{I}}

\newcommand{\E}{\mathsf{E}}

\newcommand{\x}{\mathbf{x}}

\DeclareMathOperator{\Z}{\mathbb{Z}}
\DeclareMathOperator{\R}{\mathbb{R}}
\DeclareMathOperator{\Q}{\mathbb{Q}}

\DeclareMathOperator{\Pic}{\mathrm{Pic}}

\DeclareMathOperator{\Cl}{\mathrm{Cl}}

\newcommand{\M}{\overline{M}_{0,n}}

\usetikzlibrary{calc}
\usetikzlibrary{graphs,graphs.standard,quotes}

\title[Polytopes of Effective Boundary Expressions of Divisors on $\M$]{Polytopes of Effective Boundary Expressions of Divisors on $\M$}

\author[Cavey]{Ian Cavey}
\address{Department of Mathematics\\University of Illinois Urbana-Champaign\\1409 W. Green Street (MC-382)\\Urbana, IL 61801\\ USA}
\email{cavey@illinois.edu}

\author[Genlik]{Deniz Genlik}
\address{Department of Mathematics\\University of Illinois Urbana-Champaign\\1409 W. Green Street (MC-382)\\Urbana, IL 61801\\ USA}
\email{genlik@illinois.edu}

\date{\today}

\begin{document}

\maketitle
\begin{abstract}
For a divisor on $\M$, we introduce the polytope of its effective boundary expressions. We establish structural properties of these polytopes under the forgetful maps of $\M$ forgetting marked points, and give equivalent graph-theoretic descriptions. We compute these polytopes for several families of divisors. For psi-classes and their pullbacks by forgetful maps, we show that the polytopes are unimodular simplices. For the log-canonical class and its modifications by psi-classes, we prove that the nonnegative parts of the corresponding polytopes recover spanning forest polytopes and the subtour elimination (Held--Karp relaxation) polytope of the symmetric traveling salesman problem. As an application, we obtain a Minkowski-like decomposition of the subtour elimination polytope into simplices. Finally, for symmetric level-one $\mathfrak{sl}_p$ conformal block divisors, we show that the defining inequalities are local Tur\'an bounds and the $0/1$-points are balanced Tur\'an graphs. Moreover, for $p=2$ and $p=n/2$, these polytopes recover the perfect matching and fractional perfect matching polytopes.
\end{abstract}

\section{Introduction}

The moduli space $\M$ is a smooth projective variety parametrizing stable rational complex curves with $n$ marked points \cite{Knudsen83II,Knudsen83III,Knudsen83I}. The interior
\begin{equation*}
M_{0,n}\subseteq \M
\end{equation*}
parametrizes irreducible such curves. The boundary $\partial \M = \M \setminus M_{0,n}$ is a normal crossings divisor with $2^{n-1}-n-1$ irreducible components,
\begin{equation*} \partial \M = \bigcup D_I, \end{equation*}
indexed by partitions $I \sqcup I^c = [n]$ of the set $[n]\coloneqq \{1,2,\ldots,n\}$ such that $2\leq \vert I  \vert , \vert I^c  \vert \leq n-2$, and where we consider $D_I = D_{I^c}$ \cite{Kapranov93,Keel}. We often index these boundary divisors by
\begin{equation*}
\II_n = \{ I\subseteq [n] \,  \vert n\notin I, \, 2 \leq \vert I \vert \leq n-2 \}.
\end{equation*}

The classes of the boundary divisors generate the divisor class group of $\M$ with certain linear relations among them, known as Keel's relations \cite{Keel}. Thus, divisor classes on $\M$ have many expressions as linear combinations of boundary divisors. Given a divisor $D$ on $\M$, we are interested in the set of all \emph{effective} linear combinations of boundary divisors such that the sum is linearly equivalent to $D$,
\begin{equation*}
  Q(D) = \left\{\, (a_I)_{I\in\II_n}\in\R^{\II_n} \, \bigg \vert \,
  [D] = \sum_{I\in\II_n} a_I\,[D_I],\ \ a_I\geq 0 \,\right\}.
\end{equation*}
We show that $Q(D)\subseteq\R^{\II_n}\cong\R^{\,2^{n-1}-n-1}$ is a (possibly empty) convex polytope which depends only on the class $[D]$, and ask the guiding question of this
paper:

\begin{quote}
\textit{How does the geometry of a divisor class $[D]$ on $\M$ appear in the combinatorics of its polytope of effective boundary expressions $Q(D)$?}   
\end{quote}

As we will show, the combinatorics arising from natural families of divisor classes on $\M$ are closely related to the polytopes of spanning trees and forests, Hamiltonian paths and cycles, perfect matchings, and Tur\'an subgraphs of complete graphs.

\subsection{Main structural results}

First, we establish the following structural results. 

\begin{ThmA}
For any divisors $D,D'$ on $\M$, we have the following:
    \begin{enumerate}
        \item (Lemma \ref{lem:subadditivity and scaling}) For any $c> 0$, we have
        \begin{equation*} Q(cD) = c \,Q(D), \text{ and } Q(D+D') \supseteq Q(D)+Q(D'). \end{equation*}
        \item (Proposition \ref{prop:pullback on polytopes}) The forgetful map $\pi \colon \overline{M}_{0,n+1}\to \M$ induces a bijection
        \begin{equation*} Q(D) \longrightarrow Q(\pi^*(D)) \end{equation*}
        given on linear combinations of boundary divisors by
        \begin{equation*} \sum_{I\in\II_n} a_I D_I \, \longmapsto \, 
        \sum_{I\in\II_n} a_I\big(D_I+D_{[n]\setminus I}\big).\end{equation*}
        \item (Proposition \ref{prop:inverse to projection}) The coordinate projection $(a_{I})\mapsto (a_{\{i,j\}})_{1\leq i < j \leq n-1}$ restricts to a bijection from $Q(D)$ onto its image.
    \end{enumerate}
\end{ThmA}

The last point in the above Theorem is a useful technical tool as it allows us to describe isomorphic polytopes in $n-1 \choose 2$ coordinates rather than recording all $2^{n-1}-n-1$ boundary divisor coefficients. This result also provides the conceptual bridge from divisor expressions to graph theory, as we may therefore identify each boundary divisor expression for $D$ with a weighting $(a_{\{i,j\}})$ of the edges $e= \{i,j\}$ of the complete graph $K_{n-1}$. We therefore write $\E_{n-1} = \{ \{i,j\} \,  \vert \, 1 \leq i < j \leq n-1 \}$ for the edge set of $K_{n-1}$, and index these coordinates as $a_e$ rather than $a_{\{i,j\}}$ to emphasize this connection.

Throughout this paper, we make use of several coordinatizations of $Q(D)$. We write $P(D)$ and $P^+(D)$ for the projections of $Q(D)$ onto coordinates $(a_e)_{e \in \E_{n-1}}$ and $(a_e)_{e \in \E_n}$ respectively. We also frequently use the change of coordinates $x_e = 1-a_e$ on both $P(D)$ and $P^+(D)$, and denote the resulting polytopes in $x_e$-coordinates by $\widetilde{P}(D)$ and $\widetilde{P}^+(D)$ respectively. It follows from Theorem A that all these polytopes are affinely isomorphic, and we summarize these projections and changes of coordinates in the following diagram:
\begin{equation}
\begin{tikzcd}[column sep=small]
    & Q(D)\arrow[d] \\
    \R^{\E_n} &\arrow[l,phantom,"\supseteq"] P^+(D)\arrow[d] \arrow[rrr,<->,"x_e=1-a_e"] &&& \widetilde{P}^+(D)\arrow[d]\\
    \R^{\E_{n-1}}&\arrow[l,phantom,"\supseteq"] P(D) \arrow[rrr,<->,"x_e=1-a_e"] &&& \widetilde{P}(D).\\
\end{tikzcd}
\end{equation}
Each divisor class on $\M$ has a unique expression in the Kapranov basis centered at the $n^\text{th}$ marking:

\begin{equation*}[D] = \alpha(D) \psi_n + \sum_{I\in \II_n, \vert I  \vert \geq 3} \beta_I(D) [D_I]. \end{equation*}

For notational convenience, we also write $\beta_I(D) = 0$ for any divisor $D$ whenever $I\in \II_n$ has $ \vert I  \vert =2$. The following theorem gives the defining inequalities of these polytopes in terms of these coefficients.

\begin{ThmB}
For any divisor $D$ on $\overline{M}_{0,n}$, we have
    \begin{equation*}
    P(D) = \left\{ (a_e)_{e \in \E_{n-1}}\in \R^{\E_{n-1}} \, \bigg \vert \, \sum_{e \in \E_{n-1}} a_{e} = \alpha(D), \beta_I(D) + \sum_{e \in \E_I}a_{e}\geq 0\, \text{ for all }I\in \II_n \right\}\quad \text{(Corollary \ref{cor:P(D) inequalities})}
    \end{equation*}
or equivalently under the change of variable $a_e \mapsto x_e=1-a_e$,
    \begin{equation*}\widetilde{P}(D) = \left\{ (x_e)_{e \in \E_{n-1}}\in \R^{\E_{n-1}} \bigg \vert \, \sum_{e \in \E_{n-1}}x_e = r_{[n-1]}(D),\, \sum_{e \in \E_I} x_e \leq r_I(D)\, \text{ for all }I\in \II_n \right\}\quad \text{(Proposition \ref{prop:P tilde inequalities})}
    \end{equation*}
where
\begin{equation*}
r_I(D)={ \vert I \vert \choose 2} + \beta_I(D) \quad \text{and} \quad r_{[n-1]}(D)={n-1 \choose 2}- \alpha(D).
\end{equation*}
\end{ThmB}

This description of $\widetilde{P}(D)$ resembles that of the base polytope of an extended polymatroid \cite[Chapter 44]{Schrijver03}. We note, however, that the polytopes $P(D)$ and $\widetilde{P}(D)$ can be non-integral even when $D$ is an integral divisor class. See Figure \ref{fig:nonintegral vertices} and Remark \ref{rem:matroid_vs_ours}.

For the applications that follow, we frequently use the following structural result relating the polytopes associated to a divisor $D$ on $\M$ and its pullback to $\overline{M}_{0,n+1}$ by the forgetful map.

\begin{ThmC}
(Corollary \ref{cor: P^+ and P relationship}) For a divisor $D$ on $\M$, we have
\begin{equation*}
P^+(D)=P(\pi^*(D)) 
\end{equation*}
where $\pi  \colon \overline{M}_{0,n+1}\to \M$ is the forgetful map forgetting the last marked point.
\end{ThmC}

\begin{rem}[Index conventions]\label{rem:index conventions}
Theorems A--C concern a divisor $D$ on $\M$, so that $P(D),\widetilde{P}(D)\subseteq
\R^{\E_{n-1}}$ and $P^{+}(D),\widetilde{P}^{+}(D)\subseteq\R^{\E_{n}}$. In applications, it is generally more convenient to place the divisor on
$\overline{M}_{0,n+1}$, so that $P(D)$ is indexed by the edges of $K_n$.  In that
case,
\begin{equation*}
\II_{n+1}=\{\,I\subseteq[n]\, \vert \,2\leq  \vert I \vert \leq n-1\,\},
\quad
r_I(D)=\binom{ \vert I \vert }{2}+\beta_I(D),
\quad
r_{[n]}(D)=\binom{n}{2}-\alpha(D).
\end{equation*}
We do not impose a global shift of indices, and each result below specifies its ambient
moduli space.
\end{rem}

\subsection{Main applications: Natural divisors on $\M$ and graph theory polytopes}

The structural results above convert the Kapranov basis coefficients of a divisor into linear constraints on edge weights of a complete graph.  For the natural divisor classes considered below, these constraints coincide with familiar polyhedral descriptions from graph theory and combinatorial optimization. Conversely, the integral and $0/1$-points of the resulting polytopes produce explicit effective boundary expressions of the corresponding divisor classes.

\subsubsection{Psi-classes and their pullbacks}

Our first application is to the psi-classes and their pullbacks by forgetful maps. The class $\psi_i$ is defined as the first Chern class of the $i^{\text{th}}$ tautological line bundle (see Section \ref{sec:background}). More generally, we consider Kapranov classes $X_{S,i}$ indexed by subsets $S\subseteq [n]$ with $i\in S$ and $ \vert S \vert \geq 3$, defined as the pullback of $\psi_i$ to $\M$ via the map $\M\to \overline{M}_{0,S}$ forgetting the marked points not in $S$.

Intersection numbers of classes are known as Kapranov degrees \cite{BELL}:
\begin{equation*}
\int_{\overline{M}_{0,n}} X_{S_1, i_1} \cdots X_{S_{n-3}, i_{n-3}}
\end{equation*}
which have an increasing number of recent combinatorial, algebraic and geometric applications \cite{CGM,GGL, GGL2,Reinke26, Silversmith22, silversmith24}.

\begin{ThmD} (Theorem \ref{thm:Kapranov polytopes} and Corollary \ref{cor:Kapranov P plus})
    Let $S \subseteq [n]$ with $ \vert S \vert \geq 3$ and $i \in S$. For the  Kapranov class $X_{S,i}$ on $\M$, the polytope $Q(X_{S,i})$ is a unimodular simplex of dimension
    \begin{equation*}
    { \vert S  \vert -1\choose 2}-1.
    \end{equation*}
    The vertices of $Q(X_{S,i})$ are indexed by pairs $\{j,k\}\subseteq S\setminus\{i\}$, and correspond to boundary divisor expressions for $X_{S,i}$ are given by
    \begin{equation*}
    X_{S,i}= \sum_{\substack{I\subseteq[n]\,: i\notin I,\ j,k\in I \\ 2\leq \vert I\cap S \vert \leq \vert S  \vert -2}}[D_I].
    \end{equation*}
\end{ThmD}

\subsubsection{The log-canonical class $\kappa$, spanning trees and Hamiltonian cycles}

Next, we study the log-canonical class 
\begin{equation*}\kappa = K_{\M}+\partial \M,\end{equation*}
a natural ample divisor on $\M$. For clarity, we also write $\kappa^{(n)}$ for the log-canonical class on $\M$. The following theorem relates $\kappa$ to the spanning forest polytopes of $K_n$, defined as convex hulls of the $0/1$-points in $\R^{\E_n}$ corresponding to spanning forests of $K_n$ with a fixed number of connected components. We also refer to the \emph{subtour elimination polytope} (or \emph{Held--Karp polytope}) of $K_n$. This polytope is a well-known relaxation of the Hamiltonian cycle polytope of $K_n$ in combinatorial optimization, and is often used to approximate solutions to the traveling salesman problem \cite{Schrijver03}. This polytope is non-integral in general, and its integer points are $0/1$-points encoding Hamiltonian cycles of $K_n$.

\begin{ThmE}
(Corollary \ref{cor:spanning tree polytope and kappa})
    The following hold:
    \begin{enumerate}
        \item For any integer $t\geq 0$, \begin{equation*}\widetilde{P}(\kappa^{(n+1)}+t\psi_{n+1})\cap \R^{\E_n}_{\geq 0}\end{equation*} is the $(t+1)$-connected component spanning forest polytope of $K_n$. 
        \item The polytope 
        \begin{equation*} \widetilde{P}^+(\kappa^{(n)})\cap \R^{\E_n}_{\geq 0} = \widetilde{P}(\kappa^{(n+1)}-\psi_{n+1})\cap\R^{\E_n}_{\geq 0} \end{equation*}
        is the subtour elimination polytope of $K_n$.
    \end{enumerate}
\end{ThmE}

The truncation $x_e\geq 0$ in this theorem corresponds to the bound $a_e\leq 1$ on the coefficients of the divisors $D_e = D_{\{i,j\}}$ in the boundary divisor expressions.

For $I\subseteq[n]$, let $K_I$ denote the complete graph on vertex set $I$. Let $k(G)$ denote the number of connected components of a graph $G$. As a consequence of Theorem E, we obtain effective boundary expressions for the log-canonical class indexed by spanning trees and Hamiltonian cycles.  

\begin{ThmF}
    Let $T$ be a spanning tree and $C$ be a Hamiltonian cycle on $K_n$. Then, on $\overline{M}_{0,n+1}$, we have
    \begin{equation*}
        \kappa^{(n+1)}= \sum_{I\in \II_{n+1}} \left( k(T\cap K_I) -1\right)[D_I], \quad \text{(Theorem \ref{thm:tree expression for kappa})}
    \end{equation*}
and on $\M$, we have
    \begin{equation*}
    \kappa^{(n)} = \sum_{I\in \II_n} (k(C\cap K_I)-1)[D_I].  \quad \text{(Theorem \ref{thm:Hamiltonian cycle formula for kappa})}.
    \end{equation*}
\end{ThmF}

\subsubsection{A decomposition of the Held--Karp relaxation polytope of traveling salesman problem}

The log-canonical class on $\M$ has a decomposition into a sum of certain Kapranov classes
\begin{equation*} \kappa = \omega_4+\omega_5+\cdots+\omega_n, \end{equation*}
where $\omega_i = X_{[i],i}$. By Theorem A, there is therefore an inclusion 
\begin{equation*} Q(\kappa) \supseteq Q(\omega_4)+Q(\omega_5)+\cdots+Q(\omega_n). \end{equation*}
This inclusion is strict, in general, meaning that there are effective boundary expressions for $\kappa$ that cannot be decomposed into a sum of such expressions for the $\omega_i$'s. The following result states such decompositions do exist for expressions corresponding to points in the subtour elimination polytope under Theorem E.

\begin{ThmG}
Let $\mathbf{1}\in\R^{\E_n}$ denote the all-ones vector.  Then, we have:
\begin{enumerate}
\item (Theorem \ref{thm:kappa omega decomp}) The cube-truncated polytope of effective boundary expressions of $\kappa$ has
      the Minkowski-like decomposition
      \begin{equation*}
      P^+(\kappa)\cap[0,1]^{\E_n}
      =\left(P^+(\omega_4)+\cdots+P^+(\omega_n)\right)
       \cap[0,1]^{\E_n}.
      \end{equation*}
\item (Corollary
\ref{cor:Held-Karp_Decomposition}) The graph complement $\mathbf{1}-P_{\mathrm{SEP},n}$ of the subtour elimination polytope $P_{\mathrm{SEP},n}$ satisfies the following Minkowski-like decomposition:
\begin{equation*}
\mathbf{1}-P_{\mathrm{SEP},n}= \left( P^+(\omega_4)+P^+(\omega_5)+\cdots+P^+(\omega_n) \right)\cap [0,1]^{\E_n}.
\end{equation*}
Equivalently,
\begin{equation*}
P_{\mathrm{SEP},n}
= \left( \mathbf{1}- \left( P^+(\omega_4)+P^+(\omega_5)+\cdots+P^+(\omega_n) \right) \right) \cap\R_{\geq0}^{\E_n}.
\end{equation*}
\end{enumerate}
\end{ThmG}

The proof of part (2) uses Lov\'asz's Splitting-off Theorem \cite{Lovasz1976}. Under the identification of Theorem E, the $0/1$-points of $P^+(\kappa)\cap[0,1]^{\E_n}$ are graph complements of Hamiltonian cycles in $K_n$. On these points, the decomposition in Theorem G has a combinatorial interpretation that successively splitting-off vertices of the Hamiltonian cycle records the corresponding simplex summands $P^+(\omega_i)$. This splitting-off operation gives a description of the Held--Karp relaxation polytope as the nonnegative orthant truncation of a Minkowski sum of simplices. To our knowledge, such a description of this polytope was not known previously.

\subsubsection{Level-one $\mathfrak{sl}_p$ conformal block divisors and Tur\'an graphs }

Conformal blocks are vector bundles on moduli spaces $\overline{M}_{g,n}$ of stable $n$-pointed curves of genus $g$, that are attached to a simple Lie algebra, a level, and a tuple of dominant weights. Their first Chern classes are called \emph{conformal block divisors} \cite{Fakhruddin12}. We focus on the level-one conformal block divisors $D^{\mathfrak{sl}_p}_{1,(1^n,0)}$ on $\overline{M}_{0,n+1}$ \cite{F11,GiansiracusaGibney12}.

On the graph-theoretic side, the Tur\'an number $\operatorname{ex}(m,K_{p+1})$ is the maximum number of edges in a graph on $m$ vertices that does not contain a subgraph isomorphic to $K_{p+1}$. Tur\'an's theorem states that equality is attained precisely by complete $p$-partite graphs whose part sizes differ by at most one \cite{Turan41,Turan54}. Our next theorem shows that the inequalities for the polytope of  effective boundary representations for $D^{\mathfrak{sl}_p}_{1,(1^n,0)}$ are exactly these local Tur\'an bounds.

\begin{ThmH}
Let $n\geq4$ and  $2\leq p \leq n$ be integers with $p\mid n$, and let $D^{\mathfrak{sl}_p}_{1,(1^n,0)}$ denote the level-one $\mathfrak{sl}_p$ conformal block divisor on $\overline{M}_{0,n+1}$ attached to the weight vector $(1^n,0)$.
Then:
\begin{enumerate}
\item (Proposition \ref{prop:P_tilde_sl_p_weights_all_1s}) The polytope of effective boundary expressions of
$D^{\mathfrak{sl}_p}_{1,(1^n,0)}$ is, in the complemented coordinates,
\begin{equation*}
  \widetilde{P}\left(D_{1,(1^n,0)}^{\mathfrak{sl}_p}\right)
  =\left\{ \x\in \R^{\E_{n}} \, \middle \vert \, 
  \sum_{e\in\E_n}x_e = \operatorname{ex}(n,K_{p+1}),\ \
  \sum_{e\in\E_I}x_e\leq \operatorname{ex}( \vert I \vert ,K_{p+1})
  \ \text{ for all } I\in\II_{n+1} \right\},
\end{equation*}
where $\mathrm{ex}(m,K_{p+1})$ is the Tur\'an number.
\item (Corollary \ref{cor:01points_of_slp_conformal}) The $0/1$-points of $\widetilde{P}\left(D_{1,(1^n,0)}^{\mathfrak{sl}_p}\right)$
are exactly the incidence vectors of the balanced Tur\'an graphs
\begin{equation*}
T_{A_1,\ldots,A_p}=K_{A_1,\ldots,A_p}\cong K_{\frac{n}{p},\ldots,\frac{n}{p}},
   \qquad [n]=A_1\sqcup\cdots\sqcup A_p.
\end{equation*}
\item (Corollary \ref{cor:slp_expressions_from_Turan_graphs}) Fix such a balanced Tur\'an graph $T_{A_1,\ldots,A_p}$.  For $I\in\II_{n+1}$ put
\begin{equation*}
\sigma_i(I)\coloneqq   \vert A_i\cap I  \vert \quad \text{for }1\leq i\leq p,
\end{equation*}
and let $\tau_1(I),\ldots,\tau_p(I)$ be the part sizes of the Tur\'an graph $T( \vert I  \vert ,p)$, that
is, equal to either $\lceil \vert I  \vert /p\rceil$ or $\lfloor \vert I  \vert /p\rfloor$.  Then, the corresponding
effective boundary expression is
\begin{equation*}
  D_{1,(1^n,0)}^{\mathfrak{sl}_p}=\sum_{I \in \II_{n+1}}c_I\left(T_{A_1,\ldots,A_p}\right)\left[D_I \right],
  \qquad
  c_I\left(T_{A_1,\ldots,A_p}\right)=\sum_{i=1}^p\left(\binom{\sigma_i(I)}{2}-\binom{\tau_i(I)}{2}\right).
\end{equation*}
\end{enumerate}
\end{ThmH}

Part (3) of Theorem H gives an explicit family of effective boundary expressions indexed by balanced $p$-partitions of the markings. To the best of our knowledge, these boundary expressions are new, even though these conformal block divisor classes themselves are well-known.

Finally, we recall two matching polytopes.  Assume that $n$ is even.  A \emph{perfect matching} of $K_n$ is a collection of $n/2$ pairwise disjoint edges covering every vertex, and the perfect matching polytope $P_{\mathrm{PM},n}$ is the convex hull of their incidence vectors. The \emph{fractional perfect matching polytope} $P_{\mathrm{FPM},n}$ is a relaxation of the perfect matching polytope $P_{\mathrm{PM},n}$. One has $P_{\mathrm{PM},n}\subseteq P_{\mathrm{FPM},n}$, and their $0/1$ points are precisely the perfect matchings. Unlike the perfect mathcing polytope $P_{\mathrm{PM},n}$, the fractional polytope $P_{\mathrm{FPM},n}$ can have nonintegral vertices \cite{Balinski65,Ed65,Schrijver03}. The following result shows that these polytopes naturally emerge in the geometry of $\mathfrak{sl}_2$ level-one conformal block divisors.

\begin{ThmK}
Let $n\geq4$ be an even integer. Then, for the divisor $D_{1,(1^n,0)}^{\mathfrak{sl}_2}$ on $\overline{M}_{0,n+1}$ we have:
\begin{enumerate}
\item (Theorem \ref{thm:perfect_matching_and_sl2}) The nonnegative truncation of $\widetilde{P}\left(2D_{1,(1^n,0)}^{\mathfrak{sl}_2}\right)$
is the perfect matching polytope of $K_n$:
\begin{equation*}
  \widetilde{P}\left(2D_{1,(1^n,0)}^{\mathfrak{sl}_2}\right)\cap\R^{\E_n}_{\geq0}
  =P_{\mathrm{PM},n}.
\end{equation*}
\item (Theorem \ref{thm:fractional_perfect_and_sl2}) The polytope of effective boundary expressions of
$D_{1,(1^n,0)}^{\mathfrak{sl}_{n/2}}$ is the fractional perfect matching polytope of
$K_n$:
\begin{equation*}
  P\left(D_{1,(1^n,0)}^{\mathfrak{sl}_{\frac{n}{2}}}\right)=P_{\mathrm{FPM},n}.
\end{equation*}
\end{enumerate}
\end{ThmK}

\subsection{Acknowledgments} We thank David Anderson, Karthekeyan Chandrasekaran, Chandra Chekuri, June Huh, David Jensen, Siddarth Kannan, Matt Larson and Rob Silversmith for useful discussions. D.\,G. is supported by an AMS--Simons Travel Grant.

\subsection{AI disclosure}

The authors used generative AI tools, including Claude and ChatGPT, for literature discovery, exploratory computations in small cases, and occasional language and grammar editing. The authors take full responsibility for all the content in this work.

\section{Background}\label{sec:background}

This section collects some of the geometric and combinatorial background used throughout the paper. We first recall the divisor theory of $\M$, including boundary divisors,
tautological classes, forgetful morphisms, and the Kapranov basis.  We then fix
graph theoretic notation and review the linear programming polytopes that arise
from our descriptions of effective boundary expressions. The remaining background material will be provided as needed throughout the paper.

\subsection{Background on $\M$}

We recall only the facts about $\M$ that are needed for the construction and analysis of effective boundary expression polytopes.

\subsubsection{Divisors on $\M$}
Throughout the paper, $n\geq 4$ and we work over $\mathbb{C}$, and all divisor classes are taken with real
coefficients. We write $\Cl(\M)_{\R}=\Cl(\M)\otimes_{\Z}\R$. Since $\M$ is smooth \cite{Knudsen83II}, we have $\Cl(\M)=\Pic(\M)$. Moreover, $\Pic(\M)$ is free of rank
\begin{equation*}
2^{n-1}-\binom{n}{2}-1.
\end{equation*}
The cycle class map induces an isomorphism $\Pic(\M)_{\Q} \cong H^2(\M,\Q)$ \cite{Keel}. Equivalently, after tensoring with $\R$, we identify
\begin{equation*}
\Cl(\M)_{\R}=\Pic(\M)_{\R}\cong H^2(\M,\R)\cong N^1(\M)_{\R}.
\end{equation*}
Thus, throughout the paper, we work with divisor classes in
$\Cl(\M)_{\R}$, write $[D]$ for the class of a divisor expression $D$,
and pass between divisor classes, cohomology classes, and numerical divisor
classes without further warning.

Boundary divisors on $\M$ are indexed by partitions $I\sqcup I^c = [n]$ with $ \vert I  \vert , \vert I^c  \vert \geq 2$. The corresponding boundary divisor is denoted by either $D_I$ or $D_{I^c}$. We define two indexing sets for boundary divisors on $\M$,
\begin{equation*}
\II_n \coloneqq   \{ I\,  \vert \, I\subseteq [n-1], 2\leq \vert I  \vert \leq n-2 \}\quad \text{and}\quad \II_n^+ \coloneqq   \{ I\,  \vert \, I\subseteq [n], 2\leq \vert I  \vert \leq n-2 \}.
\end{equation*}

The set $\II_n$ indexes boundary divisors on $\M$ irredundantly by always using the subset not containing $n$. The set $\II_n^+$ indexes the boundary divisors on $\M$ redundantly, but is more convenient for writing certain formulas.

The moduli space $\overline{M}_{0,n}$ has a universal curve
\begin{equation*}
\epsilon:\overline{U}_{0,n}\longrightarrow \overline{M}_{0,n}.
\end{equation*}
Let $\omega_{\epsilon}$ denote the relative dualizing sheaf, and let 
$\sigma_i:\overline{M}_{0,n}\longrightarrow \overline{U}_{0,n}$ denote the $i^{\text{th}}$ 
tautological section arising from the $i^{\text{th}}$ marked point. For $i\in [n]$, the 
\emph{$i^{\text{th}}$ tautological line bundle} $\mathcal{L}_i$ is then defined by the 
pullback
\begin{equation*}
\mathcal{L}_i\coloneqq   \sigma_i^*\omega_{\epsilon},
\end{equation*}
and the \emph{$i^{\text{th}}$ psi-class} (also called the \emph{$i^{\text{th}}$ 
tautological class}) is its first Chern class
\begin{equation*}
\psi_i\coloneqq   c_1(\mathcal{L}_i).
\end{equation*}

Let $S\subseteq [n]$ with $ \vert S \vert \geq 3$. Let $\overline{M}_{0,S}$ denote $\overline{M}_{0, \vert S \vert }$ where the points are labeled by elements of  $S$, and let
\begin{equation*}
\pi_{S^{c}}: \overline{M}_{0,n}\longrightarrow \overline{M}_{0,S}
\end{equation*}
be the forgetful map forgetting the marked points in $S^{c}\coloneqq   [n]
\setminus S$. For $i\in S$, define\footnote{The intersection numbers
    $\displaystyle \int_{\overline{M}_{0,n}} X_{S_1, i_1} \cdots X_{S_{n-3}, i_{n-3}}$
are called \emph{Kapranov degrees} in \cite{BELL}. We choose the terminology \emph{Kapranov class} to match theirs.} the \emph{Kapranov class} $X_{S,i}$ associated to the pair $(S,i)$ to be the pullback of $\psi_i$ on $\overline{M}_{0,n}$ by $\pi_{S^{c}}$:
\begin{equation*}
X_{S,i}\coloneqq   \pi_{S^{c}}^{*}(\psi_i).
\end{equation*}
A special family of these classes is the \emph{omega classes}. The $i^{\text{th}}$ \emph{omega class} is defined as
\begin{equation*}
\omega_i \coloneqq   X_{[i],i}.
\end{equation*}
These classes arise naturally in connection with the Keel-Tevelev embedding \cite{KT} of $\overline{M}_{0,n}$ into $\mathbb{P}^{1}\times \cdots \times\mathbb{P}^{n-3}$, and their top intersection numbers compute the multidegrees of this embedding \cite{CGM}. These intersection numbers admit combinatorial descriptions in terms of column-restricted parking functions \cite{CGM} and lazy tournaments \cite{GGL2}.  More generally, products of psi and omega classes admit positive, multiplicity-free expressions in boundary strata via slide rules \cite{GGL}.  The omega-class intersection numbers are also special cases of the more general Kapranov degrees studied in \cite{BELL}.
\begin{rem}
Throughout, for $m\in[n]$ we write $\pi_m$ for the forgetful map dropping the marked
point labelled $m$. Thus, $\pi_{n+1} \colon \overline{M}_{0,n+1}\to\M$ and
$\pi_{n} \colon \M\to\overline{M}_{0,n-1}$, and $\pi_m=\pi_{\{m\}}$ in the notation
$\pi_{S^c}$ introduced above.
\end{rem}

Let $\pi_{n+1}:\overline{M}_{0,n+1}\longrightarrow \overline{M}_{0,n}$ be the forgetful map forgetting the last marked point. The \emph{first kappa class} on $\overline{M}_{0,n}$ is defined by
\begin{equation*}
\kappa_1={\pi_{n+1}}_{\star}(\psi_{n+1}^{2})\in H^{2}(\overline{M}_{0,n},\mathbb{Q}).
\end{equation*}

\begin{rem}\label{rem:kappa_1_and_kappa_same}
The class $\kappa_1$ on $\M$ matches with the log-canonical class $\kappa=K_{\M}+\partial \M$. For this reason, in the rest of the paper we will refer to it as $\kappa$ or $\kappa^{(n)}$ if we would like to emphasize that it is a divisor class on $\overline{M}_{0,n}$.  
\end{rem}

\subsubsection{Relations among the divisors of $\M$}
The linear relations among the boundary divisors are given by the following relations, known as Keel's relations \cite{Keel}.

\begin{lem}[\cite{Keel}]\label{lem:Keels_quadruple_relations}
The relations among the linear equivalence classes of boundary divisors on $\M$ are generated by
\begin{equation*}
\sum_{\substack{i,j \in I \\ k,l \notin I}} \left[ D_I\right ]
=
\sum_{\substack{i,k \in I \\ j,l \notin I}} \left[ D_I\right]
=
\sum_{\substack{i,l \in I \\ j,k \notin I}} \left [ D_I \right]
\end{equation*}
for any four distinct elements $ i, j, k, l \in \{1, 2, \ldots, n\}$. All sums here are over $I\in \II_n^+$ satisfying the stated containment conditions.
\end{lem}

\begin{lem}\label{lem:some_relations_for_kappa_psi_and_D_A}
The following relations of divisor classes hold.
\begin{enumerate}
\item[$\bullet$] For any distinct $i,j\in[n]$, on $\M$:
\begin{align}
\kappa &=\sum_{I\in\II_n^{+}:\,i,j\notin I}( \vert I  \vert -1)\left[ D_{I}\right],\label{eqn:kappa_1_wrt_divisors}
\end{align}
\item[$\bullet$] For any distinct $i,j,k\in[n]$, on $\M$:
\begin{align}
\psi_k &=\sum_{\substack{I\in\II_n^{+}:\,k\in I\\ i,j\notin I}}\left[ D_{ I}\right],\label{eqn:psi_wrt_divisors}
\end{align}
\item[$\bullet$] For any $k\in[n]$
\begin{align}
\pi_{n+1}^{\star}(\psi_k) &= \psi_k-[D_{ \{k,n+1\}}],\label{eqn:psi_pullback}
\end{align}
\item[$\bullet$] Relating $\overline{M}_{0,n-1}$ and $\M$ via $\pi_{n} \colon \M\to\overline{M}_{0,n-1}$:
\begin{align}
\kappa^{(n)} &= \pi_{n}^{\star}(\kappa^{(n-1)})+\psi_n.\label{eqn:kappa_pull_back_formula}
\end{align}
\end{enumerate}
Finally, for
$I\in\II_n^{+}$ one has the relation of Cartier divisors on $\overline{M}_{0,n+1}$:
\begin{align}
\pi_{n+1}^{\star}(D_{I}) &=D_{I}+D_{I\cup \{n+1\}}.\label{eqn:boundary_pullback}
\end{align}
\end{lem}

\begin{proof}
These relations are standard, see \cite{Arbarello11} for a textbook account.  Taking Remark \ref{rem:kappa_1_and_kappa_same} into account equations
\eqref{eqn:kappa_1_wrt_divisors} and \eqref{eqn:psi_wrt_divisors} follow from \cite[Theorem 2.2]{Arbarello98}. The identity \eqref{eqn:kappa_pull_back_formula} is the case
$a=1$ of \cite[Equation (1.10)]{Arbarello96}. Lastly, for equations \eqref{eqn:psi_pullback},
\eqref{eqn:boundary_pullback}, see \cite[Equation (6)]{Arbarello87}.
\end{proof}

Omega classes provide the following natural decomposition of the class $\kappa$.

\begin{cor}\label{cor:kappa_divisor_decomoposition_into_omega_i_s}
For $n\geq 4$, the following relation holds on $\overline{M}_{0,n}$:
\begin{equation}
\kappa=\omega_4+\omega_5+\cdots+\omega_n.
\end{equation}
\end{cor}

\begin{proof}
We will prove the statement by induction. Firstly, notice that we have $\omega_n=\psi_n$ on $\overline{M}_{0,n}$. By equations \eqref{eqn:kappa_1_wrt_divisors} and \eqref{eqn:psi_wrt_divisors} in Lemma \ref{lem:some_relations_for_kappa_psi_and_D_A}, we have $\kappa=\psi_4$ on $\overline{M}_{0,4}$. Hence, we get the base case $n=4$ of the induction. Now, assume that the statement holds for $n-1$. Then, by the functoriality of pullback and the fact $\omega_n=\psi_n$, we have
\begin{equation*}
\kappa^{(n)} = \pi_{n}^{\star}(\kappa^{(n-1)})+\psi_n
= \pi_{n}^{\star}(\omega_4+\cdots+\omega_{n-1})+\omega_n
= \omega_4+\cdots+\omega_{n-1}+\omega_n
\end{equation*}
which completes the inductive step of the proof.
\end{proof}

\begin{rem}
Throughout this paper, we will use the Kapranov basis for $\Cl(\M)_{\R}$, which is
\begin{equation*}
\{ \psi_n \}\cup \{ [D_I] : I\in \II_n, \vert I  \vert \geq 3 \}.
\end{equation*}
\end{rem}

\begin{lem}\label{lem:psi_i_in_kapranov}
For $i\in [n]$, on $\overline{M}_{0,n+1}$, we have the following Kapranov basis expression for $\psi_i$:
\begin{equation*}
\psi_i = (n-2)\psi_{n+1} - \sum_{\substack{I \in \II_{n+1}, \vert I \vert \geq 3\\ i\in I}} ( \vert I  \vert -2)D_{I}.
\end{equation*}
\end{lem}

\begin{proof}
From equation (1) in the proof of \cite[Proposition 1.2]{bruno11}, after replacing the number of markings there by $n+1$ and setting $j=n+1$, we obtain, for $i\in[n]$,
\begin{equation*}
\psi_i
=(n-2)\psi_{n+1}
-\sum_{\substack{J\subseteq[n]\,:\, i\notin J\\ 1\leq  \vert J \vert \leq n-2}}
 \left(n-1- \vert J\cup\{n+1\} \vert  \right)D_{J\cup\{n+1\}}.
\end{equation*}
Since $ \vert J\cup\{n+1\} \vert = \vert J \vert +1$, the coefficient in the sum is $n-2- \vert J \vert $.
Set $I=[n]\setminus J$.  Then, $i\notin J$ is equivalent to $i\in I$ and
\begin{equation*}
D_{J\cup\{n+1\}}
=D_{[n+1]\setminus(J\cup\{n+1\})}
=D_I.
\end{equation*}
Moreover,
\begin{equation*}
n-2- \vert J \vert =n-2-(n- \vert I \vert )= \vert I \vert -2.
\end{equation*}
Consequently,
\begin{equation*}
\psi_i
=(n-2)\psi_{n+1}
-\sum_{\substack{I\in\II_{n+1}\\ i\in I}}
( \vert I \vert -2)D_I.
\end{equation*}
The terms with $ \vert I \vert =2$ have coefficient zero, so the sum may be restricted to $ \vert I \vert \geq3$, as claimed.
\end{proof}

\begin{lem}\label{lem:kappa in kapranov basis}
The expansion of the log-canonical class $\kappa$ on $\overline{M}_{0,n+1}$ in the Kapranov basis is given by
\begin{equation*}
 \kappa = {n-1\choose 2}\psi_{n+1}- \sum_{I\in \II_{n+1}, \vert I \vert \geq 3} { \vert I \vert -1\choose 2}[D_I].
\end{equation*}
\end{lem}

\begin{proof}
    Taking $i=n$ and $j=n+1$ in equation \eqref{eqn:kappa_1_wrt_divisors}, we have
    \begin{equation*}\kappa = \sum_{\substack{I\in \II_{n+1}^+\\ I\not\ni n,n+1}} ( \vert I  \vert -1)[D_I] =  \sum_{\substack{I\in \II_{n+1}\\ I\not\ni n}} ( \vert I  \vert -1)[D_I]. \end{equation*}
    Summing the boundary expressions for $\psi_{n+1}$ in equation \eqref{eqn:psi_wrt_divisors} over all choices of indices $1\leq i < j \leq n-1$, we also have
    \begin{align*}
        {n-1 \choose 2}\psi_{n+1} & = \sum_{1\leq i < j \leq n-1} \sum_{\substack{I\in \II_{n+1}\\ i,j\in I}}[D_I]\\
        & = \sum_{I\in \II_{n+1}} { \vert I\cap [n-1] \vert \choose 2}[D_I]\\
        & = \sum_{\substack{I\in \II_{n+1}\\ n\notin I}} { \vert I \vert \choose 2} [D_I]+\sum_{\substack{I\in \II_{n+1}\\ n\in I}} { \vert I  \vert -1 \choose 2} [D_I].
    \end{align*}
    Combining these, we obtain the desired relation
    \begin{equation*}{n-1 \choose 2}\psi_{n+1} - \kappa = \sum_{\substack{I\in \II_{n+1}\\ n\notin I}} \left( { \vert I \vert \choose 2}-( \vert I  \vert -1)\right) [D_I]+\sum_{\substack{I\in \II_{n+1}\\ n\in I}} { \vert I  \vert -1 \choose 2} [D_I] = \sum_{I\in \II_{n+1}} { \vert I  \vert -1 \choose 2} [D_I]. \end{equation*}
\end{proof}

\subsection{Background on some graph theoretic and linear programming notions}

The pair-coordinate models developed later identify the coefficients of the divisors $D_{\{i,j\}}$ with weights on the edges of a complete graph.  We therefore recall the graph-theoretic notation needed to express the defining inequalities of our polytopes, followed by the linear programming relaxation of the symmetric traveling salesman problem that appears in the study of the log-canonical class.

\subsubsection{Graph theory notations}
Unless explicitly stated otherwise, all graphs in this subsection are finite, undirected, and simple, with a specified spanning vertex set. Multigraphs will appear only later, in connection with splitting-off operations, when we consider graphs with parallel edges.

We write
\begin{equation*}
\E_n \coloneqq \{ \{i,j\} \,  \vert \, 1 \leq i < j \leq n \}
\end{equation*}
for the edge set of the complete graph $K_n$ on vertices labeled by $[n] = \{1,2,\dots,n\}$. For a subgraph $G\subseteq K_n$, we write $\E(G)$ for its edge set, $\mathds{1}_G\in\{0,1\}^{\E_n}$ for its \emph{incidence vector}, defined by 
\begin{equation*}
(\mathds{1}_G)_e=
\begin{cases}
1 & \text{if }e \in \E(G),\\
0 & \text{if }e \notin \E(G),
\end{cases}
\end{equation*}
and $\overline{G}$ for the complement of $G$ in $K_n$, so that $\mathds{1}_{\overline{G}}=\mathbf{1}-\mathds{1}_G$ where $\mathbf{1}=(1,\dots,1)\in\R^{\E_n}$.

For $I \subseteq [n]$, we also write $K_I\subseteq K_n$ for the complete subgraph on the vertices labeled by $I$, and $\E_I$ for the edge set of $K_I$. For a graph $G\subseteq K_n$ and a subset $I\subseteq[n]$, let $G_I$ denote
the induced subgraph on vertex set $I$, with isolated vertices retained, and write
\begin{equation*}
\E_I(G) \coloneqq \E(G)\cap\E_I.
\end{equation*}
Thus, $ \vert \E_I(G) \vert $ is the number of edges of $G_I$. We write $k(G)$ for the
number of connected components of $G$. Also, define the \emph{cut} determined by $I$ to be
\begin{equation*}
\delta(I) \coloneqq \{\{u,v\}\in \E_n\, \vert \, u\in I,\ v \notin I \}.   
\end{equation*}
By definition it is clear that 
\begin{equation*}
\delta(I)=\delta([n]\setminus I).
\end{equation*}

\begin{rem}
When we restrict a graph to a subset of vertices, isolated vertices are retained. This
convention is important when counting the connected components of $G\cap K_I$.
\end{rem}

We define an edge weight of a complete graph $K_n$ as a function
\begin{equation*}
\x  \colon \E_n\longrightarrow \mathbb{R} \quad \text{with the notation } x_e \coloneqq \x (e)\quad \text{for any }e \in \E_n.
\end{equation*}
Similar notation is used for other letters when they represent an edge weight function corresponding to an edge weighted complete graph. Let $i$ be a vertex of $K_n$ with edge weight function $\x   \colon \E_n\longrightarrow \mathbb{R}$, then the degree of $i$ with respect to this edge weighted complete graph is defined as the sum of the weights of all edges attached to $i$:
\begin{equation*}
\deg_{\x } (i) \coloneqq   \sum_{j\in [n]\setminus \{i\}}x_{\{i,j\}}=\sum_{e \in \delta(\{i\})}x_e.
\end{equation*}
For any subset $S\subseteq \E_n$ of edges of the complete graph $K_n$, we set
\begin{equation*}
\x (S) \coloneqq \sum_{e \in S}x_e.
\end{equation*}
\begin{rem}
We use both notations $\x (S)$ and $\sum_{e \in S}x_e$ throughout the paper. Also, note that we have
\begin{equation*}
\deg_{\x } (i) =\x \left(\delta(\{i\})\right).
\end{equation*}
\end{rem}

\begin{rem}
  We note that the index set $\II_{n+1}=\{\,I\subseteq[n] : 2\leq \vert I \vert \leq n-1\,\}$ introduced above for boundary divisors on $\overline{M}_{0,n+1}$ is exactly the set of vertex subsets of $K_n$ having at least two elements and not equal to $[n]$. It is therefore also the natural index set for the \emph{clique inequalities} $\x(\E_I)\leq c_{I}$ appearing in the linear programming relaxations below.  
\end{rem}

\subsubsection{Hamiltonian cycle and Held--Karp relaxation polytopes}

We next recall the Hamiltonian cycle polytope and its subtour elimination relaxation.

\begin{defn}
The \emph{symmetric traveling salesman polytope (Hamiltonian cycle polytope)} $P_{\mathrm{STSP},n}$ of the complete graph $K_n$ is the convex hull of the incidence vectors of Hamiltonian cycles of $K_n$.
\end{defn}

The complete set of inequalities defining $P_{\mathrm{STSP},n}$ is not known, in general \cite{Karp1972,Padberg1980OnTS}. The following simpler polytope was introduced by Dantzig--Fulkerson--Johnson \cite{DFJ}.

\begin{defn}[\cite{DFJ}]\label{def:HK}
The \emph{subtour (Held--Karp) relaxation polytope} of the symmetric traveling salesman polytope is defined by
\begin{equation*}
P_{\mathrm{SEP},n}=\left \{(x_{e})_{e \in \E_n} \in \R^{ \E_n}_{\geq 0} \, \bigg \vert \, x_e\leq 1\,\, \forall e \in \E_n,\, \deg_{\x }(i)=2\,\, \forall i\in [n],\,  \x(\delta(I))\geq 2\,\, \forall I \, \text{with } \varnothing \neq I \subsetneq [n] \right \}.
\end{equation*}
\end{defn}

The subtour elimination polytope is one of the central linear programming relaxations of the symmetric traveling salesman problem. Dantzig--Fulkerson--Johnson observed that the $0/1$-points of $P_{\mathrm{SEP},n}$ are the Hamiltonian cycles of $K_n$. However, the inclusion 
\begin{equation*} P_{\mathrm{SEP},n} \supseteq P_{\mathrm{STSP},n} \end{equation*}
is strict in general. The polytope $P_{\mathrm{SEP},n}$ is also known as the \emph{Held--Karp relaxation polytope} due to works \cite{HK70,HK71} of Held--Karp. See \cite[Chapter 58]{Schrijver03}, \cite[Chapter 2]{Traub2024}, and the references therein for more details on the traveling salesman problem and the subtour elimination polytope $P_{\mathrm{SEP},n}$.

In Lemma \ref{lem:Held-Karp rank description}, we show that $P_{\mathrm{SEP},n}$ can also be described as
\begin{equation}\label{eqn:P_SEP_alternative_form}
P_{\mathrm{SEP},n}= \left \{ (x_{e})_{e \in \E_n} \in \R^{ \E_n}_{\geq 0}\, \big \vert \, \x(\E_n) =n,\, \x(\E_I) \leq \vert I \vert -1 \,  \text{ for all }I\in \II_{n+1} \right \}.
\end{equation}
We do not claim originality for this description, however we were unable to locate a reference for this rewriting.

\section{ Effective Boundary Expression Polytopes}

In this section, we introduce the polytope of effective boundary expressions of a divisor on $\M$ and establish its basic properties. We then show that projection to the pair boundary coefficients gives an equivalent graph-theoretic description. After changing coordinates $x_e=1-a_e$, the defining inequalities become clique inequalities on a complete graph. These descriptions will be used throughout the applications in the remaining of the paper.

\subsection{The polytope of effective boundary expressions}

Let $\V_n$ denote the free real vector space on the set of boundary divisors of $\M$, with basis $\{D_I\}_{I\in\II_n}$, and let
\begin{equation*}
  \mathrm{Cone}_n \coloneqq \left\{\, \sum_{I\in\II_n} a_I D_I \ \bigg \vert \ a_I\geq 0 \,\right\} =\bigoplus_{I \in \II_n} \R_{\geq 0} \cdot D_I \subseteq \V_n
\end{equation*}
be the nonnegative orthant in this basis.  Let $\W_n\subseteq\V_n$ be the subspace of those linear combinations of boundary divisors that are linearly equivalent to $0$. Since the boundary classes span $\Cl(\M)_{\R}$ \cite{Keel}, we have a short exact sequence of real vector spaces
\begin{equation}\label{eq:projection onto classgroup}
    0 \to \W_n \to \V_n \to \Cl(\M)_{\R} \to 0,
\end{equation}
in which the right-hand map sends a formal linear combination of boundary divisors to its linear equivalence class. Comparing the sizes of the bases for $\V_n$ and $\Cl(\M)_{\R}$, one sees that
\begin{equation*}
\dim \V_n = 2^{n-1}-n-1 \quad \text{and}\quad \dim \W_n = {n-1 \choose 2}-1.
\end{equation*}
For any distinct $a,b,c,d \in [n]$, define vectors
\begin{equation}\label{eqn:T_vectors_abcd}
T_{ab \mid cd}
 \coloneqq \sum_{\substack{I\in\II_n^+\,, a,b\in I \\ c,d\notin I}}D_I \in \V_n.
\end{equation}
By Lemma \ref{lem:Keels_quadruple_relations}, $\W_n$ is spanned by vectors of the form 
\begin{equation}\label{eq:keel vector}
v_{i,j,k,\ell} \coloneqq T_{ij\mid k \ell}-T_{ik\mid j \ell}=
\sum_{\substack{I\in \II_n^+, i,j\in I\\ k,\ell\notin I}} D_I
-
\sum_{\substack{I\in \II_n^+, i,k\in I\\ j,\ell\notin I}} D_I,
\end{equation}
where $i,j,k,\ell$ are distinct elements of $[n]$. We call the vectors $v_{i,j,k,\ell}$ the \emph{Keel vectors}. 

For any $D\in \V_n$, the affine subspace
\begin{equation*}
D+\W_n = \{ D+E \, \vert \, E \in \W_n \}\subseteq \V_n 
\end{equation*}
consists of all linear combinations of boundary divisors $E \in \V_n$ linearly equivalent to $D$. Since any divisor $D$ on $\M$ is linearly equivalent to a sum of boundary divisors, the affine subspace $D+\W_n\subseteq \V_n$ is well defined for any divisor $D$ and depends only on the linear equivalence class $[D]$. Our main object of study is the following:

\begin{defn}
For any divisor $D$ on $\M$, the \emph{polytope of effective boundary expressions of $D$}, $Q(D)\subseteq \V_n$, is the set of real linear combinations of boundary divisors with nonnegative coefficients that are linearly equivalent to $D$, i.e.
\begin{equation*}
Q(D)\coloneqq   (D+\W_n)\cap \mathrm{Cone}_n = \left\{\, (a_I)_{I\in\II_n}\in\R^{\II_n} \, \bigg \vert \,
  [D] = \sum_{I\in\II_n} a_I\,[D_I],\ \ a_I\geq 0 \,\right\} \subseteq \V_n.
\end{equation*}
\end{defn}

\begin{lem}
    The subset $Q(D)\subseteq \V_n$ depends only on the linear equivalence class of $D$, and is a polytope. 
\end{lem}

\begin{proof}
    It is clear from the definition that $Q(D)$ depends only on the linear equivalence class of $D$. By definition, $Q(D)$ is the intersection of a translated linear subspace and a polyhedral cone, so it follows that $Q(D)$ is a polyhedron. It remains to show that $Q(D)$ is bounded. Observe that the sum of the coefficients on each Keel vector $v_{i,j,k,\ell}$ is $0$, since there are an equal number of terms in each of the two sums in \eqref{eq:keel vector}. Since the Keel vectors span $\W_n$, the sum of the coefficients of the boundary divisors $D_I$ on any element of $\W_n$ is $0$. If we let $H\subseteq \V_n$ denote the hyperplane of elements $\sum_{I\in \II_n}a_I D_I\in \V_n$ with $\sum_{I\in \II_n}a_I = 0$, this observation implies that $\W_n\subseteq H$. But then, for any $D\in \V_n$, we have
    \begin{equation*}Q(D) = (D+\W_n)\cap \mathrm{Cone}_n \subseteq (D+H) \cap \mathrm{Cone}_n. \end{equation*}
    The latter is easily seen to be bounded, and so $Q(D)$ is a polytope as claimed. 
\end{proof}

\begin{lem}\label{lem:subadditivity and scaling}
    For divisors $D,D'$ on $\M$ and $t>0$, we have
    \begin{equation*}
    Q(tD)= t\cdot Q(D) \quad \text{and} \quad Q(D)+Q(D')\subseteq Q(D+D').
    \end{equation*}
\end{lem}

\begin{proof}
    For the first claim, we observe that for a set of coefficients $\{a_I\}_{I\in \II_n}$, we have:
    \begin{equation*}
    \sum_{I\in \II_n} a_I D_I\in Q(D) \quad \text{if and only if}\quad \sum_{I\in \II_n} ta_I D_I\in Q(tD).
    \end{equation*}
    This establishes the equality $Q(tD)= t\cdot Q(D)$.
    
    For the second claim, given $\sum_{I\in \II_n}a_ID_I\in Q(D)$ and $\sum_{I\in \II_n}a_I 'D_I\in Q(D')$, we have
    \begin{equation*}
    \sum_{I\in \II_n}(a_I+a_I')D_I\in Q(D+D').
    \end{equation*}
    This gives the inclusion $Q(D)+Q(D')\subseteq Q(D+D')$.
\end{proof}

\subsection{Pullback by forgetful maps}

Let $\pi \colon\overline{M}_{0,n+1}\to \M$ be the map forgetting the point marked $n+1$. The pullback by $\pi$ gives an injective linear map $\Cl(\M)_{\R}\to\Cl(\overline{M}_{0,n+1})_{\R}$. We can lift this to a linear map
\begin{equation*}
 \pi^*  \colon \V_n\to \V_{n+1}
\end{equation*}
by defining 
\begin{equation}
    \pi^*(D_I) = D_I+D_{I\cup\{n+1\}} = D_I+D_{[n]\setminus I}
\end{equation}
for all $I\in \II_n^+$ so that we get a commuting diagram

\begin{equation}\label{eq:pullback cd}
    \begin{tikzcd}
    \V_n \arrow[r,"\pi^*"] \arrow[d] & \V_{n+1} \arrow[d]\\
    \Cl(\M)_{\R} \arrow[r,"\pi^*"] & \Cl(\overline{M}_{0,n+1})_{\R},
\end{tikzcd}
\end{equation} 
where the vertical maps send a divisor to its linear equivalence class.

\begin{lem}
    The pullback map $\pi^*  \colon \V_n\to\V_{n+1}$ is injective with left-inverse
    \begin{equation*}
    \sum_{I\in \II_{n+1}}a_I D_I \mapsto \sum_{I\in \II_{n}}a_I D_I,
    \end{equation*}
    and restricts to an inclusion $\pi^*  \colon \W_n\to\W_{n+1}$.
\end{lem}

\begin{proof}
    Let $\sum_{I\in \II_n}a_ID_I\in \V_n$. Then, we compute
    \begin{equation*}\pi^*\left( \sum_{I\in \II_n}a_ID_I \right) = \sum_{I\in \II_n}a_I(D_I+D_{[n]\setminus I}). \end{equation*}
    One sees that for any $I\in \II_n\subseteq\II_{n+1}$, the coefficient on $D_I$ in this expression is $a_I$, which establishes the first claim. The inclusion $\pi^*(\W_n)\subseteq \W_{n+1}$ follows from the commutativity of \eqref{eq:pullback cd} since then for $D\in \W_n$, we have 
    \begin{equation*}[\pi^*(D)] = \pi^*([D]) = \pi^*(0) = 0. \end{equation*}
\end{proof}

\begin{lem}\label{lem:permutations_of_keel_vectors}
Let $a,b,c,d\in[n]$ be distinct. Then, we have the equality of the following sets:
\begin{equation*}
\left\{
v_{p,q,r,s}:
(p,q,r,s)\text{ is a permutation of }(a,b,c,d)
\right\}
=
\left\{
v_{i,j,k,d}:
(i,j,k)\text{ is a permutation of }(a,b,c)
\right\}.
\end{equation*}
In particular, every Keel vector involving the index $d$ can be written
with $d$ in the fourth position.
\end{lem}

\begin{proof}
Set
\begin{equation*}
A \coloneqq T_{ab\mid cd},\qquad
B \coloneqq T_{ac\mid bd},\qquad
C \coloneqq T_{ad\mid bc}.
\end{equation*}
Every Keel vector whose indices are supported on $\{a,b,c,d\}$ (i.e. indices that are permutations of $a,b,c,d$) is a difference of two distinct elements of $\{A,B,C\}$ since $D_I=D_{I^c}$ and consequently
\begin{equation*}
T_{ij \mid k \ell}=T_{k \ell \mid ij}
\end{equation*}
for $i,j,k,\ell \in [n]$ distinct. Keeping $d$ in the fourth position gives
\begin{equation*}
\begin{aligned}
v_{a,b,c,d}=A-B, &\quad&
v_{a,c,b,d}=B-A,\\
v_{b,a,c,d}=A-C, &\quad&
v_{b,c,a,d}=C-A,\\
v_{c,a,b,d}=B-C, &\quad&
v_{c,b,a,d}=C-B.
\end{aligned}
\end{equation*}
These are all six differences between $A,B,C$, proving the
claim.
\end{proof}

\begin{lem}\label{lem:pullback keel vectors}
    The quotient $\W_{n+1}/\pi^*(\W_n)$ is spanned by the Keel vectors $v_{i,j,k,n+1}$ for $i,j,k\in [n]$.
\end{lem}

\begin{proof} 
    The subspace $\pi^*(\W_n)\subseteq\W_{n+1}$ is spanned by the images of the Keel vectors $v_{i,j,k,\ell}\in \W_n$ for $i,j,k,\ell\in [n]$, which we compute to be 
    \begin{align*}
        \pi^*(v_{i,j,k,\ell}) & = 
    \sum_{\substack{I\in \II_n^+, i,j\in I\\ k,\ell\notin I}} (D_I+D_{I\cup\{n+1\}})
    -
    \sum_{\substack{I\in \II_n^+, i,k\in I\\ j,\ell\notin I}} (D_I+D_{I\cup\{n+1\}}) \\
    & = 
    \sum_{\substack{I\in \II_{n+1}^+, i,j\in I\\ k,\ell\notin I}} D_I
    -
    \sum_{\substack{I\in \II_{n+1}^+, i,k\in I\\ j,\ell\notin I}} D_I.
    \end{align*} 
    This is exactly the Keel vector $v_{i,j,k,\ell}\in \W_{n+1}$ with the same indices $i,j,k,\ell\in [n]\subseteq [n+1]$. In other words, $\pi^*(\W_n)$ is spanned by the Keel vectors $v_{i,j,k,\ell}$ for which none of the indices is $n+1$. 

    Now suppose that one of the four indices is $n+1$.  By
Lemma \ref{lem:permutations_of_keel_vectors}, every Keel vector
supported on this quadruple can be written in the form $v_{a,b,c,n+1}$ for some distinct $a,b,c\in[n]$. Consequently,
\begin{equation*}
\W_{n+1}
=
\pi^*(\W_n)
+
\operatorname{span}
\left\{
v_{i,j,k,n+1}\, \vert \,
i,j,k\in[n]\text{ distinct}
\right\}.
\end{equation*}
Passing to the quotient by $\pi^*(\W_n)$ gives the result.
\end{proof}

\begin{prop}\label{prop:pullback on polytopes}
    For any $D\in \V_n$, we have 
    \begin{equation*}
    \pi^*(Q(D)) = Q(\pi^*(D))\subseteq \V_{n+1}. \end{equation*}
    In other words, the polytopes $Q(D)\subseteq \V_n$ and $Q(\pi^*(D))\subseteq \V_{n+1}$ are identified by the embedding
    \begin{equation*}
    \pi^*  \colon \V_n\hookrightarrow \V_{n+1}.
    \end{equation*}
\end{prop}

\begin{proof}
    For any $D\in \V_n$, we have $\pi^*(D+\W_n) = \pi^*(D)+\pi^*(\W_{n})$. Since $\pi^*$ is injective, we have
    \begin{equation*}
        \pi^*(Q(D)) = \pi^*((D+\W_n)\cap \mathrm{Cone}_n)
        = (\pi^*(D)+ \pi^*(\W_n)) \cap \pi^*(\mathrm{Cone}_n).
    \end{equation*}
    The inclusion $\pi^*(Q(D))\subseteq Q(\pi^*(D)) = (\pi^*(D)+\W_{n+1})\cap \mathrm{Cone}_{n+1}$ follows from this description plus the inclusions $\pi^*(\W_n)\subseteq \W_{n+1}$ and $\pi^*(\mathrm{Cone}_n)\subseteq \mathrm{Cone}_{n+1}$.

    For the other inclusion, we observe that $\pi^*(\V_n)\subseteq \V_{n+1}$ is contained in the coordinate subspace of vectors $\sum_{I\in \II_{n+1}}a_I D_I\in \V_{n+1}$ for which $a_{[n]\setminus\{i\}} = 0$ for every $i=1,\dots,n$. This is because such terms cannot appear in the expression $\pi^*(D_I) = D_I+D_{[n]\setminus I}\in \V_{n+1}$ for any $I\in \II_n$ since $ \vert I  \vert \geq 2$. In particular, $\pi^*(D)+\pi^*(\W_n)$ is contained in the coordinate subspace on which these coefficients are zero.

    Let $\xi  \colon \W_{n+1}\longrightarrow \R^n$ be the projection onto the special coordinates indexed by $D_{[n]\setminus\{i\}}$, $i\in[n]$. Since $\xi$ vanishes on $\pi^*(\W_n)$, it induces a map
    \begin{equation*}
    \overline{\xi}  \colon \W_{n+1}/\pi^*(\W_n)\longrightarrow \overline{H} \coloneqq \left\{(c_1,\ldots,c_n)\in\R^n\ \bigg \vert \ \sum_i c_i=0\right\}.
    \end{equation*}
    Moreover, we have
    \begin{equation*}
    \overline{\xi}\left (v_{i,j,k,n+1}\right)=1_k-1_j
    \end{equation*}
    where $1_j$ and $1_k$ are standard basis elements of $\R^n$.
    Hence, by Lemma \ref{lem:pullback keel vectors}, the induced map $\overline{\xi}$ is surjective. Since
    \begin{equation*}
    \dim \left(\W_{n+1}/\pi^*(\W_n) \right) = \left(\binom{n}{2}-1\right) - \left(\binom{n-1}{2}-1\right) =n-1 =\dim \overline{H},
    \end{equation*}
    it is an isomorphism. Hence, every nonzero class in the quotient has a nonzero coordinate vector in $\overline{H}$, and such a vector must have a negative coordinate as sum of the coordinates in $\overline{H}$ is zero.
    
    Now, let $q\in Q(\pi^*(D))$ and write $q=\pi^*(D)+w$ with $w\in\W_{n+1}$. The special coordinates of $q$ are those of $w$, and they are nonnegative because $q$ is effective. The preceding paragraph therefore implies that $w=0$ in $\W_{n+1}/\pi^*(\W_n)$, so $w\in\pi^*(\W_n)$. Thus $q=\pi^*(q_0)$ for some $q_0\in D+\W_n$. Applying the left inverse of $\pi^*$ shows that $q_0$ has nonnegative coefficients, hence $q_0\in Q(D)$. Therefore $q\in\pi^*(Q(D))$, proving the reverse inclusion.
\end{proof}

\subsection{A graphical parametrization of $Q(D)$}

Since $Q(D)$ is contained in the translate $D+\W_n$ of $\W_n\subseteq \V_n$, we have $\dim(Q(D))\leq \dim(\W_n) = {n-1\choose 2}-1$, much smaller than the dimension $2^{n-1}-n-1$ of the ambient space $\V_n$. For this reason, we will describe certain projections of $\V_n$ that preserve the structure of $Q(D)$. 

Let $\E_{n-1} = \{ \{i,j\} \,  \vert \, 1\leq i < j \leq n-1 \}$ denote the edge set of the complete graph $K_{n-1}$. We have $\E_{n-1}\subseteq \II_n$, so we think of $e \in\E_{n-1}$ as indexing the boundary divisor $D_e$ on $\M$. Define the following linear projection of $\V_n$ onto $\R^{\E_{n-1}}$:
\begin{equation*}
\rho_n  \colon \V_n\to \R^{\E_{n-1}} \quad \text{with}\quad \rho_n\left(\sum_{I\in \II_n} a_I D_I\right) \coloneqq   (a_{e})_{e \in \E_{n-1}}. 
\end{equation*} 
We identify $\R^{\E_{n-1}}$ with the coordinate subspace of $\V_n$ spanned by $D_{e}$ for $e \in \E_{n-1}$ so that we may write $\rho_n$ as the coordinate projection 
\begin{equation*}\rho_n\left(\sum_{I\in \II_n} a_I D_I\right)= \sum_{e \in \E_{n-1}}a_{e}D_{e}. \end{equation*}

\begin{lem}\label{lem:injectivity of projection}
    The restriction $\rho_n  \colon \W_n\to \R^{\E_{n-1}}$ is injective, and the image $\rho_n(\W_n)\subseteq \R^{\E_{n-1}}$ is the hyperplane defined by $\sum_{e \in \E_{n-1}}a_{e}=0$.
\end{lem}

\begin{proof}
    We compute the projection to $\R^{\E_{n-1}}$ of the Keel vectors $v_{i,j,k,\ell}\in \W_n$. If one of the indices $i,j,k,\ell$ is $n$, then without loss of generality we may assume that $\ell=n$ by Lemma \ref{lem:permutations_of_keel_vectors} and $i,j,k\in [n-1]$. In this case, we have
    \begin{equation*}\rho_n(v_{i,j,k,n}) = \rho_n\left(\sum_{\substack{I\in \II_n^+,i,j\in I\\ k,n\notin I}} D_I - \sum_{\substack{I\in \II_n^+,i,k\in I\\ j,n\notin I}} D_I \right) = D_{\{i,j\}}-D_{\{i,k\}}. \end{equation*}
    On the other hand if $i,j,k,\ell\in [n-1]\subseteq[n]$, we have
    \begin{equation*}\rho_n(v_{i,j,k,\ell}) = \rho_n\left(\sum_{\substack{i,j\in I\\ k,\ell\notin I}} D_I - \sum_{\substack{i,k\in I\\ j,\ell\notin I}} D_I \right) = D_{\{i,j\}}+D_{\{k,\ell\}}-D_{\{i,k\}}-D_{\{j,\ell\}}. \end{equation*}
    In either case, $\rho_n(v_{i,j,k,\ell})$ lies in the subspace defined by  $\sum_{e \in \E_{n-1}}a_{e} = 0$. Since these vectors span $\W_n$, we may conclude that $\rho_n(\W_n)$ is contained in this hyperplane as well. 
    
    The other containment follows from the observation that the hyperplane $\sum_{e \in \E_{n-1}}a_{e}=0$ is spanned by the vectors 
    \begin{equation*}
    \rho_n(v_{i,j,k,n}) = D_{\{i,j\}}-D_{\{i,k\}}
    \end{equation*}
    for $i,j,k\in [n-1]$ and
    \begin{equation*}\rho_n(v_{i,j,k,n}+v_{k,i,\ell,n}) = D_{\{i,j\}}-D_{\{k,\ell\}}\end{equation*}
    for $i,j,k,\ell\in [n-1]$. Finally, since $\dim(\W_n)={n-1\choose 2}-1 = \dim(\R^{\E_{n-1}})-1$, the restriction of $\rho$ to $\W_n$ is injective as claimed.
\end{proof}

\begin{defn}
    For any divisor $D$ on $\overline{M}_{0,n}$, we define the polytope
    \begin{equation*}
    P(D) \coloneqq  \rho_n(Q(D))
    \end{equation*} to be the projection of $Q(D)$ onto $\R^{\E_{n-1}}$.
\end{defn}

\begin{cor}\label{cor:projection_Q_to_P_is_affine_iso}
For every divisor $D$ on $\M$, the coordinate projection $\rho_n  \colon Q(D)\longrightarrow P(D)$
is an affine isomorphism.  In particular, if $Q(D)$ is nonempty, then
$\dim Q(D)=\dim P(D)$.
\end{cor}

\begin{proof}
Surjectivity holds by the definition $P(D)=\rho_n(Q(D))$.  Now suppose for
$q,q'\in Q(D)$, we have $\rho_n(q)=\rho_n(q')$.  Since $q$ and $q'$ represent the same divisor class, $q-q'\in\W_n$.  Moreover,
$\rho_n(q-q')=0$.  Lemma \ref{lem:injectivity of projection} gives
$q-q'=0$, so $q=q'$. Thus, the restriction is bijective, and since it is the restriction of a linear map, it is an affine isomorphism.
\end{proof}

The following proposition describes the inverse of this projection map, expressing the coefficients of the remaining boundary divisors in a boundary expression for $[D]$ in terms of the coefficients on $\{ D_e\}_{e \in \E_{n-1}}$. First, we introduce some notation. For a divisor $D$ on $\M$, let $\alpha = \alpha(D)$ and $\beta_I = \beta(D)$, indexed $I\in \II_n$ with $ \vert I  \vert \geq 3$, be the coefficients in the expression of $[D]$ in the Kapranov basis:
\begin{equation*}[D] = \alpha \psi_n + \sum_{I\in \II_n, \vert I  \vert \geq 3} \beta_I [D_I]. \end{equation*}
\begin{rem}
To simplify notation, we also define $\beta_e = \beta_e(D)= 0$ for all $e \in \E_{n-1}$ and all divisors $D$.
\end{rem}

\begin{prop}\label{prop:inverse to projection}
    The projection $\rho_n  \colon Q(D)\to P(D)$ is a bijection with inverse $\iota_n  \colon P(D)\to Q(D)$ given by
    \begin{equation*}\iota  \colon (a_e)_{e \in \E_{n-1}}\mapsto \sum_{I\in \II_n}(\beta_I + \mathbf{a}(\E_I))D_I. \end{equation*}
\end{prop}

Before the proof we need two lemmas.

\begin{lem}\label{lem:Djk_basis_decomp}
    For any $e \in \E_{n-1}$, the expression for $[D_e]$ in the Kapranov basis centered at $n^{\text{th}}$ marked point is
    \begin{equation}\label{eq:D_e in Kapranov}
        [D_e] = \psi_n - \sum_{\substack{I\in \II_n, \vert I \vert \geq 3\\ e\subseteq I}}[D_I]. 
    \end{equation}
\end{lem}

\begin{proof}
    For $e=\{i,j\}\in \E_{n-1}$, we may express $\psi_n$ as
    \begin{equation*}\psi_n = \sum_{\substack{I\in \II_n^+, i,j\notin I\\ n\in I}} [D_I] = \sum_{I\in \II_n, i,j\in I} [D_I] = \sum_{I\in \II_n, e\subseteq I} [D_I]
    \end{equation*}
    by Lemma \ref{lem:some_relations_for_kappa_psi_and_D_A}. The desired formula is given by isolating $[D_e]$.
\end{proof}

\begin{lem}\label{lem:image of translation of W}
    The subset $\rho_n(D+\W_n) \subseteq \R^{\E_{n-1}}$ is the hyperplane defined by $\mathbf{a}(\E_{n-1}) = \alpha(D)$.
\end{lem}

\begin{proof}
    By Lemma \ref{lem:injectivity of projection}, it suffices to check that $\rho_n(D')$ is contained in this subspace for some $D'\in D+\W_n$. For the edge $\{1,2\} \in \E_{n-1}$, we choose
    \begin{equation*}D' = \sum_{I\in \II_n, \{1,2\}\subseteq I}\alpha D_I + \sum_{I\in \II_n, \vert I  \vert \geq 3}\beta_I D_I \end{equation*}
    which is in $D+\W_n$ by Lemma \ref{lem:Djk_basis_decomp}. The only $e \in \E_{n-1}$ for which $D_e$ appears in this expression is $e= \{1,2\}$, so we have $\rho_n(D') = \alpha D_{\{1,2\}}$. This point is clearly contained in the hyperplane $\sum_{e \in \E_{n-1}}a_e = \alpha$, which completes the proof.
\end{proof}

\begin{proof}[Proof of Proposition \ref{prop:inverse to projection}]
    Let $(a_e)_{\E_{n-1}}\in P(D)$. Then, for any $e \in \E_{n-1}$, it is clear that the coefficient on $D_e$ in the linear combination $D' = \iota((a_e)_{e \in \E_{n-1}})\in \V_n$ is $a_e$. Since the restriction
    \begin{equation*}
    \rho_n\vert_{D+\W_n}  \colon D+\W_n\longrightarrow \rho_n(D+\W_n) \subseteq \R^{\E_{n-1}}
    \end{equation*}
    is a bijection,  it suffices to show that $D' \in D+\W_n$, i.e. that $D'$ is linearly equivalent to $D$. We do this by computing the coefficients of $[D']$ in the Kapranov basis. Using Lemma \ref{lem:Djk_basis_decomp} we can compute
    \begin{align*}
        [D'] & = \sum_{I\in \II_n}(\beta_I + \sum_{e\subseteq I}a_e)[D_I] \\
        & = \sum_{e \in \E_{n-1}}a_e [D_e]+ \sum_{I\in \II_n, \vert I  \vert \geq 3}(\beta_I + \sum_{e\subseteq I}a_e)[D_I]\\
        & = \sum_{e \in \E_{n-1}}a_e \left(\psi_n - \sum_{I\in \II_n, e\subsetneq I}[D_I]\right)+ \sum_{I\in \II_n, \vert I  \vert \geq 3}(\beta_I + \sum_{e\subseteq I}a_e)[D_I]\\
        & = \left( \sum_{e \in \E_{n-1}}a_e \right)\psi_n + \sum_{I\in \II_n, \vert I  \vert \geq 3}\beta_I [D_I],
    \end{align*} 
    where in the final equality we have switched the order of summation to cancel all terms of the form $a_e[D_I]$ with $e\subsetneq I$. Since $(a_e)_{e \in \E_{n-1}}\in P(D) \subseteq \rho_n(D+\W_n)$, we have by Lemma \ref{lem:image of translation of W} that $\alpha = \sum_{e \in \E_{n-1}}a_e$. This shows that the expression for $[D']$ in the Kapranov basis is the same as that for $[D]$, completing the proof.
\end{proof}

As an immediate corollary, we obtain the inequalities defining $P(D)\subseteq \R^{\E_{n-1}}$.

\begin{cor}\label{cor:P(D) inequalities}
    For any divisor $D$ on $\overline{M}_{0,n}$, we have
    \begin{equation*}P(D) = \left\{ (a_e)_{e \in \E_{n-1}}\in \R^{\E_{n-1}}_{\geq 0} \, \bigg \vert \, \mathbf{a}(\E_{n-1}) = \alpha(D), \beta_I(D) + \mathbf{a}(\E_{I})\geq 0 \text{ for all }I \in \II_n \right\}. \end{equation*}
\end{cor}

\begin{proof}
    By Lemma \ref{lem:image of translation of W} and Proposition \ref{prop:inverse to projection} $(a_e)_{e \in \E_{n-1}}\in P(D)$ if and only if $\mathbf{a}(\E_{n-1}) = \alpha(D)$ and
    \begin{equation*} \sum_{I\in \II_n}(\beta_I + \mathbf{a}(\E_I))D_I \in Q(D). \end{equation*}
    By the definition of $Q(D)$, this means that the coefficients on each $D_I$ are nonnegative, which are the claimed conditions.
\end{proof}

Let
\begin{equation}
\rho^+_n   \colon  \V_n\to \R^{\E_n}
\end{equation}
be the projection map recording the coefficients on the divisors $D_e$ for all $e \in \E_n$. More precisely, it records the coefficient of $D_e$ if $e\in \E_{n-1}$ and records the coefficients of $D_{[n]\setminus e}$ if $e\in \E_{n}\setminus \E_{n-1}$.
We can also consider the image of $Q(D)$ in $\R^{\E_n}$ under this projection map.
\begin{defn}
Define the polytope $P^+(D)$ to be the image of $Q(D)$ under the projection map $\rho^+_n   \colon  \V_n\to \R^{\E_n}$:
\begin{equation*}
P^+(D) \coloneqq \rho^+_n (Q(D)).
\end{equation*}
\end{defn}
This polytope $P^+(D)$ has a more symmetric definition than $P(D)$ since $P(D)$ depends on the choice of index $n$ to omit from the divisors $D_{\{i,j\}}$. However, Corollary \ref{cor:projection_Q_to_P_is_affine_iso} shows that one does not obtain any new polytopes in this way as the projection $\rho_n$ factors through the projection $\rho^+_n$. In other words, the polytopes $Q(D)$, $P^+(D)$ and $P(D)$ are all affinely isomorphic. The corollaries below make the relationship between these constructions more transparent and they are structurally useful in applications.

\begin{cor}\label{cor: P^+ and P relationship}
    Let $\pi  \colon \overline{M}_{0,n+1}\to \M$ be the map forgetting the marked point $n+1$. Then, for any divisor $D$ on $\M$, there is an equality
    \begin{equation*}P^+(D) = P(\pi^*(D))\subseteq \R^{\E_n}. \end{equation*}
\end{cor}

\begin{proof}
    By Proposition \ref{prop:pullback on polytopes}, $Q(\pi^*(D))\subseteq \V_{n+1}$ is the image of $Q(D)\subseteq \V_{n}$ under the map $$\pi^*    \colon  D_I\mapsto D_I+D_{I\cup\{n+1\}}.$$ This map preserves the coefficient on $D_e$ for all $e \in \E_n$, which implies the claim.
\end{proof}

\begin{cor}\label{cor:extension by 0}
    The map identifying $\R^{\E_n}$ with the coordinate subspace $\R^{\E_{n+1}}$ defined by $a_{\{i,n+1\}} = 0$ for all $i\in [n]$, gives a bijection from $P^+(D)$ to $P^+(\pi^*(D))$.
\end{cor}

\begin{proof}
    The divisors $D_{\{i,n+1\}}$ for $i=1,\dots,n$ do not appear in the expression $\pi^*(D_I) = D_I+D_{I\sqcup \{n+1\}}$ for any $I\in \II_n$. This, combined with Corollary \ref{cor: P^+ and P relationship}, gives the desired claim.
\end{proof}

\begin{cor}\label{cor:P plus inverse}
Let $D$ be a divisor on $\M$.  The inverse
\begin{equation*}
\iota_D^+   \colon  P(D) \to P^{+}(D)
\end{equation*}
of the coordinate
projection $P^{+}(D)\to P(D)$ is the restriction of the affine map
$\gamma_D  \colon \R^{\E_{n-1}}\to \R^{\E_n}$ given by
\begin{equation*}
  (a_e)_{e \in \E_{n-1}} \longmapsto (a'_e)_{e \in \E_n},
  \qquad
  a'_e = \begin{cases}
    a_e, & e \in \E_{n-1},\\
    \beta_{[n-1]\setminus \{i\}}(D) + \mathbf{a}\big(\E_{[n-1]\setminus \{i\}}\big),
      & e = \{i,n\}\in \E_n\setminus \E_{n-1}.
  \end{cases}
\end{equation*}
\end{cor}

\begin{proof}
By Proposition \ref{prop:inverse to projection}, the point of $Q(D)$ over
$\mathbf{a}\in P(D)$ is $\sum_{I\in\II_n}\big(\beta_I(D)+\mathbf{a}(\E_I)\big)D_I$.  For
$e \in\E_{n-1}$ the coefficient on $D_e$ is $a_e$.  For $e=\{i,n\}$ we have
$D_{\{i,n\}}=D_{[n-1]\setminus\{i\}}$ with $[n-1]\setminus\{i\}\in\II_n$, and the
coefficient on it is $\beta_{[n-1]\setminus\{i\}}(D)+\mathbf{a}(\E_{[n-1]\setminus\{i\}})$.
\end{proof}

\begin{rem}
It is clear that the affine map $\gamma_D  \colon \R^{\E_{n-1}}\to \R^{\E_n}$ depends on the divisor class $[D]$.
\end{rem}
We summarize the relationships between the various polytopes and their different coordinatizations. Let $D$ be a divisor on $\M$, with $n\geq5$, and let
$\pi \colon \overline{M}_{0,n+1}\to\M$ forget the marking $n+1$.  Then, all of the maps in
\begin{equation}\label{eqn:polytope diagram summary}
\begin{tikzcd}[column sep=large]
  Q(\pi^*D) \arrow[r,"\rho^{+}_{n+1}"] & P^{+}(\pi^*D) \arrow[r,"\operatorname{pr}_{n+1}"] & P(\pi^*D)\\
  Q(D) \arrow[r,"\rho^{+}_{n}"]\arrow[u,"\pi^*"] & P^{+}(D) \arrow[r,"\operatorname{pr}_{n}"]\arrow[u,"\text{extend by }0"] & P(D)\arrow[u,"\iota^{+}_D"]
\end{tikzcd}
\end{equation}
are
bijections.  Here $\rho^+_n$ and $\rho^+_{n+1}$ are the
pair-coordinate projections, while
\begin{equation*}
\operatorname{pr}_n    \colon  \R^{\E_n}\longrightarrow\R^{\E_{n-1}}  
\end{equation*}
forgets the coordinates indexed by the edges incident to vertex $n$ of the complete graph $K_n$.
Corollary \ref{cor: P^+ and P relationship} identifies the
bottom-middle and top-right polytopes:
\begin{equation*}
P^+(D)=P(\pi^*D)\subseteq\R^{\E_n}.
\end{equation*}
The middle vertical map is extension by zero on the coordinates $a_{\{i,n+1\}}$ by Corollary \ref{cor:extension by 0}, and $\iota_D^+$ is the affine inverse from Corollary \ref{cor:P plus inverse}, which depends on $D$ through the coefficients $\beta_I(D)$.

\subsection{A change of coordinates from graph complementation}

For this subsection, we shift indices to work with divisors $D$ on $\overline{M}_{0,n+1}$ so that the coordinates of the polytopes $P(D)$ are parametrized by the edges $\E_n$ of the complete graph $K_n$.

\begin{defn}
For any divisor $D$ on $\overline{M}_{0,n+1}$, we define $\widetilde{P}(D)$ to be the image of $P(D)$ under the involution  $\R^{\E_n}\to \R^{\E_n}$ defined by $(a_e)_{e \in \E_n}\mapsto (1-a_e)_{e \in \E_n}$. Similarly, we define $\widetilde P^+(D)$ to be the image of $P^+(D)$ under the same coordinate change on $\R^{\E_{n+1}}$.
\end{defn}

\begin{rem}
 To distinguish between $\widetilde{P}(D)$ and $P(D)$, we use the coordinates $a_e$ for $P(D)$ and $x_e = 1-a_e$ for $\widetilde{P}(D)$. 
\end{rem}

Our description of the inequalities describing $P(D)$ is easily transformed into the inequalities describing $\widetilde{P}(D)$. We set
\begin{equation}\label{eqn:rank_functions_interms_of_alpha_beta}
 r_I(D) \coloneqq   \begin{cases}
        { \vert I \vert \choose 2} + \beta_I(D) & I\in \II_{n+1},\\
        {n \choose 2}- \alpha(D) & I = [n]
\end{cases} 
\end{equation}
where $\alpha(D)$ and $\beta_I(D)$ are the coefficients of the Kapranov basis expansion of $D$ as before.
\begin{prop}\label{prop:P tilde inequalities}
    For any divisor $D$ on $\overline{M}_{0,n+1}$:
    \begin{enumerate}
        \item The polytope $\widetilde{P}(D)$ is given by
            \begin{equation*}\widetilde{P}(D) = \left\{ (x_e)_{e \in \E_n}\in \R^{\E_n} \big \vert \, \x (\E_n) = r_{[n]}(D),\, \x (\E_I) \leq r_I(D)  \text{ for all }I \in \II_{n+1} \right\}.  \end{equation*}
        \item For any point $\x \in \widetilde{P}(D)$, the corresponding effective boundary divisor expression of $D$ is
        \begin{equation*} D = \sum_{I\in \II_{n+1}} \left( r_I(D)-\x (\E_I) \right)D_I. \end{equation*}
        \item The map $\widetilde{P}(D)\to \widetilde{P}^+(D)$ is given by $(x_e)_{e \in \E_n}\mapsto (x_e')_{e \in \E_{n+1}}$ where
        \begin{equation*} x_e' = \begin{cases}
            x_e & \text{ if } e \in \E_n,\\
            1 + r_{[n]} - r_{[n]\setminus \{ i \}}-\x(\delta(\{i\})) & \text{ if } e = \{i,n+1\} \in \E_{n+1}\setminus \E_n.
        \end{cases} \end{equation*}
    \end{enumerate}
\end{prop}

\begin{proof}
    Corollary \ref{cor:P(D) inequalities} gives the inequality description for $P(D)$:
    \begin{equation*}
    P(D) = \left\{ (a_e)_{e \in \E_n}\in \R^{\E_n} \, \big \vert \, \mathbf{a}(\E_n) = \alpha(D),  \beta_I(D)+\mathbf{a}(\E_I) \geq 0  \text{ for all }I\in \II_{n+1} \right\}.
    \end{equation*}
    We transform these conditions under the change of coordinates $x_e = 1-a_e$. For the equality, we obtain
    \begin{equation*}
    \sum_{e \in \E_n} x_e = \sum_{e \in \E_n} (1-a_e) = {n \choose 2} - \sum_{e \in \E_n} a_e = {n \choose 2} - \alpha(D).
    \end{equation*}
    Similarly, for each $I\in \II_{n+1}$ we obtain the inequality
    \begin{equation*}
    \sum_{e\subseteq I} x_e = \sum_{e\subseteq I} (1-a_e) = { \vert I \vert \choose 2} - \sum_{e\subseteq I} a_e \leq { \vert I \vert \choose 2} + \beta_I(D).
    \end{equation*}
    This proves part (1).
    For part (2), let $\x\in\widetilde P(D)$ and set $a_e=1-x_e$.  By Proposition \ref{prop:inverse to projection}, the coefficient of $D_I$ in the corresponding boundary expression is
    \begin{equation*}
    \beta_I(D)+\mathbf a(\E_I) = \beta_I(D)+\binom{\vert I \vert}{2}-\x(\E_I) = r_I(D)-\x(\E_I),
    \end{equation*}
    which proves the claimed formula.
    For part (3), the coordinates indexed by $\E_n$ are unchanged.  For $e=\{i,n+1\}$, Corollary \ref{cor:P plus inverse} gives
    \begin{equation*}
    a'_e = \beta_{[n]\setminus\{i\}}(D) + \mathbf a(\E_{[n]\setminus\{i\}}) = r_{[n]\setminus\{i\}}(D) - \x(\E_{[n]\setminus\{i\}}).
    \end{equation*}
    Consequently,
    \begin{equation*}
    x'_e = 1-a'_e = 1-r_{[n]\setminus\{i\}}(D) + \x(\E_{[n]\setminus\{i\}}).
    \end{equation*}
    Since
    \begin{equation*}
    \E_n = \E_{[n]\setminus\{i\}}\sqcup\delta(\{i\})
    \end{equation*}
    and $\x(\E_n)=r_{[n]}(D)$, we have
    \begin{equation*}
    \x(\E_{[n]\setminus\{i\}}) = r_{[n]}(D)-\x(\delta(\{i\})).
    \end{equation*}
    Therefore
    \begin{equation*}
    x'_{\{i,n+1\}} = 1+r_{[n]}(D)-r_{[n]\setminus\{i\}}(D) - \x(\delta(\{i\})),
    \end{equation*}
    as required.
\end{proof}

\section{Application I: Psi-classes and Kapranov classes}

Recall that the Kapranov class $X_{S,i}$, indexed by a subset $S\subseteq [n]$ with $ \vert S \vert \geq 3$ and $i\in S$, is the pullback of the class $\psi_i$ on $\overline{M}_{0,S}$ by the forgetful map $\M\to\overline{M}_{0,S}$. As the Kapranov class is trivial for $\vert S \vert=3$, we will consider the cases only when $\vert S \vert \geq 4$ here.

We define $\V_S$ to be the set of real linear combinations of boundary divisors on $\overline{M}_{0,S}$. We also define
\begin{equation*}
\II_S^+ \coloneqq \{ I\subseteq S, 2\leq \vert I  \vert \leq  \vert S \vert -2 \}
\end{equation*}
and
\begin{equation*}
\E_S \coloneqq \{ \{i,j\} \,  \vert \, \{i,j\}\subseteq S \}.
\end{equation*}
For any divisor $D$ on $\overline{M}_{0,S}$, we then consider $Q(D)\subseteq \V_S$ and $P^+(D)\subseteq \R^{\E_S}$ just as in the previous section.

\begin{thm}\label{thm:Kapranov polytopes}
    For any Kapranov class $X_{S,i}$ on $\M$, the polytope $Q(X_{S,i})\subseteq \V_n$ is a simplex with vertices
    \begin{equation}\label{eqn:vertices_of_Q_Kapranov_S_i}
    \sum_{\substack{I\in\II_n, 2\leq \vert I\cap S \vert \leq \vert S \vert -2\\ I \sqcup I^c\text{ separates }i\text{ from }\{j,k\} }}D_I
    \end{equation}
    indexed by pairs $\{j,k\}\subseteq S\setminus\{i\}$. This simplex is isomorphic as a lattice polytope to the convex hull of the standard basis vectors in $\R^{ \vert S \vert -1\choose 2}$. 
\end{thm}

\begin{rem}
We should note that we can rewrite equation \eqref{eqn:vertices_of_Q_Kapranov_S_i} as
\begin{equation*}
\sum_{\substack{I\in\II_n, 2\leq \vert I\cap S \vert \leq \vert S \vert -2\\ I \sqcup I^c\text{ separates }i\text{ from }\{j,k\} }}D_I
=
\sum_{\substack{I\in \II_n^+, I\cap S\in \II_S^+\\ i\notin I, j,k\in I}} D_I
\end{equation*}
by $D_I=D_{I^c}$. Taking $S=[n]$, we recover the standard boundary expressions for $\psi_i$
\begin{equation*}
\psi_i
=
\sum_{\substack{I\in\II_n^+\\
i\notin I,\ j,k\in I}}[D_I]
\end{equation*}
given in Lemma \ref{lem:some_relations_for_kappa_psi_and_D_A}.
Theorem \ref{thm:Kapranov polytopes} shows that these are precisely the
integral effective boundary expressions of $\psi_i$.  There are $\binom{n-1}{2}$ such expressions, and they are the vertices of a
simplex of dimension
\begin{equation*}
\binom{n-1}{2}-1. 
\end{equation*}
Moreover, for every integer $m\geq1$, each integral effective boundary
expression of $m\psi_i$ decomposes, uniquely up to reordering, as a sum
of $m$ integral effective boundary expressions of $\psi_i$.
\end{rem}

\begin{proof}[Proof of Theorem \ref{thm:Kapranov polytopes}]
    First, we consider the case $S=[n]$ where the Kapranov class $X_{[n],i}$ is the class $\psi_i$ on $\M$. There is an $S_n$-action on $\M$ relabeling the points that permutes the psi-classes and the boundary divisors. We will therefore check the claim for $\psi_n$ and the remaining $\psi_i$ will follow by symmetry. 
    
    Since $\psi_n$ is an element of the Kapranov basis, we have $\alpha(\psi_n)=1$ and $\beta_I(\psi_n)=0$ for all $I\in \II_n$, so by Corollary \ref{cor:P(D) inequalities},
    \begin{equation*}P(\psi_n) = \left\{ (a_{e})\in \R^{\E_{n-1}} \, \bigg \vert \, \mathbf{a}(\E_{n-1}) = 1, \mathbf{a}(\E_I)\geq 0 \text{ for all }I \in \II_n \right\}. \end{equation*}
    The inequalities for $I\in \II_n$ with $ \vert I \vert = 2$ say that $a_e\geq 0$ for all $e \in \E_{n-1}$, so the remaining inequalities for $I\in \II_n$ with $ \vert I  \vert \geq 3$ are redundant. We therefore obtain the description
    \begin{equation*}P(\psi_n) = \left\{ (a_{e})\in \R^{\E_{n-1}}_{\geq 0} \, \bigg \vert \, \mathbf{a}(\E_{n-1}) = 1 \right\}. \end{equation*}
    This is exactly the standard simplex in $\R^{\E_{n-1}} = \R^{n-1 \choose 2}$. We can then compute the vertices of $Q(\psi_n)$ by applying the map
    \begin{equation*}(a_e)_{e \in \E_{n-1}}\mapsto \sum_{I\in \II_n}\left(0 + \sum_{e\subseteq I}a_e \right)D_I \end{equation*}
    in Proposition \ref{prop:inverse to projection} to the standard basis vectors in $\R^{\E_{n-1}}$. For the basis vector corresponding to the edge $e = \{j,k\}$, this sum reduces to 
    \begin{equation*}\sum_{I\in \II_n, j,k\in I}D_I = \sum_{\substack{I\in \II_n^+, n\notin I,\\ j,k\in I}}D_I \end{equation*}
    as desired. This is the desired formula in the case $S=[n]$ and $i=n$, and we obtain by symmetry the formula for all psi-classes on $\M$.

    Now we consider pullbacks of psi-classes for some $S\subseteq[n]$ and $i\in S$. Since we have established the formula for psi-classes, the vertices of $Q(\psi_i)\subseteq \V_S$ are indexed by pairs $\{j,k\}\subseteq S\setminus\{i\}$ and given by
    \begin{equation*}\sum_{\substack{I\in \II_S^+, i\notin I\\ j,k\in I}}D_I\in \V_S. \end{equation*}
    
    By repeated applications of Proposition \ref{prop:pullback on polytopes}, the polytopes $Q(\psi_i)\subseteq \V_S$ and $Q(X_{S,i})\subseteq \V_n$ are identified via the map $\V_S\to \V_n$ defined by 
    \begin{equation*}D_I\mapsto \sum_{J\in \II_n^+, J\cap S = I}D_J. \end{equation*}
    Applying this map to the vertices of $Q(\psi_i)\subseteq \V_S$ described above gives the desired formula for the vertices of $Q(X_{S,i})\subseteq \V_n$. 
\end{proof}

It will be useful in later sections to describe the various projections of $Q(X_{S,i})$ defined in the previous section. Indeed, extracting the coefficients on divisors of the form $D_e$ for $e \in \E_n$ from the vertices described in Theorem \ref{thm:Kapranov polytopes}, we obtain the following corollary.

\begin{cor}\label{cor:Kapranov P plus}
    For any Kapranov class $X_{S,i}$ on $\M$, the polytope $P^+(X_{S,i})\subseteq \R^{\E_n}$ is a simplex with vertices indexed by edges $f = \{ j,k \}$ in the complete graph $K_{S\setminus\{i\}}$, given in coordinates $(a_e)_{e \in \E_n}$ by
    \begin{equation*} a_e = \begin{cases}
        1 & e = \{j,k\} \text{ or } e = \{i,\ell\} \text{ for some } \ell \in S\setminus \{i,j,k\},\\
        0 & \text{ else}.
    \end{cases} \end{equation*}
\end{cor}

\begin{proof}
Fix the vertex of $Q(X_{S,i})$ indexed by $\{j,k\}\subseteq S\setminus\{i\}$. Let for an $e\in \E_n$, we have $a_e=1$. Then, by Theorem \ref{thm:Kapranov polytopes} the edge $e$ and its complement $e^c=[n]\setminus e$ separates $i$ from $\{j,k\}$. Hence, either $e=\{j,k\}$ or $e=\{i,\ell\}$ with $\ell \in S\setminus \{i,j,k\}$. As all of the vertices of $Q(X_{S,i})$ have $0$ and $1$ coordinates, we complete the proof.
\end{proof}

Since all the vertices of $Q(X_{S,i})$ have coordinates  $0$ or $1$, and $Q(X_{S,i})$ is affinely isomorphic to $P^+(X_{S,i})$ via a coordinate projection, we will often represent such a point $(a_e)_{e \in \E_n}$ as the subgraph of $K_n$ containing an edge $e$ when $a_e=1$. Note that the description given in Corollary \ref{cor:Kapranov P plus} implies that the vertices of $Q(X_{S,i})$ correspond to subgraphs of $K_n$ that are contained in $K_S$.

\begin{ex}
    Let $S = \{ 1,2,3,4,6,7 \}\subseteq [7]$ and take $i=3$. By Theorem \ref{thm:Kapranov polytopes}, $Q(X_{S,i})$ is a simplex of dimension ${ \vert S \vert -1\choose 2}-1 = 9$ with vertices indexed by the $10$ two-element subsets $\{j,k\} \subseteq S\setminus \{i\} = \{ 1,2,4,6,7 \}.$ The vertex of $Q(X_{S,i})$ corresponding to $\{j,k\} = \{2,4\}$, for example, is
    \begin{align*}
        &D_{\{1,3\}}+D_{\{2,4\}}+D_{\{3,6\}}
+D_{\{1,2,4\}}+D_{\{1,3,5\}}+D_{\{1,3,6\}}+D_{\{2,4,5\}}+D_{\{2,4,6\}}
\\
&+D_{\{3,5,6\}}+
D_{\{1,2,4,5\}}+D_{\{1,2,4,6\}}+D_{\{1,3,5,6\}}
+D_{\{2,4,5,6\}}+D_{\{1,2,4,5,6\}} \in Q(X_{\{ 1,2,3,4,6,7\},3})
    \end{align*} 
The projection of this vertex onto $P^+(X_{\{ 1,2,3,4,6,7\},3})\subseteq \R^{21}$, i.e. recording just the coefficients $D_{I}$ for $ \vert I  \vert =2$ or $\vert I \vert =7-2=5$, gives the four nonzero coordinates
    \begin{equation*} a_{13} = a_{36} = a_{37} = a_{24}=1 \end{equation*}
    (recall that we identify $D_{\{1,2,4,5,6\}}=D_{\{3,7\}}$).
    In Figure \ref{fig:Kapranov vertices}, we depict all the vertices of $P^+(X_{\{ 1,2,3,4,6,7\},3})$ as subgraphs of $K_7$. The example above corresponds to the top right-most subgraph in the figure. One can verify in this case the description of the vertices given in Corollary \ref{cor:Kapranov P plus}.
\end{ex}

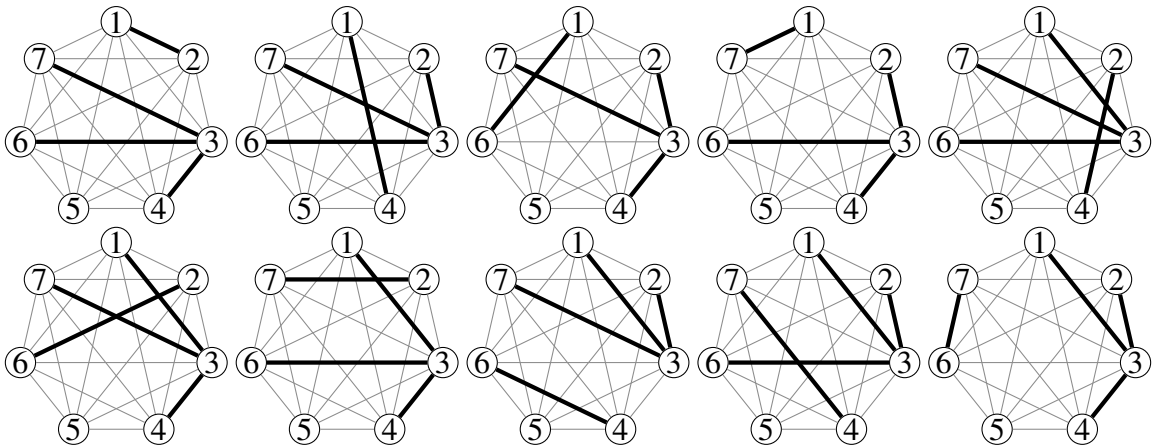
\begin{figure}[h]
\centering
\begin{tikzpicture}
  \graph[
    circular placement,
    radius=1.3cm,
    nodes={
      circle,
      draw,
      minimum size=4mm,
      inner sep=0pt
    },
    edges={black!45}
  ] {subgraph K_n [n=7, clockwise]};
  
   \draw (1) edge[ultra thick] (2);
   \draw (3) edge[ultra thick] (4);
   \draw (3) edge[ultra thick] (6);
   \draw (3) edge[ultra thick] (7);
   
\end{tikzpicture}
\begin{tikzpicture}
  \graph[
    circular placement,
    radius=1.3cm,
    nodes={
      circle,
      draw,
      minimum size=4mm,
      inner sep=0pt
    },
    edges={black!45}
  ] {subgraph K_n [n=7, clockwise]};
  
   \draw (1) edge[ultra thick] (4);
   \draw (3) edge[ultra thick] (2);
   \draw (3) edge[ultra thick] (6);
   \draw (3) edge[ultra thick] (7);
\end{tikzpicture}
\begin{tikzpicture}
  \graph[
    circular placement,
    radius=1.3cm,
    nodes={
      circle,
      draw,
      minimum size=4mm,
      inner sep=0pt
    },
    edges={black!45}
  ] {subgraph K_n [n=7, clockwise]};
  
   \draw (1) edge[ultra thick] (6);
   \draw (3) edge[ultra thick] (2);
   \draw (3) edge[ultra thick] (4);
   \draw (3) edge[ultra thick] (7);
\end{tikzpicture}
\begin{tikzpicture}
  \graph[
    circular placement,
    radius=1.3cm,
    nodes={
      circle,
      draw,
      minimum size=4mm,
      inner sep=0pt
    },
    edges={black!45}
  ] {subgraph K_n [n=7, clockwise]};
  
   \draw (1) edge[ultra thick] (7);
   \draw (3) edge[ultra thick] (2);
   \draw (3) edge[ultra thick] (6);
   \draw (3) edge[ultra thick] (4);
\end{tikzpicture}
\begin{tikzpicture}
  \graph[
    circular placement,
    radius=1.3cm,
    nodes={
      circle,
      draw,
      minimum size=4mm,
      inner sep=0pt
    },
    edges={black!45}
  ] {subgraph K_n [n=7, clockwise]};
  
   \draw (2) edge[ultra thick] (4);
   \draw (3) edge[ultra thick] (1);
   \draw (3) edge[ultra thick] (6);
   \draw (3) edge[ultra thick] (7);
\end{tikzpicture}

\begin{tikzpicture}
  \graph[
    circular placement,
    radius=1.3cm,
    nodes={
      circle,
      draw,
      minimum size=4mm,
      inner sep=0pt
    },
    edges={black!45}
  ] {subgraph K_n [n=7, clockwise]};
  
   \draw (2) edge[ultra thick] (6);
   \draw (3) edge[ultra thick] (1);
   \draw (3) edge[ultra thick] (4);
   \draw (3) edge[ultra thick] (7);
\end{tikzpicture}
\begin{tikzpicture}
  \graph[
    circular placement,
    radius=1.3cm,
    nodes={
      circle,
      draw,
      minimum size=4mm,
      inner sep=0pt
    },
    edges={black!45}
  ] {subgraph K_n [n=7, clockwise]};
  
   \draw (2) edge[ultra thick] (7);
   \draw (3) edge[ultra thick] (1);
   \draw (3) edge[ultra thick] (4);
   \draw (3) edge[ultra thick] (6);
\end{tikzpicture}
\begin{tikzpicture}
  \graph[
    circular placement,
    radius=1.3cm,
    nodes={
      circle,
      draw,
      minimum size=4mm,
      inner sep=0pt
    },
    edges={black!45}
  ] {subgraph K_n [n=7, clockwise]};
  
   \draw (4) edge[ultra thick] (6);
   \draw (3) edge[ultra thick] (1);
   \draw (3) edge[ultra thick] (2);
   \draw (3) edge[ultra thick] (7);
\end{tikzpicture}
\begin{tikzpicture}
  \graph[
    circular placement,
    radius=1.3cm,
    nodes={
      circle,
      draw,
      minimum size=4mm,
      inner sep=0pt
    },
    edges={black!45}
  ] {subgraph K_n [n=7, clockwise]};
  
   \draw (4) edge[ultra thick] (7);
   \draw (3) edge[ultra thick] (2);
   \draw (3) edge[ultra thick] (1);
   \draw (3) edge[ultra thick] (6);
\end{tikzpicture}
\begin{tikzpicture}
  \graph[
    circular placement,
    radius=1.3cm,
    nodes={
      circle,
      draw,
      minimum size=4mm,
      inner sep=0pt
    },
    edges={black!45}
  ] {subgraph K_n [n=7, clockwise]};
  
   \draw (6) edge[ultra thick] (7);
   \draw (3) edge[ultra thick] (1);
   \draw (3) edge[ultra thick] (2);
   \draw (3) edge[ultra thick] (4);
\end{tikzpicture}

\caption{The $10$ vertices of $P^+(X_{\{1,2,3,4,6,7\},3})$ for $n=7$.}
\label{fig:Kapranov vertices}
\end{figure}

\section{Application II: the log-canonical class $\kappa$ and spanning trees}

\subsection{Spanning tree and forest polytopes}

In this section we study the combinatorics of the polytopes associated the log-canonical class $\kappa$ plus a multiple of a psi-class. We first recall the definitions of the spanning tree and spanning forest polytopes.

\subsubsection{Spanning Tree and Forest Polytopes}

Recall that a forest is a graph without cycles and a tree is a connected forest. A spanning forest of $K_n$ means a subgraph of $K_n$ with vertex set $[n]$. Here we consider an isolated vertex to be a component, so a spanning forest of $K_n$ with $k$ connected components has exactly $n-k$ edges.

\begin{defn}\label{defn:Spanning_tree_poltope}
The \emph{spanning tree polytope} $P_{\mathrm{ST},n}\subseteq \R^{n \choose 2}$ of the complete graph $K_n$ is the convex hull of incidence vectors of spanning trees of $K_n$. More generally, the \emph{$k$-forest polytope} $P_{\mathrm{SF},n}^{(k)}\subseteq \R^{n \choose 2}$ of the complete graph $K_n$ is the convex hull of incidence vectors of spanning forests of $K_n$ with $k$ connected components.
\end{defn}

Figure \ref{fig:Spanning trees} shows several spanning trees of $K_6$, each corresponding to a vertex of $P_{\mathrm{ST},6}\subseteq \R^{15}$.

\begin{figure}[h]
\centering
\begin{tikzpicture}
  \graph[
    circular placement,
    radius=1.3cm,
    nodes={
      circle,
      draw,
      minimum size=4mm,
      inner sep=0pt
    },
    edges={black!45}
  ] {subgraph K_n [n=6, clockwise]};
  
   \draw (1) edge[ultra thick] (2);
    \draw (2) edge[ultra thick] (3);
   \draw (3) edge[ultra thick] (4);
   \draw (3) edge[ultra thick] (6);
   \draw (5) edge[ultra thick] (6);
\end{tikzpicture}
\begin{tikzpicture}
  \graph[
    circular placement,
    radius=1.3cm,
    nodes={
      circle,
      draw,
      minimum size=4mm,
      inner sep=0pt
    },
    edges={black!45}
  ] {subgraph K_n [n=6, clockwise]};
  
   \draw (1) edge[ultra thick] (4);
    \draw (1) edge[ultra thick] (2);
   \draw (3) edge[ultra thick] (2);
   \draw (3) edge[ultra thick] (6);
    \draw (5) edge[ultra thick] (6);
\end{tikzpicture}
\begin{tikzpicture}
  \graph[
    circular placement,
    radius=1.3cm,
    nodes={
      circle,
      draw,
      minimum size=4mm,
      inner sep=0pt
    },
    edges={black!45}
  ] {subgraph K_n [n=6, clockwise]};
  
   \draw (1) edge[ultra thick] (2);
   \draw (1) edge[ultra thick] (3);
   \draw (1) edge[ultra thick] (4);
    \draw (1) edge[ultra thick] (5);
   \draw (1) edge[ultra thick] (6);

\end{tikzpicture}
\begin{tikzpicture}
  \graph[
    circular placement,
    radius=1.3cm,
    nodes={
      circle,
      draw,
      minimum size=4mm,
      inner sep=0pt
    },
    edges={black!45}
  ] {subgraph K_n [n=6, clockwise]};

  \draw (3) edge[ultra thick] (1);
  \draw (3) edge[ultra thick] (2);
   \draw (3) edge[ultra thick] (4);
   \draw (3) edge[ultra thick] (5);
   \draw (3) edge[ultra thick] (6);
\end{tikzpicture}

\caption{Some spanning trees of $K_6$.}
\label{fig:Spanning trees}
\end{figure}
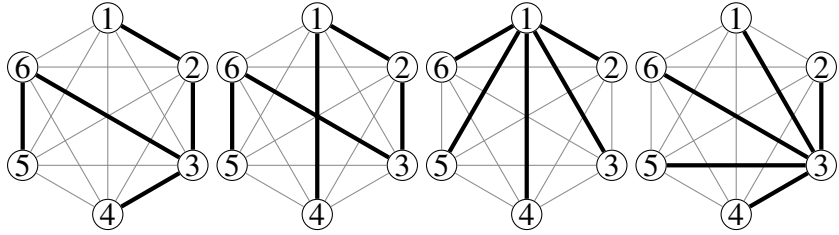

The following inequality description of the spanning forest polytopes is given by Edmonds \cite{Ed71} in matroid terms. See also \cite{Schrijver03}.
\begin{prop}[\cite{Ed71}]\label{prop:forest_and_spanning_tree_formulations_of_Edmonds}
We have
\begin{equation*}
P_{\mathrm{SF},n}^{(k)}= \left \{ (x_{e})_{e \in \E_n} \in \R^{ \E_n}_{\geq 0}\, \big \vert \, \x(\E_n) =n-k,\, \x(\E_I) \leq \vert I \vert -1 \, \text{ for all } I\in \II_{n+1} \right \},
\end{equation*}
and in particular we have
\begin{equation*}
P_{\mathrm{ST},n}
=
\left \{ (x_{e})_{e \in \E_n} \in \R^{ \E_n}_{\geq 0}\, \big \vert \, \x(\E_n) =n-1,\, \x(\E_I) \leq \vert I \vert -1 \, \text{ for all } I\in \II_{n+1} \right \}.
\end{equation*}
\end{prop}

\subsection{The log canonical class $\kappa$}

Now we return to the log canonical class $\kappa$ on $\overline{M}_{0,n+1}$. 

\begin{thm}\label{thm:kappa + t psi polytope}
    For any $t\in \R$, the divisor class $\kappa+t\psi_{n+1}$ on $\overline{M}_{0,n+1}$ has the following:
    \begin{enumerate}
        \item The polytope $\widetilde{P}(\kappa+t\psi_{n+1})$ is given by
        \begin{equation*}\widetilde{P}(\kappa+t\psi_{n+1}) = \left\{ (x_e)_{e \in \E_n}\in \R^{\E_n} \big \vert \x(\E_n) = n-1-t, \x(\E_I) \leq  \vert I  \vert -1 \text{ for all }I \in \II_{n+1} \right\}. \end{equation*}
        \item For any point $\x\in \widetilde{P}(\kappa^{(n+1)}+t\psi_{n+1})$, the corresponding boundary divisor expression is
        \begin{equation*} \kappa^{(n+1)}+t\psi_{n+1} = \sum_{I\in\II_{n+1}} \left (\vert I \vert-1-\x(\E_I)\right)[D_I].\end{equation*}
        \item The map $\widetilde{P}(\kappa^{(n+1)}+t\psi_{n+1}) \to \widetilde{P}^+(\kappa^{(n+1)}+t\psi_{n+1})$ is given by $(x_e)_{e \in \E_n}\mapsto (x_e')_{e \in \E_{n+1}}$ where
        \begin{equation*} x_e' = \begin{cases}
            x_e & \text{ if } e \in \E_n,\\
            2-t - \x(\delta(\{i\})) & \text{ if } e = \{i,n+1\} \in \E_{n+1}\setminus \E_n.
        \end{cases} \end{equation*}
    \end{enumerate}
\end{thm}

\begin{proof} This result is an application of Proposition \ref{prop:P tilde inequalities}.

By Lemma \ref{lem:kappa in kapranov basis}, the Kapranov basis expression of
$\kappa^{(n+1)}+t\psi_{n+1}$ centered at the marking $n+1$ is
\begin{equation*}
\kappa^{(n+1)}+t\psi_{n+1}
=
\left(\binom{n-1}{2}+t\right)\psi_{n+1}
-
\sum_{\substack{I\in\II_{n+1}\\  \vert I \vert \geq3}}
\binom{ \vert I \vert -1}{2}[D_I].
\end{equation*}
Consequently,
\begin{equation*}
r_{[n]}(\kappa^{(n+1)}+t\psi_{n+1})
=
\binom{n}{2}-\binom{n-1}{2}-t
=n-1-t,
\end{equation*}
and, for every $I\in\II_{n+1}$,
\begin{equation*}
r_I(\kappa^{(n+1)}+t\psi_{n+1})
=
\binom{ \vert I \vert }{2}-\binom{ \vert I \vert -1}{2}
= \vert I \vert -1.
\end{equation*}
Here, the same formula also holds for $ \vert I \vert =2$, because by convention
$\beta_I(\kappa^{(n+1)}+t\psi_{n+1})=0$ in that case. Parts (1) and (2) now follow from parts (1) and (2) of
Proposition \ref{prop:P tilde inequalities}.

For part (3), Proposition \ref{prop:P tilde inequalities} part (3) gives
\begin{align*}
x'_e
&=
1+r_{[n]}(\kappa^{(n+1)}+t\psi_{n+1})-r_{[n]\setminus\{i\}}(\kappa^{(n+1)}+t\psi_{n+1})
-\x \left(\delta(\{i\}) \right)\\
&=
1+(n-1-t)-(n-2)-\x \left(\delta(\{i\}) \right)\\
&=
2-t-\x \left(\delta(\{i\}) \right),
\end{align*}
for
$e=\{i,n+1\}$, which proves the claim.
\end{proof}

Using the description given in part (1) of Theorem \ref{thm:kappa + t psi polytope}, we observe close relationships between these polytopes and several well-known graph theory polytopes, which motivated the change of variables $x_e=1-a_e$.

\begin{cor}\label{cor:spanning tree polytope and kappa}
    For $\kappa^{(n+1)}$ the log-canonical class on $\overline{M}_{0,n+1}$, we have
    \begin{enumerate}
        \item For any integer $t\geq 0$, $\widetilde{P}(\kappa+t\psi_{n+1})\cap \R^{\E_n}_{\geq 0}$ is the $(t+1)$-connected component forest polytope of $K_n$:
        \begin{equation*}
        \widetilde{P}\left(\kappa^{(n+1)}+t\psi_{n+1}\right)
        \cap\R_{\geq0}^{\E_n}
        =
        P_{\mathrm{SF},n}^{(t+1)}.
        \end{equation*}
        In particular, for $t=0$ the polytope $\widetilde{P}(\kappa^{(n+1)})\cap \R^{\E_n}_{\geq 0}$ is the spanning tree polytope of $K_n$:
        \begin{equation*}
        \widetilde{P} \left(\kappa^{(n+1)}\right) \cap\R_{\geq0}^{\E_n}=P_{\mathrm{ST},n},
        \end{equation*}
        while for $t\geq n$ the polytope $\widetilde{P}(\kappa^{(n+1)}+t\psi_{n+1})\cap \R^{\E_n}_{\geq 0}$ is empty.
        \item The polytope $\widetilde{P}(\kappa^{(n+1)}-\psi_{n+1})\cap\R^{\E_n}_{\geq 0} =  \widetilde{P}^+(\kappa^{(n)})\cap \R^{\E_n}_{\geq 0}$ is the subtour elimination polytope of $K_n$:
        \begin{equation*}
        \widetilde{P}\left(\kappa^{(n+1)}-\psi_{n+1}\right)
        \cap\R_{\geq0}^{\E_n}
        =
        \widetilde{P}^{+} \left(\kappa^{(n)}\right)\cap\R_{\geq0}^{\E_n}
        =
        P_{\mathrm{SEP},n}.
        \end{equation*}
    \end{enumerate}
\end{cor}

\begin{proof}
    These graph polytope interpretations of $\widetilde{P}(\kappa+t\psi_{n+1})\geq \R^{\E_n}_{\geq 0}$ follow immediately from comparing the description in Theorem \ref{thm:kappa + t psi polytope} to the descriptions of the spanning forest and subtour elimination polytopes given in \ref{prop:forest_and_spanning_tree_formulations_of_Edmonds} and equation \eqref{eqn:P_SEP_alternative_form}. 

    To identify 
    \begin{equation*} \widetilde{P}(\kappa^{(n+1)} - \psi^{(n+1)}_{n+1}) = \widetilde{P}^+(\kappa^{(n)}), \end{equation*}
    recall that we have the identity
    \begin{equation*} \kappa^{(n+1)} - \psi^{(n+1)}_{n+1} = \pi^*(\kappa^{(n)}), \end{equation*}
    given by Lemma \ref{lem:some_relations_for_kappa_psi_and_D_A}. Hence, it follows from Corollary \ref{cor: P^+ and P relationship} that we have the identity
    \begin{equation*} P(\kappa^{(n+1)} - \psi^{(n+1)}_{n+1}) = P(\pi^*(\kappa^{(n)})) = P^+(\kappa^{(n)}). \end{equation*}
    The equality is preserved under the change of coordinates $x_e = 1-a_e$, which completes the proof.
\end{proof}

\begin{rem}\label{rem:matroid_vs_ours}
For $0\leq t\leq n-1$, the nonnegative truncation in Corollary \ref{cor:spanning tree polytope and kappa} is the base polytope of a matroid obtained by repeated truncations of the graphic matroid of $K_n$. Without the inequalities $x_e\geq0$, however, the clique inequalities $\x(\E_I)\leq \vert I \vert -1$ do not by themselves give the usual full edge-subset rank description of a matroid base polytope. The untruncated polytope $\widetilde{P}(\kappa^{(n+1)}+t\psi_{n+1})$ is therefore distinct from the similarly defined extended polymatroid. 
\end{rem}

\begin{cor}\label{cor:kappa-t-psi-boundary-effective-range}
The divisor class $\kappa^{(n+1)}+t\psi_{n+1}$ is linearly equivalent to an
effective real linear combination of boundary divisors if and only if $t\geq -1$.
Equivalently,
\begin{equation*}
\widetilde{P}\left(\kappa^{(n+1)}+t\psi_{n+1}\right)\neq\varnothing
\quad\text{if and only if}\quad
t\geq-1.
\end{equation*}
\end{cor}

\begin{proof}
Suppose first that $\x\in\widetilde{P}(\kappa^{(n+1)}+t\psi_{n+1})$.
Summing the inequalities
\begin{equation*}
\x(\E_I)\leq n-2
\end{equation*}
over the $n$ subsets $I\subseteq[n]$ of cardinality $n-1$, and observing that
each edge belongs to exactly $n-2$ of these subsets, gives
\begin{equation*}
(n-2)\x(\E_n)\leq n(n-2).
\end{equation*}
Thus $\x(\E_n)\leq n$. Since part (1) of Theorem
\ref{thm:kappa + t psi polytope} also gives
$\x(\E_n)=n-1-t$, it follows that $t\geq-1$.

Conversely, assume $t\geq-1$ and set
\begin{equation*}
x_e=\frac{2(n-1-t)}{n(n-1)} \quad \text{for all }e\in\E_n.
\end{equation*}
Then $\x(\E_n)=n-1-t$. If $x_e\leq 0$, all the inequalities $x(\E_I )\leq \vert I \vert -1$ are
immediate for all $I \in \II_{n+1}$. If $x_e>0$, then $t\geq-1$ implies
\begin{equation*}
x_e\leq \frac{2}{n-1}.
\end{equation*}
Hence, for $I\in\II_{n+1}$ with $ \vert I \vert \leq n-1$, we get
\begin{equation*}
\x(\E_I)
=
\binom{\vert I \vert }{2}x_e
\leq
\frac{\vert I \vert (\vert I \vert-1)}{n-1}
\leq
\vert I \vert-1.
\end{equation*}
Therefore, by Theorem \ref{thm:kappa + t psi polytope}, $\x$ belongs to
$\widetilde{P}(\kappa^{(n+1)}+t\psi_{n+1})$, proving nonemptiness.
\end{proof}

\subsection{Combinatorial boundary expressions for the log-canonical class}

In this section, we give explicit boundary divisor formulas for the log-canonical class $\kappa$. Recall that $\E_I(G)$ denotes the number of edges of $G\cap K_I$ and $k(G)$ denotes the number of connected components of $G$.

\begin{thm}\label{thm:tree expression for kappa}
    Let $T$ be a spanning tree of $K_n$. Then, 
    \begin{align*}
        \kappa^{(n+1)} & = \sum_{I\in \II_{n+1}} \left( \vert I  \vert -1- \vert \E_I(T) \vert \right)[D_I]\\
        & = \sum_{I\in \II_{n+1}} \left( k(T\cap K_I) -1\right)[D_I].
    \end{align*} 
    More generally, for any spanning forest  $F$ of $K_n$ with $k(F)$ connected components, we have
    \begin{align*}
        \kappa^{(n+1)} + (k(F)-1)\psi_{n+1} & = \sum_{I\in \II_{n+1}} \left( \vert I  \vert -1- \vert \E_I(F) \vert \right)[D_I]\\
        & = \sum_{I\in \II_{n+1}} \left( k(F\cap K_I) -1\right)[D_I].
    \end{align*} 
\end{thm}

\begin{proof}
    The first formula follows from the general boundary expression from a point $\x \in \widetilde{P}(D)$ in Proposition \ref{prop:P tilde inequalities} to the incidence vector of $T$, which lies in $\widetilde{P}(\kappa^{(n+1)})$ by Corollary \ref{cor:spanning tree polytope and kappa}. The equality between the coefficients in the two expressions follows from the fact that $T\cap K_I$ is a forest on vertex set $I$, and so $ \vert I  \vert - \vert \E(T\cap K_I) \vert $ is equal to the number of connected components of $T\cap K_I$.
\end{proof}

\begin{ex}
    Consider the first spanning tree in Figure \ref{fig:Spanning trees}. To find the coefficient on $D_{1345}$, for example, in the corresponding boundary divisor expression for $\kappa^{(7)}$, one considers the restriction of $T$ to the subgraph $K_{1345}$. The forest $T\cap K_{1345}$ has $3$ connected components (the vertex $1$, the vertex $5$, and the edge between vertices $3$ and $4$) so the coefficient on $[D_{1345}]$ is $2$.
\end{ex}

Theorem \ref{thm:tree expression for kappa} recovers the standard boundary divisor expressions for $\kappa$ given in \eqref{eqn:kappa_1_wrt_divisors} for a certain type of spanning tree, as the following example shows.

\begin{ex}
    Fix $i\in [n]$ and consider the spanning star graph $T$ with edges $\{i,k\}$ for all $k\in [n]\setminus \{i\}$. We claim that the corresponding expression for $\kappa^{(n+1)}$ in Theorem \ref{thm:tree expression for kappa} is
    \begin{equation*} \kappa^{(n+1)} = \sum_{I\not\ni i,n+1}\left( \vert I  \vert -1\right) [D_I]. \end{equation*}
    Indeed, for $I\ni i$, $ \vert T_I \vert = \vert I  \vert -1$ as each vertex other than $i$ is adjacent to a unique edge, and so the coefficient of $[D_I]$ is $0$. On the other hand, for $I\not\ni i$, $ \vert T_I \vert = 0$ as there are no edges between vertices in $I$ and so the coefficient on $[D_I]$ is $ \vert I  \vert -1$ as desired.
\end{ex}

\begin{ex}
    There is a unique spanning forest $F$ of $K_n$ with $n$ connected components, namely the graph on vertex set $[n]$ with no edges. This spanning forest gives the expression
    \begin{equation*} \kappa^{(n+1)}+(n-1)\psi_{n+1} = \sum_{I\in \II_{n+1}} \left( \vert I \vert -1\right)[D_I], \end{equation*}
    since $k(F\cap K_I) = \vert I  \vert .$
\end{ex}

Similarly, the integral points of the subtour elimination polytope correspond to Hamiltonian cycles so we obtain a similar formula

\begin{thm}\label{thm:Hamiltonian cycle formula for kappa}
    Let $C$ be a Hamiltonian cycle on $K_n$. Then, 
    \begin{equation*} \kappa^{(n+1)}-\psi_{n+1} = \sum_{I\in \II_{n+1}} (k(C\cap K_I)-1)[D_I], \end{equation*}
    and 
    \begin{equation*} \kappa^{(n)} = \sum_{I\in \II_n} (k(C\cap K_I)-1)[D_I]. \end{equation*}
\end{thm}

\begin{proof}
    The first formula follows from the same argument as in Theorem \ref{thm:tree expression for kappa}. We deduce the second formula from the first and from Lemma \ref{lem:injectivity of projection}. Indeed, we have 
    \begin{equation*} \kappa^{(n+1)}-\psi_{n+1} = \pi^*(\kappa^{(n)}), \end{equation*}
    but Lemma \ref{lem:injectivity of projection} implies that for any divisor $D$ on $\M$ and boundary expression,
    \begin{equation*} \pi^*(D) = \sum_{I\in \II_{n+1}}a_I [D_I], \end{equation*}
    we have 
    \begin{equation*} D = \sum_{I\in \II_{n}}a_I [D_I], \end{equation*}
    where we only include the sum over $\II_n\subseteq \II_{n+1}$ of the corresponding boundary divisors on $\M$. This completes the proof.
\end{proof}

\begin{ex}
    Consider the Hamiltonian cycle $C$ depicted in Figure \ref{fig:Hamiltonian cycle in K5}.
    \begin{figure}[h]
    \centering
    \begin{tikzpicture}
      \graph[
        circular placement,
        radius=1.3cm,
        nodes={
          circle,
          draw,
          minimum size=4mm,
          inner sep=0pt
        },
        edges={black!45}
      ] {subgraph K_n [n=5, clockwise]};
      
        \draw (1) edge[ultra thick] (3);
        \draw (3) edge[ultra thick] (4);
        \draw (4) edge[ultra thick] (2);
        \draw (2) edge[ultra thick] (5);
        \draw (5) edge[ultra thick] (1);
    \end{tikzpicture}
    \caption{A Hamiltonian cycle $C$ in $K_5$}
    \label{fig:Hamiltonian cycle in K5}
    \end{figure}
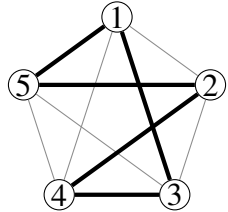
    The corresponding boundary expression for the log-canonical class of $\overline{M}_{0,5}$ is
    \begin{equation*} \kappa^{(5)} = [D_{12}]+[D_{14}]+[D_{23}]+[D_{123}]+[D_{124}]. \end{equation*}
    For example, the coefficient on $[D_{124}]$ is one because $C\cap K_{124}$ has two connected components, whereas the coefficient on $[D_{134}]$ is zero since $C\cap K_{134}$ is connected.
\end{ex}

We note that, although the sum in Theorem \ref{thm:Hamiltonian cycle formula for kappa} is indexed using only the subsets $I\subseteq [n]$ in $\II_n$, i.e. not containing $n$, the coefficient on $[D_I] = [D_{[n]\setminus I}]$ can be computed using either the restriction of $C$ to $K_I$ or $K_{[n]\setminus I}$, since 
\begin{equation*} k(C\cap K_I) = k(C\cap K_{[n]\setminus I}). \end{equation*}

\subsection{Descriptions of $\widetilde{P}(\kappa^{(n+1)})$ for small $n$}

\subsubsection{$\overline{M}_{0,4}$}

The polytope $\widetilde{P}(\kappa^{(4)})\subseteq \R^{\E_3}\simeq \R^3$ associated to the log-canonical class on $\overline{M}_{0,4}$ is a $2$-dimensional simplex. It is the convex hull of the incidence vectors of the three spanning trees of $K_3$ shown in Figure \ref{fig:Spanning trees of K_3}.

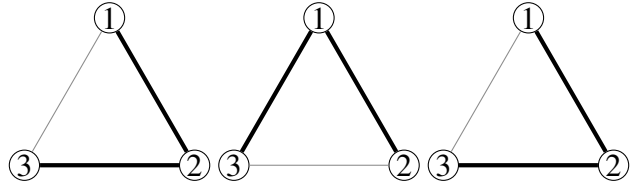
\begin{figure}[h]
    \centering
\begin{tikzpicture}
  \graph[
    circular placement,
    radius=1.3cm,
    nodes={
      circle,
      draw,
      minimum size=4mm,
      inner sep=0pt
    },
    edges={black!45}
  ] {subgraph K_n [n=3, clockwise]};
  
   \draw (1) edge[ultra thick] (2);
    \draw (2) edge[ultra thick] (3);
\end{tikzpicture}
\begin{tikzpicture}
  \graph[
    circular placement,
    radius=1.3cm,
    nodes={
      circle,
      draw,
      minimum size=4mm,
      inner sep=0pt
    },
    edges={black!45}
  ] {subgraph K_n [n=3, clockwise]};
  
   \draw (1) edge[ultra thick] (3);
    \draw (1) edge[ultra thick] (2);
\end{tikzpicture}
\begin{tikzpicture}
  \graph[
    circular placement,
    radius=1.3cm,
    nodes={
      circle,
      draw,
      minimum size=4mm,
      inner sep=0pt
    },
    edges={black!45}
  ] {subgraph K_n [n=3, clockwise]};
  
   \draw (1) edge[ultra thick] (2);
   \draw (2) edge[ultra thick] (3);
\end{tikzpicture}

\caption{The spanning trees of $K_3$, corresponding to vertices of $\widetilde{P}(\kappa^{(4)})$.}
\label{fig:Spanning trees of K_3}
\end{figure}

In this case, $\kappa^{(4)} = \psi_4$, so the vertices of $\widetilde{P}(\kappa^{(4)})$ can be obtained from the vertices of $P(\psi_4)$ the change of variables $x_e = 1-a_e$, corresponding to graph complementation.

\subsubsection{$\overline{M}_{0,5}$}

The polytope $\widetilde{P}(\kappa^{(5)})\subseteq \R^{\E_4}\simeq \R^6$ associated to the log-canonical class on $\overline{M}_{0,5}$ is a $5$-dimensional integral polytope with $22$ vertices. There are $16$ spanning trees of $K_4$, each of whose incidence vector is a vertex of $\widetilde{P}(\kappa^{(5)})$. The remaining $6$ vertices are indexed by the edges of $K_4$, and given up to symmetry by
\begin{equation*} x_e = \begin{cases}
    -1 & \text{ if } e = \{1,2\},\\
    0 & \text{ if } e = \{ 3,4\},\\
    1 & \text { else.}
\end{cases}.\end{equation*}

The corresponding weighted graph is shown in Figure \ref{fig:Nonpositive vertex}. One can verify that the conditions $\x([4]) = 3$ and $\x(I)\leq \vert I  \vert -1$ for all $I\in \II_5$ hold for such vectors.

\begin{figure}[h]
    \centering
\begin{tikzpicture}
  \graph[
    circular placement,
    radius=1.3cm,
    nodes={
      circle,
      draw,
      minimum size=4mm,
      inner sep=0pt
    },
    edges={black!45}
  ] {subgraph K_n [n=4, clockwise]};
  
    \draw (1) edge[red, dashed, ultra thick] node [midway, above right] {$-1$} (2);
    \draw (2) edge[ultra thick] (3);
    \draw (2) edge[ultra thick] (4);
    \draw (1) edge[ultra thick] (4);
    \draw (1) edge[ultra thick] (3);
\end{tikzpicture}
\caption{The unique nonpositive vertex of $\widetilde{P}(\kappa^{(5)})$ up to symmetry.}
\label{fig:Nonpositive vertex}
\end{figure}
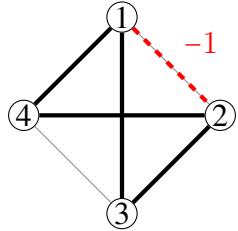

This vertex corresponds to the divisor expression
\begin{equation*} \kappa^{(5)} = [D_{34}]+[D_{35}]+[D_{45}]+2[D_{345}]. \end{equation*}

\subsubsection{$\overline{M}_{0,6}$}

The polytope $\widetilde{P}(\kappa^{(6)})\subseteq \R^{\E_5}\simeq \R^{10}$ associated to the log-canonical class on $\overline{M}_{0,6}$ is a $9$-dimensional polytope with $355$ vertices. There are $125$ spanning trees of $K_5$, each of whose incidence vector is a vertex of $\widetilde{P}(\kappa^{(6)})$. Of the remaining vertices, $220$ are integral and $10$ are non-integral. The non-integral vertices are indexed by the edges of $K_5$ and are given in coordinates up to symmetry by

\begin{equation*}x_e= \begin{cases}
    -1/2 & \text{ if } e = \{1,2 \},\{1,3 \},\{2,3 \},\text{ or }\{4,5 \},\\
    1 & \text{ else}.
\end{cases} \end{equation*}
This non-integral vertex corresponds to the effective boundary expression
\begin{equation}
    \kappa^{(6)} = 3D_{123}+\frac{3}{2}(D_{12}+D_{13}+D_{23}+D_{45}+D_{46}+D_{56})+\frac{1}{2}\sum_{\substack{A\subseteq [6],\\ \vert A \vert =3,A\ni 1}}D_A.
\end{equation}

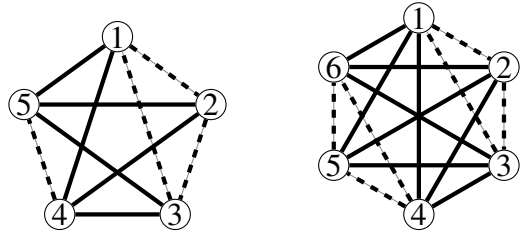
\begin{figure}[h]
\centering
\begin{tikzpicture}
  \graph[
    circular placement,
    radius=1.3cm,
    nodes={
      circle,
      draw,
      minimum size=4mm,
      inner sep=0pt
    },
    edges={black!45}
  ] {subgraph K_n [n=5, clockwise]};
  
    \draw (1) edge[dashed,ultra thick] (2);
    \draw (2) edge[dashed,ultra thick] (3);
    \draw (3) edge[dashed,ultra thick] (1);
    \draw (4) edge[dashed,ultra thick] (5);

    \draw (1) edge[ultra thick] (5);
    \draw (1) edge[ultra thick] (4);
    \draw (2) edge[ultra thick] (4);
    \draw (2) edge[ultra thick] (5);
    \draw (3) edge[ultra thick] (4);
    \draw (3) edge[ultra thick] (5);
\end{tikzpicture}
\hspace{1cm}
\begin{tikzpicture}
  \graph[
    circular placement,
    radius=1.3cm,
    nodes={
      circle,
      draw,
      minimum size=4mm,
      inner sep=0pt
    },
    edges={black!45}
  ] {subgraph K_n [n=6, clockwise]};
  
    \draw (1) edge[dashed,ultra thick] (2);
    \draw (2) edge[dashed,ultra thick] (3);
    \draw (3) edge[dashed,ultra thick] (1);
    \draw (4) edge[dashed,ultra thick] (5);
    \draw (5) edge[dashed,ultra thick] (6);
    \draw (4) edge[dashed,ultra thick] (6);

    \draw (1) edge[ultra thick] (6);
    \draw (1) edge[ultra thick] (5);
    \draw (1) edge[ultra thick] (4);
    \draw (2) edge[ultra thick] (4);
    \draw (2) edge[ultra thick] (5);
    \draw (2) edge[ultra thick] (6);
    \draw (3) edge[ultra thick] (4);
    \draw (3) edge[ultra thick] (5);
    \draw (3) edge[ultra thick] (6);
\end{tikzpicture}
\caption{Graphs depicting corresponding nonintegral vertices of $\widetilde{P}(\kappa^{(6)})$ and $\widetilde{P}^+(\kappa^{(6)})$. Solid edges have weight $1$, and dashed edges have weight $-1/2$.}
\label{fig:nonintegral vertices}
\end{figure}

\section{Application III: A decomposition of the Held--Karp relaxation polytope}

In the previous section, we studied the effective boundary expressions for the log-canonical class $\kappa^{(n)}$ on $\M$. In particular, we have identified the nonnegative parts of the polytope associated to $\kappa^{(n)}$ in coordinates with the spanning tree polytope of $K_{n-1}$ and the subtour relaxation of the Hamiltonian cycle polytope of $K_n$ respectively:
\begin{equation*}
\widetilde{P} \left(\kappa^{(n)}\right) \cap\R_{\geq0}^{\E_{n-1}}=P_{\mathrm{ST},n-1}, \text{ and } \widetilde{P}^{+} \left(\kappa^{(n)}\right)\cap\R_{\geq0}^{\E_n}
=
P_{\mathrm{SEP},n}.
\end{equation*}
\begin{rem}\label{rem:P_SEP_cube}
Same identities hold with the hypercube truncations:
\begin{equation*}
\widetilde{P} \left(\kappa^{(n)}\right) \cap [0,1]^{\E_{n-1}}=P_{\mathrm{ST},n-1}, \text{ and } \widetilde{P}^{+} \left(\kappa^{(n)}\right)\cap [0,1]^{\E_n}
=
P_{\mathrm{SEP},n}
\end{equation*}
as $x_e\leq 1$ by construction of the polytopes $\widetilde{P} \left(\kappa^{(n)}\right)$ and $\widetilde{P}^{+} \left(\kappa^{(n)}\right)$.
 
\end{rem}
The goal of this section is to study the decompositions of boundary divisor expressions of $\kappa^{(n)}$ coming from the relation given in Corollary \ref{cor:kappa_divisor_decomoposition_into_omega_i_s}:
\begin{equation*} \kappa^{(n)} = \omega_4+\omega_5+\cdots+\omega_n, \end{equation*}
where $\omega_i$ denotes the Kapranov class $X_{[i],i}$ on $\M$. 

\begin{lem}\label{lem:kappa omega polytope inclusion}
    There is a containment of polytopes
    \begin{equation*} Q(\kappa^{(n)}) \supseteq Q(\omega_4)+Q(\omega_5)+\cdots+Q(\omega_n). \end{equation*}
\end{lem}

\begin{proof}
    This follows from the decomposition $\kappa^{(n)} = \omega_4+\cdots+\omega_{n}$ and repeated applications of Lemma \ref{lem:subadditivity and scaling}.
\end{proof}

The inclusion in Lemma \ref{lem:kappa omega polytope inclusion} is strict in
general.  We will prove the cube-truncated equality first and then use it to
describe the $0/1$-points of the Minkowski sum exactly.

\begin{thm}\label{thm:kappa omega decomp}
    For $n\geq 4$, we have the following decompositions of the polytope associated to the log-canonical class $\kappa$ on $\M$, 
    \begin{align*}
        P^+(\kappa) \cap [0,1]^{\E_n} & = \left(  P^+(\kappa-\psi_n) + P^+(\psi_n)\right) \cap [0,1]^{\E_n} \\
        & = \left( P^+(\omega_4)+P^+(\omega_5)+\cdots+P^+(\omega_n) \right)\cap [0,1]^{\E_n}. 
    \end{align*} 
\end{thm}

Before giving the proof, we give an example of this decomposition for a $0/1$-point of $P^+(\kappa)$.

\begin{ex} \label{ex:splitting decomp}
Consider the Hamiltonian cycle $C=1-4-5-7-3-2-6-1$ in $K_7$.  Its complement, depicted in Figure \ref{fig:cycle complement}, is a point of $P^+(\kappa^{(7)})\subseteq\R^{\E_7}$.  Figures \ref{fig:psi term} and \ref{fig:smaller cycle complement} depict summands in $P^+(\psi_7)$ and $P^+(\kappa^{(7)}-\psi_7)$ whose sum is the complement of $C$. Indeed, Figure \ref{fig:psi term} shows the vertex of $P^+(\psi_7)$ indexed by the edge $\{3,5\}\subseteq K_6$, whose endpoints are the two neighbors of $7$ in $C$, in the sense of Corollary \ref{cor:Kapranov P plus}. Figure \ref{fig:smaller cycle complement} shows the complement of the cycle $C'=1-4-5-3-2-6-1$ in $K_6$, which is a point of
\begin{equation*}
P^+(\kappa^{(7)}-\psi_7)=P^+(\pi^*\kappa^{(6)}).
\end{equation*}
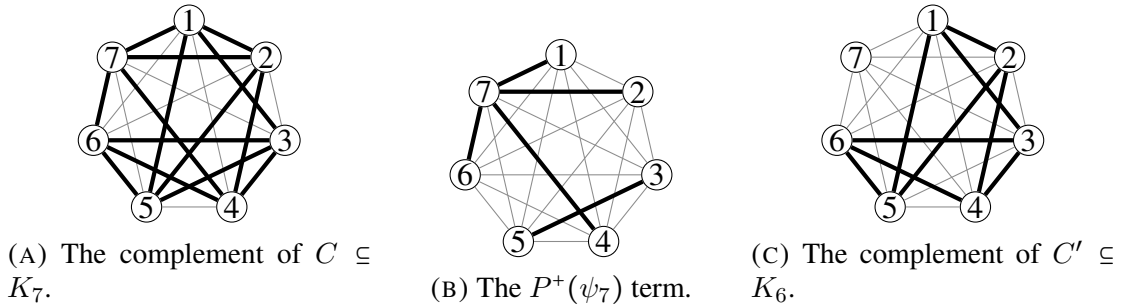
\begin{figure}[h]
\centering
\begin{subfigure}{0.28\textwidth}
\centering
\begin{tikzpicture}
  \graph[
    circular placement,
    radius=1.3cm,
    nodes={
      circle,
      draw,
      minimum size=4mm,
      inner sep=0pt
    },
    edges={black!45}
  ] {subgraph K_n [n=7, clockwise]};

    \draw (1) edge[ultra thick] (2);
    \draw (1) edge[ultra thick] (3);
    \draw (1) edge[ultra thick] (5);
    \draw (1) edge[ultra thick] (7);
    \draw (2) edge[ultra thick] (4);
    \draw (2) edge[ultra thick] (5);
    \draw (2) edge[ultra thick] (7);
    \draw (3) edge[ultra thick] (4);
    \draw (3) edge[ultra thick] (5);
    \draw (3) edge[ultra thick] (6);
    \draw (4) edge[ultra thick] (6);
    \draw (4) edge[ultra thick] (7);
    \draw (5) edge[ultra thick] (6);
    \draw (6) edge[ultra thick] (7);
\end{tikzpicture}
\caption{The complement of $C\subseteq K_7$.}
\label{fig:cycle complement}
\end{subfigure}
\begin{subfigure}{0.28\textwidth}
\centering
\begin{tikzpicture}
  \graph[
    circular placement,
    radius=1.3cm,
    nodes={
      circle,
      draw,
      minimum size=4mm,
      inner sep=0pt
    },
    edges={black!45}
  ] {subgraph K_n [n=7, clockwise]};
    \draw (1) edge[ultra thick] (7);
    \draw (2) edge[ultra thick] (7);
    \draw (3) edge[ultra thick] (5);
    \draw (4) edge[ultra thick] (7);
    \draw (6) edge[ultra thick] (7);
\end{tikzpicture}
\caption{The $P^+(\psi_7)$ term.}
\label{fig:psi term}
\end{subfigure}
\begin{subfigure}{0.28\textwidth}
\centering
\begin{tikzpicture}
  \graph[
    circular placement,
    radius=1.3cm,
    nodes={
      circle,
      draw,
      minimum size=4mm,
      inner sep=0pt
    },
    edges={black!45}
  ] {subgraph K_n [n=7, clockwise]};

    \draw (1) edge[ultra thick] (2);
    \draw (1) edge[ultra thick] (3);
    \draw (1) edge[ultra thick] (5);
    \draw (2) edge[ultra thick] (4);
    \draw (2) edge[ultra thick] (5);
    \draw (3) edge[ultra thick] (4);
    \draw (3) edge[ultra thick] (6);
    \draw (4) edge[ultra thick] (6);
    \draw (5) edge[ultra thick] (6);
\end{tikzpicture}
\caption{The complement of $C'\subseteq K_6$.}
\label{fig:smaller cycle complement}
\end{subfigure}

\caption{One splitting-off step and the corresponding decomposition of the complement of $C$ into a $P^+(\psi_7)$ summand and a $P^+(\kappa^{(7)}-\psi_7)$ summand.}
\label{fig:splitting off inductive step}
\end{figure}

The cycle $C'$ is obtained from $C$ by splitting-off the vertex $7$: one removes the edges $\{5,7\}$ and $\{7,3\}$ and adds the edge $\{3,5\}$. Repeating this procedure at the vertices $6$, $5$, and $4$ produces the full decomposition of the complement of $C$ as an element of $P^+(\omega_7)+P^+(\omega_6)+P^+(\omega_5)+P^+(\omega_4)$ shown in Figure \ref{fig:splitting off decomp}.

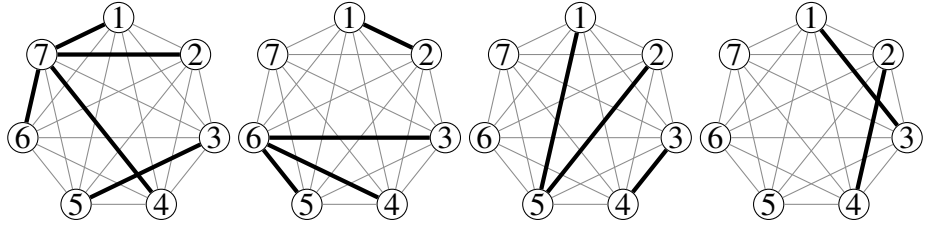
\begin{figure}[h]
\centering
\begin{tikzpicture}
  \graph[
    circular placement,
    radius=1.3cm,
    nodes={
      circle,
      draw,
      minimum size=4mm,
      inner sep=0pt
    },
    edges={black!45}
  ] {subgraph K_n [n=7, clockwise]};
    \draw (1) edge[ultra thick] (7);
    \draw (2) edge[ultra thick] (7);
    \draw (3) edge[ultra thick] (5);
    \draw (4) edge[ultra thick] (7);
    \draw (6) edge[ultra thick] (7);
\end{tikzpicture}
\begin{tikzpicture}
  \graph[
    circular placement,
    radius=1.3cm,
    nodes={
      circle,
      draw,
      minimum size=4mm,
      inner sep=0pt
    },
    edges={black!45}
  ] {subgraph K_n [n=7, clockwise]};

    \draw (1) edge[ultra thick] (2);

    \draw (3) edge[ultra thick] (6);
    \draw (4) edge[ultra thick] (6);

    \draw (5) edge[ultra thick] (6);

\end{tikzpicture}
\begin{tikzpicture}
  \graph[
    circular placement,
    radius=1.3cm,
    nodes={
      circle,
      draw,
      minimum size=4mm,
      inner sep=0pt
    },
    edges={black!45}
  ] {subgraph K_n [n=7, clockwise]};

    \draw (1) edge[ultra thick] (5);

    \draw (2) edge[ultra thick] (5);

    \draw (3) edge[ultra thick] (4);

\end{tikzpicture}
\begin{tikzpicture}
  \graph[
    circular placement,
    radius=1.3cm,
    nodes={
      circle,
      draw,
      minimum size=4mm,
      inner sep=0pt
    },
    edges={black!45}
  ] {subgraph K_n [n=7, clockwise]};

    \draw (1) edge[ultra thick] (3);

    \draw (2) edge[ultra thick] (4);
 
\end{tikzpicture}

\caption{The terms in the decomposition of the complement of $C\subseteq K_7$ from $P^+(\omega_7)$, $P^+(\omega_6),$ $P^+(\omega_5)$, and $P^+(\omega_4)$ respectively.}
\label{fig:splitting off decomp}
\end{figure}
\end{ex}

For the $0/1$-points of $P^+(\kappa^{(n)})$, complements of Hamiltonian cycles in $K_n$, the decompositions as in Theorem \ref{thm:kappa omega decomp} can be described easily by the splitting-off procedure shown in Example \ref{ex:splitting decomp}. This is insufficient to give the decomposition of an arbitrary point of $P^+(\kappa^{(n)})\cap[0,1]^{\E_n}$, however, since this not an integral polytope in general. For the general argument, we use the following Theorem of Lov\'asz.

\begin{thm}[Lov\'asz splitting-off theorem \cite{Lovasz1976}] \label{thm:lovasz_splitting_theorem}
Let $G$ be a multigraph (without loops) on vertex set $[n]$, let $k\in [n]$, and let
$\lambda\geq 2$ be an integer. For a subset $I\subseteq [n]$, write
\begin{equation*}
\delta_G(I) \coloneqq \{\{u,v\}\in \E(G)\, \vert \, u\in I,\ v\in [n]\setminus I \}.   
\end{equation*}

Assume that $\deg_G(k)$ is even and
\begin{equation*}
 \vert \delta_G(I) \vert \geq \lambda
\quad\text{for every}\quad
\varnothing\neq I\subsetneq [n]\setminus \{k\}.
\end{equation*}

Then, for every edge $e=\{k,q\}\in \E(G)$ incident to $k$, there exists another edge
$f=\{k,r\}\in \E(G)\setminus \{e\}$ incident to $k$ such that, if $G'$ is obtained from
$G$ by splitting-off $e$ and $f$ at the vertex $k$, then
\begin{equation*}
 \vert \delta_{G'}(I) \vert \geq \lambda
\quad\text{for every}\quad
\varnothing\neq I\subsetneq [n]\setminus \{k\}.
\end{equation*}
\end{thm}
For a modern treatment of splitting-off, we refer the reader to \cite[Chapter 8]{Frank2011}. With Lov\'asz splitting-off theorem, we can now give the proof of Theorem \ref{thm:kappa omega decomp}.

\begin{proof}[Proof of Theorem \ref{thm:kappa omega decomp}.]
    We begin by showing that the second equality follows from the first by induction. The base case $n=4$ is trivial since $\kappa =\psi_4 = \omega_4$ by equations \eqref{eqn:kappa_1_wrt_divisors} and \eqref{eqn:psi_wrt_divisors} in Lemma \ref{lem:some_relations_for_kappa_psi_and_D_A}. During the inductive step we write divisors on $\M$ as $D^{(n)}$ for clarity.

    By induction, suppose that
    \begin{equation} \label{eq:induction hypoth}
    P^+(\kappa^{(n-1)}) \cap [0,1]^{\E_{n-1}} = \left( P^+(\omega_4^{(n-1)})+ \cdots P^+(\omega_{n-1}^{(n-1)})\right) \cap [0,1]^{\E_{n-1}}.
    \end{equation}

    Let $\pi  \colon \M\to \overline{M}_{0,n-1}$ denote the map forgetting the $n$th marked point. By Corollary, the polytope \ref{cor:extension by 0},  $P^+(\pi^*(D^{(n-1)}))\subseteq \R^{\E_n}$ coincides with $P^+(D^{(n-1)})\subseteq \R^{\E_{n-1}}\subseteq \R^{\E_n}$ where we identify $\R^{\E_{n-1}}$ with the coordinate subspace defined by $a_{\{i,n\}} = 0$ for all $i=1,\dots,n-1$. Recalling that $\kappa^{(n)}-\psi_n^{(n)} = \pi^*(\kappa^{(n-1)})$, and $\pi^*(\omega_i^{(n-1)}) = \omega_i^{(n)}$, the induction hypothesis therefore gives the equality
    
    \begin{equation*}
    P^+(\kappa^{(n)}-\psi_n) \cap [0,1]^{\E_n} = \left( P^+(\omega_4^{(n)})+ \cdots P^+(\omega_{n-1}^{(n)})\right) \cap [0,1]^{\E_n}.
    \end{equation*}
    Adding $P^+(\psi_n^{(n)}) = P^+(\omega_n^{(n)})$, to each side, we then have
    \begin{equation*}
    \left(P^+(\kappa^{(n)}-\psi_n) \cap [0,1]^{\E_n}\right) + P^+(\psi_n^{(n)}) = \left(\left( P^+(\omega_4^{(n)})+ \cdots P^+(\omega_{n-1}^{(n)})\right) \cap [0,1]^{\E_n}\right) + P^+(\omega_n^{(n)})
    \end{equation*}
    Now we intersect both sides with $[0,1]^{\E_n}$. Since $P^+(D^{(n)}) \subseteq \R^{\E_n}_{\geq 0}$ for any divisor $D^{(n)}$, the intersections with $[0,1]^{\E_n}$ inside the summands are redundant, so we obtain 
    \begin{equation*}
    \left(P^+(\kappa^{(n)}-\psi_n)+ P^+(\psi_n^{(n)}) \right) \cap [0,1]^{\E_n} = \left( P^+(\omega_4^{(n)})+ \cdots P^+(\omega_{n-1}^{(n)}) + P^+(\omega_n^{(n)})\right) \cap [0,1]^{\E_n}
    \end{equation*}
    This completes the proof that the second equality follows from the first.

    Towards one of the inclusions of the first equality, Lemma \ref{lem:subadditivity and scaling} gives the inclusion 
    \begin{equation*}Q(\kappa)  \supseteq Q(\kappa-\psi_n) + Q(\psi_n). \end{equation*}
    This inclusion is preserved by the operations of projection onto $\R^{\E_n}$ and intersection with $[0,1]^{\E_n}$, which establishes the containment
    \begin{equation*}P^+(\kappa) \cap [0,1]^{\E_n} \supseteq \left( P^+(\kappa-\psi_n) + P^+(\psi_n) \right) \cap [0,1]^{\E_n}. \end{equation*}

    For the final inclusion, we take $a = (a_e)_{e \in \E_n}\in P^+(\kappa) \cap [0,1]^{\E_n}$ and aim to construct $a'\in P^+(\kappa-\psi_n)$ and $b\in P^+(\psi_n)$ such that $a = a'+b$. It is sufficient to assume that $a$ has rational coordinates, so we fix an integer $\mu\geq 1$ such that $\mu a_e \in \Z$ for all $e \in \E_n$.

    Let $x = (x_e)_{e \in \E_n}$ be defined by $x_e = 1-a_e$ as in the previous section, so that by Corollary \ref{cor:spanning tree polytope and kappa} and Remark \ref{rem:P_SEP_cube}, we have 
    \begin{equation*}x \in \widetilde{P}^+(\kappa) \cap [0,1]^{\E_n} = P_{\mathrm{SEP},n}. \end{equation*}

    We will construct the decomposition inductively by splitting-off from the multigraph $G = G_0$ on vertices $[n]$ for which edge $e$ has multiplicity $\mu x_e$. Indeed, since $x\in P_{\mathrm{SEP},n}$, we have $\deg_G(k) = 2\mu $ for every vertex $k$ and 
    \begin{equation*}
      \vert \delta_G(I)  \vert \geq 2\mu
    \end{equation*}
    for every $\varnothing\neq I \subseteq [n]$.

    We now apply the Lov\'{a}sz Splitting-off Theorem  \ref{thm:lovasz_splitting_theorem} to $G$ with $k=n$ and $\lambda = 2\mu$. Choose any edge $e =\{i_1,n\}$ incident to $n$, and let $f =\{j_1,n\}$ be the corresponding edge and $G_1$ the resulting graph after splitting-off as in the theorem. This new graph satisfies the same conditions except has $\deg_{G_1}(n) = 2(\mu-1)$. We may therefore repeat this splitting-off process $\mu$ times producing a sequence of indices $i_1,j_1,\dots,i_{\mu},j_{\mu}$ and a graph $G_{\mu}$ satisfying 
    \begin{equation*}\deg_{G_{\mu}}(k) =  \begin{cases}
        2\mu & k<n\\
        0 & k=n
    \end{cases}\end{equation*}
    and $  \vert \delta_G(S)  \vert \geq 2\mu$ for every $\varnothing\subsetneq S\subsetneq [n-1]$. Let $x^{G_{\mu}}\in \R^{E_n}$ denote the vector where $x^{G_{\mu}}_e$ is the multiplicity of $e$ in $G_{\mu}$. The sequence of splitting-offs taking $G$ to $G_{\mu}$ can be interpreted as the identity
    \begin{equation*}x^G = x^{G_{\mu}}- x^{i_1 j_1,n}- \cdots - x^{i_{\mu} j_{\mu},n}, \end{equation*}
    where $x^{ij,n}\in \R^{\E_n}$ is the vector with coordinates
    \begin{equation*}x^{ij,n}_e = \begin{cases}
        1 & e = \{ i,j\},\\
        -1 & e = \{i,n\},\{j,n\},\\
        0 & \text{else}.
    \end{cases} \end{equation*}
    Dividing by $\mu$ gives the decomposition
    \begin{equation*}x = \frac{1}{\mu}x^{G_{\mu}}- \frac{x^{i_1 j_1,n}+ \cdots + x^{i_{\mu} j_{\mu},n}}{\mu}. \end{equation*}
    The conditions on $G_{\mu}$ precisely say that 
    \begin{equation*}x'  \coloneqq  \frac{1}{\mu} x^{G_{\mu}}\in P_{\mathrm{SEP},n-1}\subseteq \R^{\E_{n-1}}.\end{equation*}

    To transform back into our original coordinates, we set 
    \begin{equation*}a'_e = \begin{cases}
        1-x'_e & e \in \E_{n-1}\\
        0 & \text{else}
    \end{cases} \end{equation*}
    and define the vectors
    \begin{equation*}b^{ij,n} = \begin{cases}
        1 & e= \{ i,j \} \text{ or } \{\ell,n\} \text{ for some } \ell\in [n]\setminus\{i,j,n\},\\
        0 & \text{else}
    \end{cases} \end{equation*}
    and 
    \begin{equation*}b = \frac{b^{i_1 j_1,n}+ \cdots + b^{i_{\mu} j_{\mu},n}}{\mu} \end{equation*}
    This transformation is defined so that $a = a'+b$. Moreover, by Corollary \ref{cor:spanning tree polytope and kappa}, we have an element $a\in P^+(\kappa-\psi_n)$. By Corollary \ref{cor:Kapranov P plus} each $b^{ij,n}$ is a vertex of $P^+(\psi_n)$ so we have $b\in P^+(\psi_n)$ since it is a convex combination of the vertices. This gives the desired decomposition of the point $a\in P^+(\kappa)\cap [0,1]^{\E_n}$ which completes the proof.
\end{proof}

Although it is covered in the proof, we want to highlight the following Minkowski-like decomposition of subtour elimination relaxation polytope $P_{\mathrm{SEP},n}$.

\begin{cor}\label{cor:Held-Karp_Decomposition}
The graph complement $\mathbf{1}-P_{\mathrm{SEP},n}$ of the Held--Karp relaxation polytope $P_{\mathrm{SEP},n}$ satisfies the following decomposition:
\begin{equation*}
\mathbf{1}-P_{\mathrm{SEP},n}= \left( P^+(\omega_4)+P^+(\omega_5)+\cdots+P^+(\omega_n) \right)\cap [0,1]^{\E_n}.
\end{equation*}
\end{cor}
\begin{proof}
This directly follows from Remark \ref{rem:P_SEP_cube} and Theorem \ref{thm:kappa omega decomp}.
\end{proof}

\begin{cor}\label{cor:01 points of Minkowski sum}
Let $n\geq 4$, then we have
\begin{equation*}
    (P^+(\omega_4) + \cdots + P^+(\omega_n))\cap \{0,1\}^{\E_{n}} = \text{Hamiltonian cycle complements in }K_{n},
    \end{equation*}
    and
    \begin{equation*}
    (P(\omega_4) + \cdots + P(\omega_n))\cap \{0,1\}^{\E_{n-1}} = \text{Hamiltonian path complements in }K_{n-1}.
    \end{equation*}
\end{cor}

\begin{proof}
Hamiltonian cycle claim follows from Corollary \ref{cor:Held-Karp_Decomposition} as $0/1$-points of subtour elimination polytope $P_{\mathrm{SEP},n}$ are Hamiltonian cycles. Now, the projection $\operatorname{pr}_n  \colon \R^{\E_n}\longrightarrow \R^{\E_{n-1}}$ forgetting the edges attached to the vertex $n$ gives us
\begin{equation*}
\operatorname{pr}_n \left(P^+(\omega_4) + \cdots + P^+(\omega_n)\right)=(P(\omega_4) + \cdots + P(\omega_n))
\end{equation*}
as we have $\operatorname{pr}_n(P^+(D))=P^(D)$ for any divisor $D$ on $\M$ by definition and projection commutes with Minkowski sum. Note that this projection sends $0/1$-points to $0/1$-points and forgetting the edges connected to vertex $n$ sends a Hamiltonian cycle complement in $K_n$ to a Hamiltonian path complement in $K_{n-1}$, and every Hamiltonian path in $K_{n-1}$ can be extended to a Hamiltonian cycle in $K_n$ by joining its ends to the vertex $n$. Then,  we have
\begin{equation*}
    \text{Hamiltonian path complements in }K_{n-1}\subseteq (P(\omega_4) + \cdots + P(\omega_n))\cap \{0,1\}^{\E_{n-1}}.
\end{equation*}

Now, we will prove the reverse inclusion to complete the proof of the second identity in Corollary \ref{cor:01 points of Minkowski sum}. Let $\mathbf{a}=(a_e)_{e \in \E_{n-1}}\in (P(\omega_4) + \cdots + P(\omega_{n-1}) + P(\omega_n))\cap \{0,1\}^{\E_{n-1}}$. We have an inclusion 
\begin{equation*} P(\omega_4) + \cdots + P(\omega_{n-1}) + P(\omega_n)\subseteq P(\kappa^{(n)}), \end{equation*} 
and by Corollary \ref{cor:spanning tree polytope and kappa} the $0/1$-points of $P(\kappa^{(n)})$ are graph complements of spanning trees of $K_{n-1}$. Hence, $\mathbf{a}$ is the complement of a spanning tree $T$ of $K_{n-1}$. 

By assumption, there exist $\mathbf{b}=(b_e)_{e \in \E_{n-1}}\in P(\omega_4) + \cdots + P(\omega_{n-1})$ and $\mathbf{c}=(c_e)_{e \in \E_{n-1}}\in P(\omega_n)$ such that $\mathbf{a}=\mathbf{b}+\mathbf{c}$. Since each $b_e$ and $c_e$ are nonnegative and $a_e \in \{0,1\}$, we have $b_e,c_e \in [0,1]$. Hence, we have
\begin{equation*}
\mathbf{b}\in (P(\omega_4) + \cdots + P(\omega_{n-1}))\cap [0,1]^{\E_{n-1}}=\mathbf{1}-P_{\mathrm{SEP},n-1}
\end{equation*}
or equivalently
\begin{equation*}
\overline{\mathbf{b}}\coloneqq   \mathbf{1}-\mathbf{b}\in P_{\mathrm{SEP},n-1}.
\end{equation*}
For the incidence vector $\overline{\mathbf{a}}\coloneqq   \mathbf{1}-\mathbf{a}$ of the spanning tree $T$, we have
\begin{equation*}
\overline{\mathbf{a}}=\mathbf{1}-\mathbf{a}=\mathbf{1}-\mathbf{b}-\mathbf{c}=\overline{\mathbf{b}}-\mathbf{c}.
\end{equation*}
Since $\overline{\mathbf{b}}\in P_{\mathrm{SEP},n-1}$, we have $\overline{\mathbf{b}}(\delta (i))=2$ for all $i\in [n-1]$. As a result, we see that
\begin{equation*}
\deg_{T}(i)=\overline{\mathbf{a}}(\delta(i))=\overline{\mathbf{b}}(\delta (i))-\mathbf{c}(\delta (i))=2-\mathbf{c}(\delta (i))\leq 2
\end{equation*}
for any $i \in [n-1]$. Hence, $T$ is a spanning tree whose vertices have degree at most $2$, i.e., a Hamiltonian path of $K_{n-1}$ which completes the proof.
\end{proof}

\subsection{Sharpness of Theorem \ref{thm:kappa omega decomp}}

Theorem \ref{thm:kappa omega decomp} gives a sufficient condition for a point $a\in Q(\kappa^{(n)})$ to lie in the Minkowski sum
\begin{equation*}  Q(\omega_4)+\cdots+Q(\omega_n) \subseteq Q(\kappa^{(n)}). \end{equation*}
Indeed, Theorem \ref{thm:kappa omega decomp} states that if $a\in Q(\kappa^{(n)})$ has $a_e\leq 1$ for all $e \in \E_n\subseteq \II_n^+$, then \begin{equation*}
a\in Q(\omega_4)+\cdots+Q(\omega_n).
\end{equation*}

In this section, we show several directions in which this sufficient condition cannot be relaxed. The first such relaxation we consider is the condition that $a_e\leq 1$ for all $e \in \E_{n-1}\subseteq \II_n$ rather than all $e \in \E_{n}\subseteq \II_n^+$.

\begin{cor}\label{cor:strictness P^+ to P}
    For all $n \geq 5$, the containment 
    \begin{equation*} (P(\omega_4)+\cdots+P(\omega_n))\cap [0,1]^{\E_{n-1}}\subseteq P(\kappa^{(n)})\cap [0,1]^{\E_{n-1}} \end{equation*}
    is strict. In particular, the containment
    \begin{equation*} Q(\omega_4)+\cdots+Q(\omega_n) \subseteq Q(\kappa^{(n)}) \end{equation*}
    is also strict for all $n\geq 5$.
\end{cor}

\begin{proof}
    By Corollary \ref{cor:01 points of Minkowski sum}, the $0/1$-points of the Minkowski sum are complements of Hamiltonian paths in $K_{n-1}$. On the other hand, the $0/1$-points of $P(\kappa^{(n)})$ are complements of spanning trees in $K_{n-1}$ by Corollary \ref{cor:spanning tree polytope and kappa}. Since $K_{n-1}$ has spanning trees that are not Hamiltonian paths for $n-1\geq 4$, we see that the inclusion
    \begin{equation*}
    (P(\omega_4)+\cdots+P(\omega_n))\cap [0,1]^{\E_{n-1}}\subseteq P(\kappa^{(n)})\cap [0,1]^{\E_{n-1}}
    \end{equation*}
    must be strict for $n\geq 5$ as their $0/1$ points differ. The second claim follows as $(P(\omega_4)+\cdots+P(\omega_n))$ and $P(\kappa^{(n)})$ are coordinate projections of $Q(\omega_4)+\cdots+Q(\omega_n)$ and $Q(\kappa^{(n)})$, respectively.
\end{proof}

On the other hand, weakening the upper bound of $1$ on the $a_e$ coordinates is also not possible.

\begin{prop}
    For all $n \geq 6$ and $t>0$, the containment 
    \begin{equation*} (P^+(\omega_4)+\cdots+P^+(\omega_n))\cap [0,1+t]^{\E_{n-1}}\subseteq P^+(\kappa^{(n)})\cap [0,1+t]^{\E_{n-1}} \end{equation*}
    is strict.
\end{prop}

\begin{proof}
    By Corollary \ref{cor:strictness P^+ to P}, the set
    \begin{equation*} Q(\kappa^{(n)})\setminus (Q(\omega_4)+\cdots+Q(\omega_n)) \end{equation*}
    is nonempty, so we may choose a point in its projection onto $(a_e)_{e \in \E_n}$ coordinates
    \begin{equation*} a\in P^+(\kappa^{(n)})\setminus (P^+(\omega_4)+\cdots+P^+(\omega_n)). \end{equation*}
    By Theorem \ref{thm:kappa omega decomp}, such a point must have $a_e>1$ for some $e \in \E_n$. By taking an appropriate weighted sum with any point in $P^+(\kappa^{(n)})\cap [0,1]^{\E_{n-1}}$ we may assume that $a_e\leq 1+t$ for all $e \in \E_n$. Since the divisor class $\kappa^{(n)}$ is fixed under pullback by the $S_n$ action on $\M$, we may further assume that there is some $e \in \E_n\setminus \E_{n-1}$ for which $a_e>1$. 

    On the other hand, we have $P^+(\omega_k)\subseteq \R^{\E_k}\subseteq \R^{\E_n}$ where coordinates indexed by edges $e \in \E_n\setminus \E_k$ are zero. Additionally, $P^+(\omega_k)$ is a $0/1$-polytope. These facts imply that any point 
    \begin{equation*} (b_e)_{e \in \E_n} \in P^+(\omega_4)+\cdots+P^+(\omega_n) \end{equation*}
    has $b_e\leq 1$ for all $e \in \E_n\setminus \E_{n-1}$. This implies that the previously constructed point 
    \begin{equation*} a\in P^+(\kappa^{(n)})\cap [0,1+t]^{\E_{n-1}} \end{equation*}
    does not lie in the Minkowski sum, as desired.
\end{proof}

\section{Application IV: level-one $\mathfrak{sl}_p$ conformal block divisors and Turán graphs}

In this section, we will compute the polytope of effective boundary expressions ${P}^+(D^{\mathfrak{sl}_p}_{1,\Vec{d}})$ for the conformal block divisors $D^{\mathfrak{sl}_p}_{1,\Vec{d}}$ on $\overline{M}_{0,n}$.  This is equivalent to computing the polytope of effective boundary expressions ${P}(D^{\mathfrak{sl}_p}_{1,(\Vec {d},0)})$ for the conformal block divisor $D^{\mathfrak{sl}_p}_{1,(\Vec {d},0)}$ on $\overline{M}_{0,n+1}$ as we have
\begin{equation*}
P^+(D^{\mathfrak{sl}_p}_{1,\Vec{d}})=P(\pi_{n+1}^* (D^{\mathfrak{sl}_p}_{1,\Vec{d}}))=P(D^{\mathfrak{sl}_p}_{1,(\Vec {d},0)})
\end{equation*}
by Corollary \ref{cor: P^+ and P relationship} and Remark \ref{rem:conformal_pullback_slp}. Since the latter amounts to computing the Kapranov basis expression of $D^{\mathfrak{sl}_p}_{1,(\Vec {d},0)}$, we will focus on $P(D^{\mathfrak{sl}_p}_{1,(\Vec {d},0)})$ and $\widetilde{P}(D^{\mathfrak{sl}_p}_{1,(\Vec {d},0)})$. Moreover, we will analyze the particular case of $$\Vec{d}=(1,\ldots,1)$$ and its connection to  Tur{\'a}n graphs and Tur{\'a}n polytopes separately. Lastly, we will focus on the cases where $n$ is even and $p=2,\frac{n}{2}$. We relate these two cases to the perfect matching polytope and the fractional perfect matching polytope of $K_n$, respectively.

\subsection{Background for Application IV}

Here, we briefly recall the conformal block divisors that will be used. Also, we will provide some background for Tur{\'a}n numbers, Tur{\'a}n graphs and Tur{\'a}n polytopes, perfect matching polytopes and fractional perfect matching polytopes which we use to identify the polytopes of effective boundary expressions of conformal block divisors we consider.

\subsubsection{Conformal block divisors}

 Let $\mathfrak{g}$ be a simple Lie algebra, $\ell$ be a positive integer, and $\overline{\lambda}=(\lambda_1,\ldots, \lambda_n)$ be an $n$-tuple of dominant integral weights of $\mathfrak{g}$ of level at most $\ell$. The \emph{conformal block} $\mathcal{V}^{\mathfrak{g}}({\ell, \overline{\lambda}})$ associated to this data is a vector bundle on the moduli space $\overline{M}_{g,n}$ of stable curves of genus $g$ with $n$ marked points. Here $\ell$ is called the \emph{level} of the  conformal block. The \emph{conformal block divisor} $D^{\mathfrak{g}}({\ell, \overline{\lambda}})$ is the first Chern class of the conformal block $\mathcal{V}^{\mathfrak{g}}({\ell, \overline{\lambda}})$. Systematic study of conformal block divisors has been initiated in \cite{Fakhruddin12}.

Our focus will be the case of conformal block divisors with $\mathfrak{g}=\mathfrak{sl}_p$ where $p\geq 2$ and  the level $\ell=1$ for $\M$. The level-one dominant weights of $\mathfrak{sl}_p$ are
\begin{equation*}
\varpi_0,\varpi_1,\ldots,\varpi_{p-1}
\end{equation*}
where $\varpi_0$ denotes the trivial weight and $\varpi_1,\ldots,\varpi_{p-1}$ are the fundamental weights. For the remaining of the paper, we will adopt the following notation
\begin{equation*}
D^{\mathfrak{sl}_p}_{1,\Vec{d}}\coloneqq   D^{\mathfrak{sl}_p}({1,(\varpi_{d_1},\ldots,  \varpi_{d_n})})\quad \text{where} \quad \Vec{d}=(d_1,\ldots,d_n)\in\{0,1,\ldots,p-1\}^n.
\end{equation*}
\begin{rem}
The conformal block divisor $D^{\mathfrak{sl}_p}_{1,\Vec{d}}$ on $\overline{M}_{0,n}$ is the zero divisor class unless $\sum_{i=1}^n d_i= p(c+1)$ for some $c\in \{1,\ldots, n-3\}$ \cite[Lemma 3.1]{GiansiracusaGibney12}. For this reason, we will only consider the case where $p$ divides  $\sum_{i=1}^n d_i$.
\end{rem}

\begin{rem}\label{rem:conformal_pullback_slp}
Let $\Vec{d}=(d_1,\ldots,d_n)\in \{0,1,\ldots,p-1\}^n$ and set 
\begin{equation*}
(\Vec {d},0)\coloneqq   (d_1,\ldots,d_n,0).
\end{equation*}
Then, we have
\begin{equation*}
D^{\mathfrak{sl}_p}_{1,(\Vec {d},0)}=\pi_{n+1}^* (D^{\mathfrak{sl}_p}_{1,\Vec{d}})
\end{equation*}
for the conformal block divisors $D^{\mathfrak{sl}_p}_{1,(\Vec {d},0)}$ on $\overline{M}_{0,n+1}$ and $D^{\mathfrak{sl}_p}_{1,\Vec{d}}$ on $\overline{M}_{0,n}$, c.f. \cite[Proposition 2.4]{Fakhruddin12}.
\end{rem}

Fix an integer $p\geq 2$. Consider a weight vector 
\begin{equation*}
\Vec {d}=(d_1,\ldots,d_n)\quad \text{such that}\quad \sum_{i=1}^n d_i\equiv 0 \mod{p}.
\end{equation*}
For $I \subseteq [n]$, set
\begin{equation*}
\Vec {d}(I) \coloneqq   \sum_{i\in I}d_i 
\end{equation*}
and define
\begin{equation*}
\theta(I) \coloneqq 
\langle \Vec{d}(I)\rangle_p
\left(p-\langle \Vec{d}(I)\rangle_p\right)
\end{equation*}
where
$\left \langle a \right \rangle _p \in \{0,1,\ldots, p-1\}$ denotes the residue of $a$ modulo $p$.
By \cite{F11}, we have the following identity.

\begin{lem}[\cite{F11}]\label{lem:slp_conformal_divisor_in_kapranov}
On $\overline{M}_{0,n+1}$, we have
\begin{equation*}
D^{\mathfrak{sl}_p}_{1,(\Vec {d},0)}
=\frac {1}{2p}\left (
\sum_{i=1}^n \theta (\{i\}) \psi_i-\sum_{e \in \E_n}\theta(e)[D_e]
-
\sum_{I\in \II_{n+1}, \vert I \vert \geq 3}\theta (I)[D_I]\right).
\end{equation*}
\end{lem}

\begin{proof}
Combining the results of \cite[Theorem 4.5]{F11} and \cite[Proposition 4.8]{F11}, we see that
\begin{equation*}
D^{\mathfrak{sl}_p}_{1,(\Vec {d},0)}
=\frac {1}{2p}\left (
\sum_{i=1}^n \theta (\{i\}) \psi_i
-
\sum_{I\in \II_{n+1}}\theta (I)[D_I]\right).
\end{equation*}
Separating $\II_{n+1}$ as $\II_{n+1} = \E_n \sqcup \{I\in \II_{n+1}\, \vert \, \vert I \vert \geq 3\}$, we obtain our desired expression.
\end{proof}

For further details on $\mathfrak{sl}_p$ conformal block divisors, we refer the reader to \cite{AlexeevGibneySwinarski14} and \cite{GiansiracusaGibney12}.

\subsubsection{Tur{\'a}n graphs and polytopes }
 Let $m$ and $p$ be positive integers with $p<m$. 

\begin{defn}
A \emph{Tur\'an graph} is a complete $p$-partite graph on $m$ vertices whose part sizes
differ by at most one. In our convention, a \emph{Tur\'an graph} is said to be a \emph{balanced Tur\'an graph} if the number of parts $p$ divides the number of vertices $m$.
\end{defn}

By definition, we see that  if $m=qp+\langle m \rangle_p$, then a Tur\'an graph has $\langle m \rangle_p$ parts of size $q+1$ and $p-\langle m \rangle_p$ parts of size $q$. Hence, a Tur\'an graph is given by a complete $p$-partite graph $K_{A_1,\ldots,A_p}$
where
\begin{equation*}
[m]=A_1 \sqcup \cdots \sqcup A_p
\end{equation*}
with $ \vert A_1  \vert =\ldots= \vert A_{\langle m \rangle_p} \vert = q+1$ and $ \vert A_{\langle m \rangle_p+1}  \vert =\ldots= \vert A_p \vert = q$. For this reason, we will denote such a complete $p$-partite graph by $T_{A_1,\ldots,A_p}$ to remind that it is a Tur\'an graph and we will denote their isomorphism class by $T(m,p)$ which can be considered as
\begin{equation*}
T(m,p)=K_{{q+1},\ldots,{q+1},{q}, \ldots,{q}}.
\end{equation*} Furthermore, we will denote the  complement of a  Tur\'an graph $T_{A_1,\ldots,A_p}$ in the complete graph $K_m$ by $\overline{T}_{A_1,\ldots,A_p}$. It is clear that the complement $\overline{T}_{A_1,\ldots,A_p}$ is a disjoint union of $p-\langle m \rangle_p$ many complete graphs, each isomorphic to $K_{q}$ and $\langle m \rangle_p$ many complete graphs where each isomorphic to $K_{q+1}$. More precisely, we have
\begin{equation*}
\overline{T}_{A_1,\ldots,A_p}=K_{A_1}\sqcup\cdots \sqcup K_{A_p}.
\end{equation*}
We will denote the isomorphism class of $\overline{T}_{A_1,\ldots,A_p}$ by $\overline{T}(m,p)$ which can be considered as
\begin{equation*}
\overline{T}(m,p) = \underbrace{K_{q+1}\sqcup \cdots \sqcup K_{q+1}}_{\langle m \rangle_p\text{ many}}\sqcup \underbrace{K_{q}\sqcup \cdots \sqcup K_{q}}_{p-\langle m \rangle_p\text{ many}}.
\end{equation*}

\begin{defn}
A graph on $m$ vertices is \emph{$K_{p+1}$-free} if it has no subgraph isomorphic to $K_{p+1}$.
\end{defn}

\begin{figure}[h]
\centering
\setlength{\tabcolsep}{2pt}
\begin{tabular}{c|c|c|c}

\begin{tikzpicture}
  \graph[circular placement, radius=1.3cm,
    nodes={circle,draw,minimum size=4mm,inner sep=0pt},
    edges={black!45}] {subgraph K_n [n=8, clockwise]};

  \foreach \a in {1,2,3}{
    \foreach \b in {4,5,6,7,8}{
      \draw (\a) edge[ultra thick] (\b);
    }
  }
  \foreach \a in {4,5,6}{
    \foreach \b in {7,8}{
      \draw (\a) edge[ultra thick] (\b);
    }
  }
\end{tikzpicture}
&
\begin{tikzpicture}
  \graph[circular placement, radius=1.3cm,
    nodes={circle,draw,minimum size=4mm,inner sep=0pt},
    edges={black!45}] {subgraph K_n [n=7, clockwise]};

  \foreach \a in {1,2,3,4}{
    \foreach \b in {5,6,7}{
      \draw (\a) edge[ultra thick] (\b);
    }
  }
\end{tikzpicture}
&
\begin{tikzpicture}
  \graph[circular placement, radius=1.3cm,
    nodes={circle, draw, minimum size=4mm, inner sep=0pt},
    edges={black!45}]
    {subgraph K_n [n=9, clockwise]};
  \foreach \a in {1,2,6} {\foreach \b in {3,4,5,7,8,9} {\draw (\a) edge[ultra thick] (\b);}}
  \foreach \a in {3,4,8} {\foreach \b in {5,7,9} {\draw (\a) edge[ultra thick] (\b);}}
\end{tikzpicture}
&
\begin{tikzpicture}
  \graph[circular placement, radius=1.3cm,
    nodes={circle,draw,minimum size=4mm,inner sep=0pt},
    edges={black!45}] {subgraph K_n [n=8, clockwise]};
  \foreach \a/\b in {1/2,1/3,1/5,1/6,1/7,1/8,2/4,2/5,2/6,2/7,2/8,3/4,
                     3/5,3/6,3/7,3/8,4/5,4/6,4/7,4/8,5/6,5/7,6/8,7/8}
    \draw (\a) edge[ultra thick] (\b);
\end{tikzpicture}
\\[2mm]
$T_{\{1,2,3\},\{4,5,6\},\{7,8\}}$ & $T_{\{1,2,3,4\},\{5,6,7\}}$ & $T_{\{1,2,6\},\{3,4,8\},\{5,7,9\}}$ & $T_{\{1,4\},\{2,3\},\{5,8\},\{6,7\}}$ \\[5mm]

\begin{tikzpicture}
  \graph[circular placement, radius=1.3cm,
    nodes={circle,draw,minimum size=4mm,inner sep=0pt},
    edges={black!45}] {subgraph K_n [n=8, clockwise]};

  \foreach \i/\j in {1/2,1/3,2/3,4/5,4/6,5/6,7/8}
    \draw (\i) edge[ultra thick] (\j);
\end{tikzpicture}
&
\begin{tikzpicture}
  \graph[circular placement, radius=1.3cm,
    nodes={circle,draw,minimum size=4mm,inner sep=0pt},
    edges={black!45}] {subgraph K_n [n=7, clockwise]};

  \foreach \i/\j in {1/2,1/3,1/4,2/3,2/4,3/4,5/6,5/7,6/7}
    \draw (\i) edge[ultra thick] (\j);
\end{tikzpicture}
&
\begin{tikzpicture}
  \graph[circular placement, radius=1.3cm,
    nodes={circle, draw, minimum size=4mm, inner sep=0pt},
    edges={black!45}]
    {subgraph K_n [n=9, clockwise]};
  \draw (1) edge[ultra thick] (2); \draw (2) edge[ultra thick] (6); \draw (6) edge[ultra thick] (1);
  \draw (3) edge[ultra thick] (4); \draw (4) edge[ultra thick] (8); \draw (8) edge[ultra thick] (3);
  \draw (5) edge[ultra thick] (7); \draw (7) edge[ultra thick] (9); \draw (9) edge[ultra thick] (5);
\end{tikzpicture}
&
\begin{tikzpicture}
  \graph[circular placement, radius=1.3cm,
    nodes={circle,draw,minimum size=4mm,inner sep=0pt},
    edges={black!45}] {subgraph K_n [n=8, clockwise]};
  \foreach \a/\b in {1/4,2/3,5/8,6/7}
    \draw (\a) edge[ultra thick] (\b);
\end{tikzpicture}
\\[2mm]
\hspace{2mm}$ \overline{T}_{\{1,2,3\},\{4,5,6\},\{7,8\}}$ \hspace{2mm} & \hspace{2mm} $\overline{T}_{\{1,2,3,4\},\{5,6,7\}}$ \hspace{2mm} & \hspace{2mm} $\overline{T}_{\{1,2,6\},\{3,4,8\},\{5,7,9\}}$ \hspace{2mm} & \hspace{2mm} $\overline{T}_{\{1,4\},\{2,3\},\{5,8\},\{6,7\}}$
\end{tabular}
\caption{Some Turán graphs and their complements.}
\label{fig:turan_K8_p4}
\end{figure}
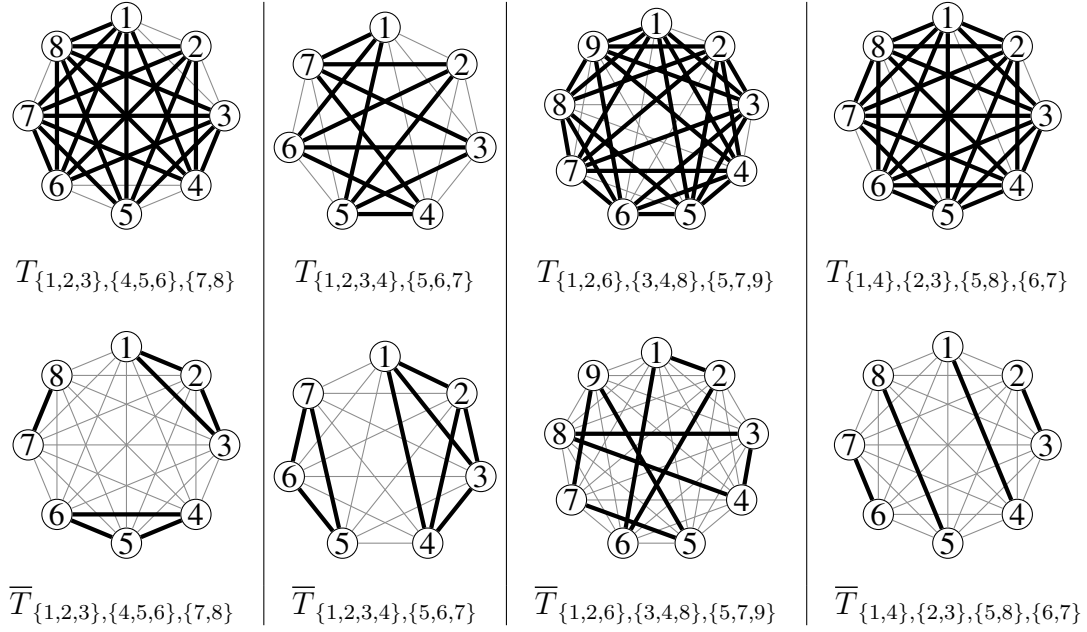

\begin{thm}[Tur{\'a}n's Theorem,\cite{Turan41,Turan54}]\label{thm:Turans_Theorem}
Let $m$ and $p$ be two positive integers such that $p<m$. For a $K_{p+1}$-free graph $G$ on $m$ vertices, the number of edges of $G$ is bounded:
\begin{equation*}
 \vert \E_G \vert \leq \mathrm{ex}(m,K_{p+1})
\end{equation*}
the \emph{Tur{\'a}n number} $\mathrm{ex}(m,K_{p+1})$ is given by
\begin{equation*}
\mathrm{ex}(m,K_{p+1})=\frac{(p-1)(m^2-\langle m\rangle_p ^2)}{2p}+\binom{\langle m\rangle_p}{2}=\frac{(p-1)m ^2 -\langle m\rangle_p
\left(p-\langle m\rangle_p\right)}{2p}
\end{equation*}
The equality holds if and only if $G$ is the  Tur{\'a}n graph $T(m,p)$.
\end{thm}

In \cite{Raymond18}, Tur\'an polytopes are defined for the hypergraph setting. Specializing the general definitions in \cite{Raymond18} to the usual graph setting, we get the following versions.

\begin{defn}
Let $n$ and $p$ be two positive integers with $p\leq n$. The \emph{Tur\'an polytope} $P_{\text{Tur\'an},n}^{(p)}$ is the convex hull of incidence vectors of $K_{p+1}$-free graphs in the complete graph $K_n$ and the \emph{clique relaxation Tur\'an polytope} $P_{\mathrm{CT},n}^{(p)}$ is given by
\begin{equation*}
P_{\mathrm{CT},n}^{(p)}=\left\{ (x_e)\in [0,1]^{\E_{n}} \big \vert \,\mathbf{x} (\E_I) \leq \mathrm{ex}( \vert I  \vert ,K_{p+1})\,\text{ for all } I \in \II_{n+1} \right\}.
\end{equation*}
\end{defn}

It is clear from the Tur\'an's theorem that we have $P_{\text{Tur\'an},n}^{(p)}\subseteq P_{\mathrm{CT},n}^{(p)}$.

\subsubsection{Perfect matching and fractional perfect matching polytopes}

\begin{defn}
Let $n\geq 4$ be an even integer.  The \emph{perfect matching polytope} $P_{\mathrm{PM},n}$ of the complete graph $K_n$ is the convex hull of the incidence vectors of the perfect matchings in $K_n$.
\end{defn}

\begin{rem}
The perfect matching polytope  $P_{\mathrm{PM},n}$ is empty when $n$ is odd.
\end{rem}
The following is a set description of $P_{\mathrm{PM},n}$ given by clique inequalities
\begin{equation*}\label{eqn:perfect_matching_clique_description}
P_{\mathrm{PM},n}=\left\{ (x_e)\in \R_{\geq 0}^{\E_{n}} \bigg \vert \, \x (\E_n) = \frac{n}{2}, \x (\E_I)\leq \left \lfloor \frac{ \vert I  \vert }{2}\right\rfloor \text{ for all }I\in \II_{n+1}\right\}
\end{equation*}
which is stated as Lemma \ref{lem:Appendix_perfect_matching_description} and proved in the appendix.

\begin{defn} The fractional perfect matching polytope $P_{\mathrm{FPM},n}$ is defined as
\begin{equation*}
 P_{\mathrm{FPM},n}=\left\{ (a_e)\in \R_{\geq 0}^{\E_{n}} \bigg \vert \,\mathbf{a} (\delta (i)) =1\,\text{for all }i\in [n] \right\}.
\end{equation*}
\end{defn}

An equivalent clique inequality description of the fractional perfect matching is
\begin{equation*}
P_{\mathrm{FPM},n}=\left\{ (a_e)\in [0,1]^{\E_n} \bigg \vert \, \mathbf{a}(\E_n) =\frac{n}{2}, \mathbf{a}(\E_I)\geq \max \left\{0, \vert I \vert -\frac{n}{2}\right\} \text{ for all }I\in \II_{n+1}\right\}.
\end{equation*}
For the equivalence of these two set descriptions, see Lemma \ref{lem:Appendix_fractional_perf_match_description} and its proof. The fractional perfect matching polytope is a face of the fractional matching polytope, see \cite[Section 30]{Schrijver03}  for a definition.

The perfect matching polytope is contained in the fractional perfect matching polytope and their integral vertices are exactly perfect matchings of $K_n$:
\begin{equation*}
P_{\mathrm{PM},n}\subseteq P_{\mathrm{FPM},n} \quad \text{and}\quad P_{\mathrm{PM},n}\cap \{0,1\}^{\E_n}=P_{\mathrm{FPM},n}\cap \{0,1\}^{\E_n}.
\end{equation*}
By definition each vertex of the perfect matching polytope is integral. However, this is not the case for the fractional perfect matching polytope. Yet, each vertex of the fractional perfect matching polytope is half-integral \cite{Balinski65}.

\begin{ex}
The weighted graph in Figure \ref{fig:half_integral_vertex_of_fractional_perfect_mathching_polytope} corresponds to a half-integral vertex of $P_{\mathrm{FPM},6}$.
\begin{figure}[h]
\centering
\begin{tikzpicture}
  \graph[
    circular placement,
    radius=1.3cm,
    nodes={
      circle,
      draw,
      minimum size=4mm,
      inner sep=0pt
    },
    edges={black!45}
  ] {subgraph K_n [n=6, clockwise]};
  
   \draw (1) edge[ultra thick,dashed] (2);
   \draw (2) edge[ultra thick,dashed] (3);
   \draw (3) edge[ultra thick,dashed] (1);
   \draw (4) edge[ultra thick,dashed] (5);
   \draw (5) edge[ultra thick,dashed] (6);
   \draw (6) edge[ultra thick,dashed] (4);
\end{tikzpicture}
\caption{The weights $x_e=\frac{1}{2}$ for $e \in \{\{1,2\},\{2,3\},\{3,1\},\{4,5\},\{5,6\},\{6,4\}\}$ and $x_e=0$ for all other edges. }
\label{fig:half_integral_vertex_of_fractional_perfect_mathching_polytope}
\end{figure}
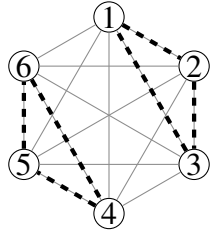
\end{ex}

For further discussions of the perfect matching polytope and the fractional perfect matching polytope, we refer the reader to \cite[Chapters 25-30]{Schrijver03}, \cite[Chapter 11]{Korte08} and references therein.

\subsection{The level-one $\mathfrak{sl}_p$ conformal block divisors and Tur\'an graphs}

Now, we are ready to investigate the polytope of effective boundary expressions of the conformal block divisor $D^{\mathfrak{sl}_p}_{1,(\Vec {d},0)}$ and its specializations.

\subsubsection{The general case of $P(D^{\mathfrak{sl}_p}_{1,(\Vec {d},0)})$}

The general case is given by the following proposition.

\begin{prop}
The divisor $D^{\mathfrak{sl}_p}_{1,(\Vec {d},0)}$ on $\overline{M}_{0,n+1}$ has the following Kapranov basis expression
\begin{equation*}
D^{\mathfrak{sl}_p}_{1,(\Vec {d},0)}=\alpha_{n+1}\left(\Vec{d}\right) \psi_{n+1} + \sum_{I\in \II_{n+1}, \vert I  \vert \geq 3} \beta_I\left(\Vec{d}\right)  [D_I]
\end{equation*}
and the polytopes of effective boundary expressions are given by
\begin{equation*}
P(D^{\mathfrak{sl}_p}_{1,(\Vec {d},0)})= \left\{ (a_e)_{e \in \E_{n}}\in \R^{\E_{n}} \, \bigg \vert \, \mathbf{a}(\E_n) = \alpha_{n+1}\left(\Vec{d}\right) , \beta_I\left(\Vec{d}\right) + \mathbf{a}(\E_I)\geq 0 \text{ for all }I \in \II_{n+1} \right\}.
\end{equation*}
where
\begin{equation*}
\alpha_{n+1}\left(\Vec{d}\right)=\frac{n-2}{2p}
\sum_{i=1}^n \theta(\{i\})
-
\frac{1}{2p}
\sum_{e \in \E_n}\theta(e)
\end{equation*}
and
\begin{equation*}
\beta_I(\Vec{d})= -\frac{1}{2p}
\left(
\theta(I)+( \vert I  \vert -2)\sum_{i\in I}\theta(\{i\})
-\sum_{e \in \E_I}\theta(e) \right).
\end{equation*}
\end{prop}
\begin{proof}
Substituting the expressions for $[D_e]$ Lemma \ref{lem:Djk_basis_decomp} and $\psi_i$ in Lemma \ref{lem:slp_conformal_divisor_in_kapranov} to the expression for $D^{\mathfrak{sl}_p}_{1,(\Vec {d},0)}$ in Lemma \ref{lem:slp_conformal_divisor_in_kapranov}, we get
\begin{equation*}
\begin{aligned}
D^{\mathfrak{sl}_p}_{1,(\Vec {d},0)} =&\frac {1}{2p}\left (
\sum_{i=1}^n \theta (\{i\}) \left((n-2)\psi_{n+1} - \sum_{\substack{I \in \II_{n+1}, \vert I \vert \geq 3\\ i\in I}} ( \vert I  \vert -2)D_{I}\right)\right)\\
&\hspace{5em}-\frac {1}{2p}\left (\sum_{e \in \E_n}\theta(e)\left(\psi_{n+1} - \sum_{\substack{I\in \II_{n+1}, \vert I \vert \geq 3\\ e\subseteq I}}[D_I]\right)
+
\sum_{I\in \II_{n+1}, \vert I \vert \geq 3}\theta (I)[D_I]\right).
\end{aligned}
\end{equation*}
Keeping track of the coefficients of $\psi_{n+1}$ and $[D_I]$ for each $I\in \II_{n+1}$ with $ \vert I \vert \geq 3$, we get
\begin{equation*}
D^{\mathfrak{sl}_p}_{1,(\Vec {d},0)}=\alpha_{n+1}\left(\Vec{d}\right) \psi_{n+1} + \sum_{I\in \II_{n+1}, \vert I  \vert \geq 3} \beta_I\left(\Vec{d}\right)  [D_I]
\end{equation*}
where $\alpha_{n+1}\left(\Vec{d}\right)$ and $\beta_I\left(\Vec{d}\right)$ are exactly the expressions given above and the description of  $P(D^{\mathfrak{sl}_p}_{1,(\Vec {d},0)})$ follows from Corollary \ref{cor:P(D) inequalities}. Hence,
we complete the proof.
\end{proof}

Following the equation \eqref{eqn:rank_functions_interms_of_alpha_beta}, we set
\begin{equation*}
r_I\left(\Vec{d}\right)\coloneqq   r_I\left(D^{\mathfrak{sl}_p}_{1,(\Vec {d},0)}\right)=\begin{cases}
        { \vert I \vert \choose 2} + \beta_I(\Vec{d}) & I\in \II_{n+1},\\
        {n \choose 2}- \alpha_{n+1}(\Vec{d}) & I = [n].
\end{cases} 
\end{equation*}

\begin{cor}\label{cor:P_tilde_for_slp_all_1s}
We have
\begin{equation*}
\widetilde{P} \left(D_{1,\Vec{d}}^{\mathfrak{sl}_p}\right)=\left\{ (x_e)\in \R^{\E_{n}} \bigg \vert \, \x (\E_n) = r_{[n]}(\Vec{d}), \x (\E_I)\leq r_{I}(\Vec{d}) \text{ for all }I\in \II_{n+1}\right\}
\end{equation*}
where
\begin{equation*}
r_{I}(\Vec{d})\coloneqq   \binom{ \vert I \vert }{2}-\frac{1}{2p}
\left(
\theta(I)+( \vert I  \vert -2)\sum_{i\in I}\theta(\{i\})
-\sum_{e \in \E_I}\theta(e) \right)\quad \text{for all } I\in \II_{n+1}\cup \{[n]\}.
\end{equation*}
\end{cor}

\begin{proof}
The proof follows from Proposition \ref{prop:P tilde inequalities} and the fact that $\theta ([n])=0$ as $p$ divides $\sum_{i=1}^n d_i$ by assumption.
\end{proof}

\subsubsection{The polytope $\widetilde{P} \left(D_{1,(1^n,0)}^{\mathfrak{sl}_p}\right)$ and its connection to the Tur\'an polytope}
We further investigate the polytopes obtained above when we specialize the weights to $d_i=1$ for all $i\in [n]$. Hence, we see that  $p\, \vert \, n$ and we have
\begin{equation}\label{eqn:theta_specialized_all_1s}
\theta (I)=\langle \vert I \vert \rangle_p
\left(p-\langle \vert I  \vert \rangle_p\right).
\end{equation}

\begin{prop}\label{prop:P_tilde_sl_p_weights_all_1s} Let $n\in p \Z_{\geq 1}$. Then, we have
\begin{equation*}
\widetilde{P} \left(D_{1,(1^n,0)}^{\mathfrak{sl}_p}\right)=\left\{ (x_e)\in \R^{\E_{n}} \big \vert \, \x (\E_n) = \mathrm{ex} (n, K_{p+1}), \x (\E_I)\leq \mathrm{ex} ( \vert I \vert , K_{p+1}) \text{ for all }I\in \II_{n+1}\right\}.  
\end{equation*}
\end{prop}

\begin{proof}
By Corollary \ref{cor:P_tilde_for_slp_all_1s} and equation \eqref{eqn:theta_specialized_all_1s}, we see that
\begin{equation*}
\begin{aligned}
r_I(1^n)
&=\binom{ \vert I \vert }{2}-\frac{1}{2p}
\left(
\langle \vert I \vert \rangle_p
\left(p-\langle \vert I  \vert \rangle_p\right)+( \vert I  \vert -2)\sum_{i\in I}(p-1)
-\sum_{e \in \E_I}2(p-2) \right)\\
&=\binom{ \vert I \vert }{2}-\frac{1}{2p}\left(
\langle \vert I \vert \rangle_p
\left(p-\langle \vert I  \vert \rangle_p\right)+( \vert I  \vert -2) \vert I \vert (p-1)
-\binom{ \vert I  \vert }{2}2(p-2) \right)\\
&=2\frac{(p-1)}{p}\binom{ \vert I \vert }{2}-\frac{1}{2p}\left(
\langle \vert I \vert \rangle_p
\left(p-\langle \vert I  \vert \rangle_p\right)+( \vert I  \vert -2) \vert I \vert (p-1)\right)\\
&=\frac{(p-1)}{2p}(2 \vert I \vert ( \vert I \vert -1)- \vert I \vert ( \vert I \vert -2))-\frac{\langle \vert I \vert \rangle_p
\left(p-\langle \vert I  \vert \rangle_p\right)}{2p}\\
&=\frac{(p-1) \vert I \vert ^2 -\langle \vert I \vert \rangle_p
\left(p-\langle \vert I  \vert \rangle_p\right)}{2p}
\end{aligned}
\end{equation*}
which is exactly the Tur\'an number $\mathrm{ex} ( \vert I \vert , K_{p+1})$.
\end{proof}
Next, we focus on the $0/1$ points of this polytope.
\begin{cor}\label{cor:01points_of_slp_conformal}
Let $p\geq 2$ be an integer such that $p \mid n$. The $0/1$ points of $\widetilde{P} \left(D_{1,(1^n,0)}^{\mathfrak{sl}_p}\right)$ correspond to balanced Tur\'an graphs $T_{A_1,\ldots,A_p}$. More precisely, they are precisely the incidence vectors of complete $p$-partite graphs $$K_{A_1,\ldots,A_p}\simeq K_{\frac{n}{p},\ldots, \frac{n}{p}}$$ where $[n]=A_1 \sqcup \cdots \sqcup A_p$.
\end{cor}

\begin{proof}
Let $G$ be a subgraph of $K_n$ corresponding to a $0/1$ point $\mathbf{x}$ of $\widetilde{P} \left(D_{1,(1^n,0)}^{\mathfrak{sl}_p}\right)$. For $I \subseteq [n]$ with $ \vert I \vert = p+1$, we have
\begin{equation*}
\mathbf{x}(\E_I)\leq \mathrm{ex}( \vert I  \vert , K_{p+1})=\mathrm{ex}(p+1, K_{p+1})=\frac{(p-1)((p+1)^2-1)}{2p}=\frac{p^2+p-2}{2}=\binom{p+1}{2}-1.
\end{equation*}
As a result, we see that restriction of the subgraph $G$ to any $K_I\subseteq K_n$ with $ \vert I \vert =p+1$ has fewer edges than the number of edges of $K_{p+1}$. Hence, $G$ is $K_{p+1}$ free. Then, by Tur\'an's Theorem \ref{thm:Turans_Theorem} and the fact that $\mathbf{x}(\E_n)=\mathrm{ex}(n, K_{p+1})$, we see that $G$ must be a Tur\'an graph of the form given in the statement noting that $p \mid n$.

Conversely, assume that we have a graph
\begin{equation*}
G=K_{A_1,\ldots,A_p}\simeq K_{\frac{n}{p},\ldots, \frac{n}{p}}
\end{equation*}
where $[n]=A_1 \sqcup \cdots \sqcup A_p$. Since each restriction of $G$ to $K_I\subseteq K_n$ is still $K_{p+1}$-free, we have $\mathbf{x}(\E_I)=\mathrm{ex}( \vert I  \vert , K_{p+1})$ and since $G$ is a Tur\'an graph, we have $\mathbf{x}(\E_n)=\mathrm{ex}(n, K_{p+1})$ by the Tur\'an's Theorem \ref{thm:Turans_Theorem}.
\end{proof}

\begin{cor}\label{cor:slp_expressions_from_Turan_graphs}
Let $p\geq2$ with $p\mid n$, and let $T_{A_1,\dots,A_p}$ be a balanced
Tur\'an graph on $[n]$.  For $I\in\II_{n+1}$ set 
\begin{equation*}
\chi_i(I)\coloneqq    \vert A_i\cap I  \vert,
\end{equation*}
and let $\tau_1(I),\dots,\tau_p(I)$ be the part sizes of the Tur\'an graph $T( \vert I  \vert ,p)$, that
is, first $\langle \vert I  \vert \rangle_p$ of them equal to $\lceil \vert I  \vert /p\rceil$ and the remaining
$p-\langle \vert I  \vert \rangle_p$ equal to $\lfloor \vert I  \vert /p\rfloor$.  Then, the effective boundary
expression corresponding to $T_{A_1,\dots,A_p}$ is
\begin{equation*}
D_{1,(1^n,0)}^{\mathfrak{sl}_p}=\sum_{I \in \II_{n+1}}c_I\left(T_{A_1,\ldots,A_p}\right)\left[D_I \right],
\qquad
c_I\left(T_{A_1,\ldots,A_p}\right)=\sum_{i=1}^p\left(\binom{\chi_i(I)}{2}-\binom{\tau_i(I)}{2}\right).
\end{equation*}
\end{cor}

\begin{proof}
By Proposition \ref{prop:P tilde inequalities}, Proposition \ref{prop:P_tilde_sl_p_weights_all_1s} and Tur\'an's Theorem \ref{thm:Turans_Theorem}, we have
\begin{equation*}
c_I(T_{A_1,\ldots,A_p})=r_I(1^n,0)-\x (\E_I)=\mathrm{ex}( \vert I  \vert , K_{p+1})- \vert \E_I(T_{A_1,\ldots,A_p})  \vert = \vert \E_I(T( \vert I  \vert ,p))  \vert - \vert \E_I(T_{A_1,\ldots,A_p})  \vert .
\end{equation*}
Re-interpreting the right-hand side as the edge numbers of Tur\'an graph complements, we get
\begin{equation*}
\begin{aligned}
c_I(T_{A_1,\ldots,A_p})
&=\left(\binom{ \vert I  \vert }{2}- \vert \E(\overline{T}( \vert I  \vert ,p))  \vert \right) -  \left(\binom{ \vert I  \vert }{2}- \vert \E(\overline{T}_{A_1,\ldots,A_p}) \vert \right)\\
&= \vert \E_I(\overline{T}_{A_1,\ldots,A_p}) \vert - \vert \E_I(\overline{T}( \vert I  \vert ,p))   \vert .
\end{aligned}
\end{equation*}
It is clear that 
\begin{equation*}
 \vert \E_I(\overline{T}_{A_1,\ldots,A_p})  \vert =\sum_{i=1}^p \binom{\chi_i (I)}{2}\quad \text{and}\quad \vert \E_I(\overline{T}( \vert I  \vert ,p))   \vert =\sum_{i=1}^p \binom{\tau_i (I)}{2}.
\end{equation*}
Hence, we complete the proof.
\end{proof}

We now consider $D_{1,(1^n,0)}^{\mathfrak{sl}_p}+t\psi_{n+1}$.

\begin{cor}\label{cor:union is CT polytope}
Let $p\geq2$ with $p \mid n$. Then, we have
\begin{equation*}
\bigcup_{0\leq t \leq \mathrm{ex} (n, K_{p+1})}
\widetilde{P} \left(D_{1,(1^n,0)}^{\mathfrak{sl}_p}+t\psi_{n+1}\right)\cap\R^{\E_n}_{\geq 0}
=P^{(p)}_{\mathrm{CT},n},
\end{equation*}
the clique relaxation Tur\'an polytope.
\end{cor}

\begin{proof}
For $t\in \R$, we have
\begin{equation*}
\widetilde{P} \left(D_{1,(1^n,0)}^{\mathfrak{sl}_p}+t\psi_{n+1}\right)=\left\{ (x_e)\in \R^{\E_{n}} \big \vert \, \x (\E_n) = \mathrm{ex} (n, K_{p+1})-t, \x (\E_I)\leq \mathrm{ex} ( \vert I \vert , K_{p+1}) \text{ for all }I\in \II_{n+1}\right\}.
\end{equation*}
Taking union of these as $0\leq t \leq \mathrm{ex} (n, K_{p+1})$, we get
\begin{equation*}
\bigcup_{0\leq t \leq \mathrm{ex} (n, K_{p+1})}\widetilde{P} \left(D_{1,(1^n,0)}^{\mathfrak{sl}_p}+t\psi_{n+1}\right)\cap\R^{\E_n}_{\geq 0} = \left\{ (x_e)\in \R^{\E_{n}}_{\geq 0} \big \vert \,  \x (\E_I)\leq \mathrm{ex} ( \vert I \vert , K_{p+1}) \text{ for all }I\in \II_{n+1}\right\}.
\end{equation*}
As $\mathrm{ex} (2, K_{p+1})=1$, we have $x_e\leq 1$ for all $e \in \E_n$. Hence, the union
\begin{equation*}
\bigcup_{0\leq t \leq \mathrm{ex} (n, K_{p+1})}\widetilde{P} \left(D_{1,(1^n,0)}^{\mathfrak{sl}_p}+t\psi_{n+1}\right)\cap\R^{\E_n}_{\geq 0}
\end{equation*}
is the Raymond's clique relaxation Tur\'an polytope $P^{(p)}_{\mathrm{CT},n}$. 
\end{proof}

\subsection{Connections with the perfect matching and fractional perfect matching polytopes}

\subsubsection{The polytope $\widetilde{P} \left(2D_{1,(1^n,0)}^{\mathfrak{sl}_2}\right)$}

\begin{thm} \label{thm:perfect_matching_and_sl2} Let $n\geq 4$ be an even integer. Then, we have
\begin{equation*}
P_{\mathrm{PM},n}=\widetilde{P} \left(2D_{1,(1^n,0)}^{\mathfrak{sl}_2}\right) \cap [0,1]^{\E_n}.
\end{equation*}
\end{thm}

\begin{proof}
Multiplying the divisor $D_{1,(1^n,0)}^{\mathfrak{sl}_2}$ with $2$ amounts to making the change of variables $x_e \mapsto 2x_e-1$. Then, by Proposition \ref{prop:P_tilde_sl_p_weights_all_1s}, we obtain
\begin{equation*}
\widetilde{P} \left(2D_{1,(1^n,0)}^{\mathfrak{sl}_2}\right)=\left\{ (x_e)\in \R^{\E_{n}} \big \vert \, \x (\E_n) = \overline{\mathrm{ex}}(n, K_{3}), \x (\E_I)\leq \overline{\mathrm{ex}}( \vert I \vert ,K_{3}) \text{ for all }I\in \II_{n+1}\right\}  
\end{equation*}
where
\begin{equation*}
\begin{aligned}
\overline{\mathrm{ex}}( \vert I \vert , K_3)
=2 \mathrm{ex}( \vert I \vert , K_3)-\binom{ \vert I \vert }{2}
&=\frac{ \vert I \vert ^2 -\langle \vert I \vert \rangle_2
\left(2-\langle \vert I  \vert \rangle_2\right)}{2}-\frac{ \vert I \vert ^2- \vert I \vert }{2}\\
&=\frac{ \vert I  \vert }{2} -\frac{\langle \vert I \vert \rangle_2
\left(2-\langle \vert I  \vert \rangle_2\right)}{2}\\
&=\left\lfloor\frac{ \vert I  \vert }{2}\right\rfloor.
\end{aligned}
\end{equation*}
Hence, we conclude the proof by the description of perfect matching polytope given in Lemma \ref{lem:Appendix_perfect_matching_description} in the appendix.
\end{proof}

\subsubsection{The polytope $P \left(D_{1,(1^n,0)}^{\mathfrak{sl}_{\frac{n}{2}}}\right)$}

\begin{thm}\label{thm:fractional_perfect_and_sl2}
Let $n\geq 4$ be an even integer. Then, we have
\begin{equation*}
P_{\mathrm{FPM},n}=P \left(D_{1,(1^n,0)}^{\mathfrak{sl}_{\frac{n}{2}}}\right).
\end{equation*}
\end{thm}

\begin{proof} We set $p\coloneqq   \frac{n}{2}$ or equivalently $n\coloneqq   2p$. For, $I \in \II_{n+1}\cup \{[n]\}$, consider the expression
\begin{equation*}
\overline{\beta}_{I}\coloneqq   \binom{ \vert I \vert }{2}-\mathrm{ex}( \vert I  \vert ,K_{p+1})
=\binom{ \vert I \vert }{2}-\binom{\langle \vert I \vert \rangle_p}{2}-\frac{(p-1)( \vert I \vert ^2-\langle \vert I  \vert \rangle_p ^2)}{2p}
\end{equation*}
which is equal to $0$ if $0\leq \vert I \vert \leq p$. Now assume $p+1 \leq \vert I \vert < n=2p$. In this case, we have $ \vert I \vert =p+ \langle \vert I \vert \rangle_p$ and
\begin{equation*}
\overline{\beta}_I= \vert I \vert -p.
\end{equation*}
For $ \vert I \vert =2p$, we have $\overline{\beta}_I=p$.
Hence, we conclude that
\begin{equation*}
\overline{\beta}_I=
\begin{cases}
0 &\text{if } 0\leq \vert I \vert \leq p,\\
 \vert I \vert -p &\text{if } p+1\leq \vert I \vert \leq 2p
\end{cases}
=\max \{0, \vert I \vert -p\}
\end{equation*}
and consequently,
\begin{equation*}
P \left(D_{1,(1^n,0)}^{\mathfrak{sl}_{\frac{n}{2}}}\right)=\left\{ (a_e)\in \R_{\geq 0}^{\E_{n}} \big \vert \, \mathbf{a}(\E_n) = p, \mathbf{a}(\E_I)\geq \max \{0, \vert I \vert -p\} \text{ for all }I\in \II_{n+1}\right\}.
\end{equation*}
Substituting $p=\frac{n}{2}$, the expression agrees with the description of fractional perfect matching polytope $P_{\mathrm{FPM},n}$ given in Lemma \ref{lem:Appendix_fractional_perf_match_description} in the appendix.
\end{proof}

\appendix

\section{Alternative descriptions of linear programming polytopes}

Here we collect the set-theoretic descriptions of several linear programming polytopes used throughout the paper. These descriptions are not always the standard ones in the literature. In some cases, the corresponding polytopes are defined for more general graphs or with slightly different conventions. For the reader's convenience, we prove that the formulations used here agree with the versions appearing in standard references. Thus, we make no claim of originality for these set-theoretic descriptions.

We make repeated use of the following lemma.

\begin{lem}[Handshake Lemma]\label{lem:handshake}
Let $\x\in\R^{\E_n}$ and let $I\subseteq[n]$.  Then,
\begin{equation}\label{eqn:general_handshake}
  \sum_{i\in I}\x(\delta(i))=2\,\x(\E_I)+\x(\delta(I)).
\end{equation}
In particular, taking $I=[n]$, for which $\delta([n])=\varnothing$,
\begin{equation}\label{eqn:handshake global}
  \sum_{i\in[n]}\x(\delta(i))=2\,\x(\E_n).
\end{equation}
\end{lem}

\begin{proof}
Exchanging the order of summation,
$$\sum_{i\in I}\x(\delta(i))=\sum_{e \in\E_n}  \vert e\cap I  \vert \,x_e.$$  The coefficient
$  \vert e\cap I  \vert $ equals $2$ if $e \in\E_I$, equals $1$ if $e \in\delta(I)$, and equals $0$
otherwise.
\end{proof}

\subsection{Held--Karp relaxation polytope}

Recall that the Held--Karp relaxation polytope is defined by
\begin{equation*}
P_{\mathrm{SEP},n}=\left \{(x_{e})_{e \in \E_n} \in \R^{ \E_n}_{\geq 0} \, \bigg \vert \, x_e\leq 1\,\, \forall e \in \E_n,\, \deg_{\x }(i)=2\,\, \forall i\in [n],\,  \x(\delta(I))\geq 2\,\, \forall I \, \text{with } \varnothing \neq I \subsetneq [n] \right \}.
\end{equation*}

\begin{lem}\label{lem:Held-Karp rank description}
For $n\geq 3$, the Held--Karp relaxation polytope is given by the following description
\begin{equation*}
P_{\mathrm{SEP},n}= \left \{ (x_{e})_{e \in \E_n} \in \R^{ \E_n}_{\geq 0}\, \big \vert \, \x(\E_n) =n,\, \x(\E_I) \leq \vert I \vert -1 \,  \text{ for all }I\in \II_{n+1} \right \}.
\end{equation*}
\end{lem}

\begin{proof}
Let $(x_{e})_{e \in \E_n} \in P_{\mathrm{SEP},n}$. By the Handshake Lemma \ref{lem:handshake}, we have
\begin{equation*}
\x (\E_n)=\frac{1}{2}\sum_{i \in [n]}\x (\delta(i))=\frac{2n}{2}=n.
\end{equation*}
For all $I \in \II_{n+1}$, again by the Handshake Lemma \ref{lem:handshake}, we have
\begin{equation*}
\x (\E_I)=\frac{1}{2} \left(\sum_{i\in I}\x (\delta (i))-\x (\delta (I))\right) \leq \frac{1}{2}\left(2 \vert I \vert -2\right)= \vert I \vert -1.
\end{equation*}
Hence, we conclude the forward inclusion.

Now, conversely assume $(x_{e})_{e \in \E_n} \in \R^{ \E_n}_{\geq 0}$ with $\x(\E_n) =n$ and  $\x(\E_I) \leq \vert I \vert -1 $ for all $I\in \II_{n+1}$. In particular, we have $x_e\leq 1$ by taking $e = \{i,j\} = I$. For any $i\in [n]$, we have
\begin{equation*}
\x \left(\E_{[n] \setminus \{ i \}}\right)\leq \vert [n] \setminus \{ i \} \vert -1 =n-2.
\end{equation*}
Thus, we have
\begin{equation*}
\x (\delta (i))=n-\x \left(\E_{[n] \setminus \{ i \}}\right)\geq 2.
\end{equation*}
On the other hand, we again have, by the Handshake Lemma \ref{lem:handshake}
\begin{equation*}
\sum_{i \in [n]}\x (\delta(i))=2\x (\E_n)=2n.
\end{equation*}
As a result of the last two equations, we see that $\x (\delta (i))=2$ for each $i \in [n]$. Now, for any $\varnothing \neq I \subsetneq [n]$, we have
\begin{equation*}
\x (\delta (I))=\x(\E_n)-\x (\E_{I})-\x (\E_{[n]\setminus I})\geq n - ( \vert I \vert -1) - ( \vert [n]\setminus I \vert -1)=2.
\end{equation*}
This completes the reverse inclusion.
\end{proof}

\subsection{The perfect matching polytope of $K_n$}

\begin{lem}[\cite{Ed65}]\label{lem:Appendix_Edmonds_perfect_matching}
Let $n\geq4$ be an even integer.  Then,
\begin{equation*}
P_{\mathrm{PM},n}=\left\{ \x\in \R_{\geq 0}^{\E_{n}} \ \middle  \vert \
\x(\delta(i))=1 \ \forall i\in[n], \ \
\x(\delta(I))\geq 1 \ \forall I\in\II_{n+1} \text{ with }   \vert I \vert \text{ odd}\right\},
\end{equation*}
and this set coincides with
\begin{equation*}
\left\{ \x\in \R_{\geq 0}^{\E_{n}} \ \middle \vert \
\x(\delta(i))=1 \ \forall i\in[n], \ \
\x(\E_I)\leq \left\lfloor \tfrac{ \vert I \vert }{2}\right\rfloor
\ \forall I\in\II_{n+1} \text{ with }  \vert I \vert \text{ odd}\right\}.
\end{equation*}
\end{lem}

\begin{proof}
The first description is given in \cite[Section 25.4]{Schrijver03}.  For the second, assume
$\x(\delta(i))=1$ for all $i$ and let $I$ have odd cardinality.  By Lemma
\ref{lem:handshake}, $$\x(\delta(I))= \vert I \vert -2\,\x(\E_I),$$ so
$\x(\delta(I))\geq1$ holds if and only if
$$\x(\E_I)\leq\frac{ \vert I \vert -1}{2}=\lfloor \vert I \vert /2\rfloor .$$
\end{proof}
The following description of $P_{\mathrm{PM},n}$ is possibly well-known. However, we could not find a proper reference to cite. For this reason, we state and prove it for the convenience of the reader.
\begin{lem}\label{lem:Appendix_perfect_matching_description}
Let $n\geq 4$ be an even integer.  The perfect matching polytope $P_{\mathrm{PM},n}$ can also be defined as
\begin{equation*}
P_{\mathrm{PM},n}=\left\{ (x_e)\in \R_{\geq 0}^{\E_{n}} \bigg \vert \, \x (\E_n) = \frac{n}{2}, \x (\E_I)\leq \left \lfloor \frac{ \vert I  \vert }{2}\right\rfloor \text{ for all }I\in \II_{n+1}\right\}.
\end{equation*}
\end{lem}
\begin{proof}
Firstly, assume $\x (\delta (i)) =1$ for all $i\in [n]$ and $\x (\E_I) \leq \left \lfloor \frac{ \vert I  \vert }{2}\right\rfloor$ for all $I\in \II_{n+1}$ with $ \vert I  \vert $ odd. Then, by the Handshake Lemma \ref{lem:handshake}, we have
\begin{equation}\label{eqn:Perfect_matching_handshake}
\x (\E_n)=\frac{1}{2}\sum_{i\in [n]}\x (\delta (i))=\frac{n}{2}.
\end{equation}
To complete the forward inclusion, it remains to prove the clique inequalities for $I\in\II_{n+1}$ with $ \vert I  \vert $ even. In this case, we have by equation \eqref{eqn:general_handshake} and nonnegativity
\begin{equation*}
\x (\E_I)=\frac{1}{2}\left(\sum_{i\in I}\x (\delta (i))-\x (\delta (I))\right)\leq \frac{1}{2}\sum_{i\in I}\x (\delta (i))=\frac{1}{2}\sum_{i\in I}1=\frac{ \vert I \vert }{2}=\left \lfloor \frac{ \vert I  \vert }{2}\right\rfloor .
\end{equation*}

For the reverse inclusion,  assume $\x (\E_n)=\frac{n}{2}$ and $\x (\E_I) \leq \left \lfloor \frac{ \vert I  \vert }{2}\right\rfloor$ for all $I\in \II_{n+1}$. The only condition to prove is $\x (\delta (i))=1$ for all $i\in [n]$. We have
\begin{equation*}
\x (\delta (i))=\x (\E_n)-\x (\E_{[n]\setminus \{i\}})\geq \frac{n}{2}-\left \lfloor \frac{ \vert [n]\setminus \{i\}  \vert }{2}\right\rfloor=\frac{n}{2}-\frac{n-2}{2}=1.
\end{equation*}
By the Handshake Lemma \ref{lem:handshake}, we have
\begin{equation*} \frac{n}{2} = \x (\E_n) = \frac{1}{2}\sum_{i\in [n]}\x (\delta (i)), \end{equation*}
so we may conclude that $\x (\delta (i))=1$, completing the proof. 
\end{proof}

\subsection{The fractional perfect matching polytope of $K_n$}
Recall that we have
\begin{equation*}
 P_{\mathrm{FPM},n}=\left\{ (a_e)\in \R_{\geq 0}^{\E_{n}} \big \vert \,\mathbf{a} (\delta (i)) =1\,\text{for all }i\in [n] \right\}.
\end{equation*}

\begin{lem}\label{lem:Appendix_fractional_perf_match_description}
Let $n\geq 4$ be an even integer.  Then,
\begin{equation*}
P_{\mathrm{FPM},n}=\left\{ \mathbf{a}\in \R^{\E_n}_{\geq0} \ \middle \vert \
\mathbf{a}(\E_n) =\frac{n}{2}, \ \
\mathbf{a}(\E_I)\geq \max \left\{0, \vert I \vert -\frac{n}{2}\right\}
\ \text{ for all }I\in \II_{n+1}\right\}.
\end{equation*}
In particular every $\mathbf{a}$ in this set satisfies $a_e\leq1$ automatically, so the
same description holds with $\R^{\E_n}_{\geq0}$ replaced by $[0,1]^{\E_n}$. 
\end{lem}

\begin{proof}
Let $(a_e)\in \R_{\geq 0}^{\E_{n}}$, $I\in \II_{n+1}$ and set $J\coloneqq   [n] \setminus I$.

For the forward inclusion, assume that we have $\mathbf{a}(\delta (i))=1$ for all $i\in [n]$. Then, by the Handshake Lemma \ref{lem:handshake}, we have
\begin{equation*}
\mathbf{a} (\E_n)=\frac{1}{2}\sum_{i\in [n]}\mathbf{a} (\delta (i))=\frac{n}{2}.
\end{equation*}
Since $a_e\geq 0$, we have
\begin{equation*}
\mathbf{a} (\delta (I))\leq \sum_{i\in I} \mathbf{a} (\delta (i)) \quad \text{and}\quad \mathbf{a} (\delta (J))\leq \sum_{j\in J} \mathbf{a} (\delta (j)).
\end{equation*}
As we have $\delta (I)=\delta (J)$, we obtain 
\begin{equation*}
\mathbf{a} (\delta (I))=\mathbf{a} (\delta (J))\leq \min \left \{ \sum_{i\in I} \mathbf{a} (\delta (i)),\sum_{j\in J} \mathbf{a} (\delta (j)) \right \}=\min \left \{ \vert I  \vert , \vert J \vert \right \}=\min \{ \vert I  \vert , n- \vert I  \vert \}.
\end{equation*}
Then, by Handshake Lemma \ref{lem:handshake}, we have
\begin{equation*}
2\mathbf{a} (\E_{I})=\sum_{i\in I}\mathbf{a} (\delta (i))-\mathbf{a} (\delta (I))\geq \vert I \vert - \min \{ \vert I  \vert , n- \vert I  \vert \}= \vert I \vert +\max \{- \vert I  \vert , \vert I  \vert -n \}=\max \{0, 2 \vert I  \vert -n \}.
\end{equation*}
This completes forward inclusion.

Conversely, assume $\mathbf{a}(\E_n) =\frac{n}{2}$ and $\mathbf{a}(\E_I)\geq \max \left\{0, \vert I \vert -\frac{n}{2}\right\}$. Then, for any $i\in [n]$, we have
\begin{equation*}
\mathbf{a}(\E_n)-\mathbf{a} (\delta (i))=\mathbf{a} (\E_{[n]\setminus \{i\}})\geq \max \left\{0,n-1 -\frac{n}{2}\right\}=\frac{n}{2}-1.
\end{equation*}
Hence, we get
\begin{equation*}
\mathbf{a}(\delta (i))\leq \mathbf{a}(\E_n)-\left(\frac{n}{2}-1\right)=\frac{n}{2}-\frac{n}{2}+1=1.
\end{equation*}
Note also that we have
\begin{equation*}
\sum_{i\in [n]}\mathbf{a}(\delta (i))=2\mathbf{a} (\E_n)=n
\end{equation*}
by the Handshake Lemma \ref{lem:handshake}. As a result, for each $i\in [n]$, we have $\mathbf{a}(\delta (i))=1$. This completes the reverse inclusion. Finally, $a_e\leq \mathbf{a}(\delta(i))=1$ for both vertices of the edge $e$, by
nonnegativity.
\end{proof}

\bibliographystyle{abbrv}
\bibliography{polytopebiblio}

@article{AlexeevGibneySwinarski14,
AUTHOR = {Alexeev, Valery and Gibney, Angela and Swinarski, David},
     TITLE = {Higher-level {$\mathfrak{sl}_2$} conformal blocks divisors on
              {$\overline M_{0,n}$}},
   JOURNAL = {Proc. Edinb. Math. Soc.},
  FJOURNAL = {Proceedings of the Edinburgh Mathematical Society.},
    VOLUME = {57},
      YEAR = {2014},
    NUMBER = {1},
     PAGES = {7--30},
      ISSN = {0013-0915,1464-3839},
   MRCLASS = {14D21 (14D23 14E30)},
  MRNUMBER = {3165010},
MRREVIEWER = {Zhenbo\ Qin},
       DOI = {10.1017/S0013091513000941},
       URL = {https://doi.org/10.1017/S0013091513000941},
}

@article{Arbarello87,
AUTHOR = {Arbarello, Enrico and Cornalba, Maurizio},
     TITLE = {The {P}icard groups of the moduli spaces of curves},
   JOURNAL = {Topology},
  FJOURNAL = {Topology. An International Journal of Mathematics},
    VOLUME = {26},
      YEAR = {1987},
    NUMBER = {2},
     PAGES = {153--171},
      ISSN = {0040-9383},
   MRCLASS = {14H10 (14C22)},
  MRNUMBER = {895568},
MRREVIEWER = {Joseph\ Harris},
       DOI = {10.1016/0040-9383(87)90056-5},
       URL = {https://doi.org/10.1016/0040-9383(87)90056-5},
}

@article{Arbarello96,
AUTHOR = {Arbarello, Enrico and Cornalba, Maurizio},
     TITLE = {Combinatorial and algebro-geometric cohomology classes on the
              moduli spaces of curves},
   JOURNAL = {J. Algebraic Geom.},
  FJOURNAL = {Journal of Algebraic Geometry},
    VOLUME = {5},
      YEAR = {1996},
    NUMBER = {4},
     PAGES = {705--749},
      ISSN = {1056-3911,1534-7486},
   MRCLASS = {14H10 (14C17)},
  MRNUMBER = {1486986},
MRREVIEWER = {Montserrat\ Teixidor i Bigas},
}

@article{Arbarello98,
  AUTHOR = {Arbarello, Enrico and Cornalba, Maurizio},
     TITLE = {Calculating cohomology groups of moduli spaces of curves via
              algebraic geometry},
   JOURNAL = {Inst. Hautes \'Etudes Sci. Publ. Math.},
  FJOURNAL = {Institut des Hautes \'Etudes Scientifiques. Publications
              Math\'ematiques},
    NUMBER = {88},
      YEAR = {1998},
     PAGES = {97--127},
      ISSN = {0073-8301,1618-1913},
   MRCLASS = {14H10 (14C30)},
  MRNUMBER = {1733327},
MRREVIEWER = {Elham\ Izadi},
       URL = {http://www.numdam.org/item?id=PMIHES_1998__88__97_0},
}

@book{Arbarello11,
 AUTHOR = {Arbarello, Enrico and Cornalba, Maurizio and Griffiths,
              Phillip A.},
     TITLE = {Geometry of Algebraic Curves: {V}olume {II} with a contribution by Joseph Daniel Harris},
    SERIES = {Grundlehren der mathematischen Wissenschaften},
    VOLUME = {268},
 PUBLISHER = {Springer, Heidelberg},
      YEAR = {2011},
     PAGES = {xxx+963},
      ISBN = {978-3-540-42688-2},
   MRCLASS = {14H10 (32G15)},
  MRNUMBER = {2807457},
MRREVIEWER = {E.\ Looijenga},
       DOI = {10.1007/978-3-540-69392-5},
       URL = {https://doi.org/10.1007/978-3-540-69392-5},
}

@article{Balinski65,
  title={Integer programming: Methods, uses, computations},
  author={Balinski, Michel Louis},
  journal={Management Science},
  volume={12},
  number={3},
  pages={253--313},
  year={1965},
  publisher={INFORMS}
}

@article{BELL,
AUTHOR = {Brakensiek, Joshua and Eur, Christopher and Larson, Matt and
              Li, Shiyue},
     TITLE = {Kapranov degrees},
   JOURNAL = {Int. Math. Res. Not. IMRN},
  FJOURNAL = {International Mathematics Research Notices. IMRN},
      YEAR = {2025},
    NUMBER = {20},
     PAGES = {Paper No. rnaf306, 16},
      ISSN = {1073-7928,1687-0247},
   MRCLASS = {14H10 (05E14)},
  MRNUMBER = {4970179},
MRREVIEWER = {Zhenbo\ Qin},
       DOI = {10.1093/imrn/rnaf306},
       URL = {https://doi.org/10.1093/imrn/rnaf306},
}

@article{bruno11,
  title={On some fibrations of $\overline{M}_{0,n}$},
  author={Bruno, Andrea and Mella, Massimiliano},
  journal={arXiv:1105.3293},
  year={2011}
}

@article{CGM,
AUTHOR = {Cavalieri, Renzo and Gillespie, Maria and Monin, Leonid},
title = {Projective embeddings of {$\overline {M}_{0,n}$} and parking functions},
journal = {Journal of Combinatorial Theory, Series A},
volume = {182},
pages = {105471},
year = {2021},
issn = {0097-3165},
doi = {https://doi.org/10.1016/j.jcta.2021.105471},
url = {https://www.sciencedirect.com/science/article/pii/S0097316521000704},
author = {Renzo Cavalieri and Maria Gillespie and Leonid Monin}
}

@article {GiansiracusaGibney12,
    AUTHOR = {Giansiracusa, Noah and Gibney, Angela},
     TITLE = {The cone of type {$A$}, level 1, conformal blocks divisors},
   JOURNAL = {Adv. Math.},
  FJOURNAL = {Advances in Mathematics},
    VOLUME = {231},
      YEAR = {2012},
    NUMBER = {2},
     PAGES = {798--814},
      ISSN = {0001-8708,1090-2082},
   MRCLASS = {14E30 (14H10 14L24)},
  MRNUMBER = {2955192},
MRREVIEWER = {Scott\ R.\ Nollet},
       DOI = {10.1016/j.aim.2012.05.017},
       URL = {https://doi.org/10.1016/j.aim.2012.05.017},
}

@article{GGL,
 AUTHOR = {Gillespie, Maria and Griffin, Sean T. and Levinson, Jake},
     TITLE = {Degenerations and multiplicity-free formulas for products of
              {$\psi$} and {$\omega$} classes on {$\overline{M}_{0,n}$}},
   JOURNAL = {Math. Z.},
  FJOURNAL = {Mathematische Zeitschrift},
    VOLUME = {304},
      YEAR = {2023},
    NUMBER = {4},
     PAGES = {Paper No. 56, 37},
      ISSN = {0025-5874,1432-1823},
   MRCLASS = {14H10 (05A05 05C05 14C17 14N10)},
  MRNUMBER = {4613449},
MRREVIEWER = {Dragos\ Nicolae\ Oprea},
       DOI = {10.1007/s00209-023-03313-7},
       URL = {https://doi.org/10.1007/s00209-023-03313-7},
}

@article{GGL2,
  AUTHOR = {Gillespie, Maria and Griffin, Sean T. and Levinson, Jake},
     TITLE = {Lazy tournaments and multidegrees of a projective embedding of
              {$\overline M_{0,n}$}},
   JOURNAL = {Comb. Theory},
  FJOURNAL = {Combinatorial Theory},
    VOLUME = {3},
      YEAR = {2023},
    NUMBER = {1},
     PAGES = {Paper No. 3, 26},
      ISSN = {2766-1334},
   MRCLASS = {05E14 (05A19 05C05 05C85 14H10 14N10)},
  MRNUMBER = {4565290},
       DOI = {10.5070/c63160416},
       URL = {https://doi.org/10.5070/c63160416},
}

@article{DFJ,
  AUTHOR = {Dantzig, G. and Fulkerson, R. and Johnson, S.},
     TITLE = {Solution of a large-scale traveling-salesman problem},
   JOURNAL = {J. Operations Res. Soc. Amer.},
  FJOURNAL = {Journal of the Operations Research Society of America},
    VOLUME = {2},
      YEAR = {1954},
     PAGES = {393--410},
      ISSN = {0096-3984},
   MRCLASS = {90.0X},
  MRNUMBER = {70932},
MRREVIEWER = {H.\ W.\ Kuhn},
}

@article{Ed65,
   AUTHOR = {Edmonds, Jack},
     TITLE = {Maximum matching and a polyhedron with {$0,1$}-vertices},
   JOURNAL = {J. Res. Nat. Bur. Standards Sect. B},
  FJOURNAL = {Journal of Research of the National Bureau of Standards.
              Section B. Mathematics and Mathematical Physics},
    VOLUME = {69B},
      YEAR = {1965},
     PAGES = {125--130},
      ISSN = {0022-4340},
   MRCLASS = {90.50},
  MRNUMBER = {183532},
MRREVIEWER = {P.\ J.\ Higgins},
}

@article{Ed71,
   AUTHOR = {Edmonds, Jack},
     TITLE = {Matroids and the greedy algorithm},
   JOURNAL = {Math. Programming},
  FJOURNAL = {Mathematical Programming},
    VOLUME = {1},
      YEAR = {1971},
     PAGES = {127--136},
      ISSN = {0025-5610,1436-4646},
   MRCLASS = {90C05 (05B35)},
  MRNUMBER = {297357},
MRREVIEWER = {G.\ Bar},
       DOI = {10.1007/BF01584082},
       URL = {https://doi.org/10.1007/BF01584082},
}

@incollection{Fakhruddin12,
 AUTHOR = {Fakhruddin, Najmuddin},
     TITLE = {Chern classes of conformal blocks},
 BOOKTITLE = {Compact Moduli Spaces and Vector Bundles},
    SERIES = {Contemp. Math.},
    VOLUME = {564},
     PAGES = {145--176},
 PUBLISHER = {Amer. Math. Soc., Providence, RI},
      YEAR = {2012},
      ISBN = {978-0-8218-6899-7},
   MRCLASS = {14C17 (14D23 81T40)},
  MRNUMBER = {2894632},
MRREVIEWER = {Dmitry\ Kerner},
       DOI = {10.1090/conm/564/11148},
       URL = {https://doi.org/10.1090/conm/564/11148},
}

@article{F11,
  title={Cyclic covering morphisms on $\overline{M}_{0,n}$},
  author={Fedorchuk, Maksym},
  journal={arXiv preprint arXiv:1105.0655},
  year={2011}
}

@book{Frank2011,
  AUTHOR = {Frank, Andr\'as},
     TITLE = {Connections in Combinatorial Optimization},
    SERIES = {Oxford Lecture Series in Mathematics and its Applications},
    VOLUME = {38},
 PUBLISHER = {Oxford University Press, Oxford},
      YEAR = {2011},
     PAGES = {xxiv+639},
      ISBN = {978-0-19-920527-1},
   MRCLASS = {90-02 (05-02 05B35 05C85 05C90 90C27 90C57)},
  MRNUMBER = {2848535},
MRREVIEWER = {Mechthild\ Opperud},
}

@article{HK70,
 AUTHOR = {Held, Michael and Karp, Richard M.},
     TITLE = {The traveling-salesman problem and minimum spanning trees},
   JOURNAL = {Operations Res.},
  FJOURNAL = {Operations Research},
    VOLUME = {18},
      YEAR = {1970},
     PAGES = {1138--1162},
      ISSN = {0030-364X,1526-5463},
   MRCLASS = {90.30},
  MRNUMBER = {278710},
MRREVIEWER = {E.\ Gabowitsch},
       DOI = {10.1287/opre.18.6.1138},
       URL = {https://doi.org/10.1287/opre.18.6.1138},
}

@article{HK71,
AUTHOR = {Held, Michael and Karp, Richard M.},
     TITLE = {The traveling-salesman problem and minimum spanning trees {II}},
   JOURNAL = {Math. Programming},
  FJOURNAL = {Mathematical Programming},
    VOLUME = {1},
      YEAR = {1971},
    NUMBER = {1},
     PAGES = {6--25},
      ISSN = {0025-5610,1436-4646},
   MRCLASS = {90.30},
  MRNUMBER = {289119},
MRREVIEWER = {E.\ Gabowitsch},
       DOI = {10.1007/BF01584070},
       URL = {https://doi.org/10.1007/BF01584070},
}

@article{Kapranov93,
    AUTHOR = {Kapranov, M. M.},
     TITLE = {Veronese curves and {G}rothendieck-{K}nudsen moduli space
              {$\overline M_{0,n}$}},
   JOURNAL = {J. Algebraic Geom.},
  FJOURNAL = {Journal of Algebraic Geometry},
    VOLUME = {2},
      YEAR = {1993},
    NUMBER = {2},
     PAGES = {239--262},
      ISSN = {1056-3911,1534-7486},
   MRCLASS = {14H10 (14C05 14D99)},
  MRNUMBER = {1203685},
MRREVIEWER = {R.\ F.\ Lax},
}

@inproceedings{Karp1972,
  title={Reducibility among combinatorial problems},
  author={Karp, Richard M},
  booktitle={Complexity of Computer Computations: Proceedings of a symposium on the Complexity of Computer Computations, held March 20--22, 1972, at the IBM Thomas J. Watson Research Center, Yorktown Heights, NY},
  pages={85--103},
  year={1972},
  organization={Springer}
}

@article{Keel,
  AUTHOR = {Keel, Sean},
     TITLE = {Intersection theory of moduli space of stable {$N$}-pointed curves of genus zero},
   JOURNAL = {Trans. Amer. Math. Soc.},
  FJOURNAL = {Transactions of the American Mathematical Society},
    VOLUME = {330},
      YEAR = {1992},
    NUMBER = {2},
     PAGES = {545--574},
      ISSN = {0002-9947,1088-6850},
   MRCLASS = {14C15 (14C17 14H10)},
  MRNUMBER = {1034665},
MRREVIEWER = {Steven\ E.\ Landsburg},
       DOI = {10.2307/2153922},
       URL = {https://doi.org/10.2307/2153922},
}

@article{Knudsen83III,
  title={The projectivity of the moduli space of stable curves {III}: The line bundles on $\overline{M}_{g,n}$, and a proof of the projectivity of $\overline{M}_{g,n}$ in characteristic 0},
  author={Knudsen, Finn F},
  journal={Math. Scand.},
  volume={52},
  number={2},
  pages={200--212},
  year={1983}
}

@article{Knudsen83II,
  title={The projectivity of the moduli space of stable curves {II}: The stacks $\overline{M}_{g,n}$},
  author={Knudsen, Finn F},
  journal={Mathematica Scandinavica},
  volume={52},
  number={2},
  pages={161--199},
  year={1983},
  publisher={JSTOR}
}

@article {Knudsen83I,
    AUTHOR = {Knudsen, Finn Faye and Mumford, David},
     TITLE = {The projectivity of the moduli space of stable curves {I}: {P}reliminaries on ``det'' and ``{D}iv''},
   JOURNAL = {Math. Scand.},
  FJOURNAL = {Mathematica Scandinavica},
    VOLUME = {39},
      YEAR = {1976},
    NUMBER = {1},
     PAGES = {19--55},
      ISSN = {0025-5521,1903-1807},
   MRCLASS = {14H10 (14C05 14F05)},
  MRNUMBER = {437541},
MRREVIEWER = {P.\ E.\ Newstead},
       DOI = {10.7146/math.scand.a-11642},
       URL = {https://doi.org/10.7146/math.scand.a-11642},
}

@article{KT,
    AUTHOR = {Keel, Sean and Tevelev, Jenia},
     TITLE = {Equations for {$\overline M_{0,n}$}},
   JOURNAL = {Internat. J. Math.},
  FJOURNAL = {International Journal of Mathematics},
    VOLUME = {20},
      YEAR = {2009},
    NUMBER = {9},
     PAGES = {1159--1184},
      ISSN = {0129-167X,1793-6519},
   MRCLASS = {14H10 (14D20 14D24)},
  MRNUMBER = {2572597},
MRREVIEWER = {Arvid\ Siqveland},
       DOI = {10.1142/S0129167X09005716},
       URL = {https://doi.org/10.1142/S0129167X09005716},
}

@book{Korte08,
 AUTHOR = {Korte, Bernhard and Vygen, Jens},
     TITLE = {Combinatorial Optimization: Theory and Algorithms},
    SERIES = {Algorithms and Combinatorics},
    VOLUME = {21},
   EDITION = {Sixth},
 PUBLISHER = {Springer, Berlin},
      YEAR = {2018},
     PAGES = {xxi+698},
      ISBN = {978-3-662-56038-9; 978-3-662-56039-6},
   MRCLASS = {90-01 (05C85 68Q25 68R10 90C27 90C35)},
  MRNUMBER = {3753583},
       DOI = {10.1007/978-3-662-56039-6},
       URL = {https://doi.org/10.1007/978-3-662-56039-6},
}

@article{Lovasz1976,
 AUTHOR = {Lov\'asz, L.},
     TITLE = {On some connectivity properties of {E}ulerian graphs},
   JOURNAL = {Acta Math. Acad. Sci. Hungar.},
  FJOURNAL = {Acta Mathematica. Academiae Scientiarum Hungaricae},
    VOLUME = {28},
      YEAR = {1976},
    NUMBER = {1-2},
     PAGES = {129--138},
      ISSN = {0001-5954,1588-2632},
   MRCLASS = {05C35},
  MRNUMBER = {437391},
MRREVIEWER = {L.\ V.\ Quintas},
       DOI = {10.1007/BF01902503},
       URL = {https://doi.org/10.1007/BF01902503},
}

@inproceedings{Padberg1980OnTS,
AUTHOR = {Padberg, Manfred W. and Hong, Saman},
     TITLE = {On the symmetric travelling salesman problem: A computational
              study},
      NOTE = {Combinatorial optimization},
   JOURNAL = {Math. Programming Stud.},
  FJOURNAL = {Mathematical Programming Study},
    NUMBER = {12},
      YEAR = {1980},
     PAGES = {78--107},
      ISSN = {0303-3929},
   MRCLASS = {90C10 (05C35 90C35)},
  MRNUMBER = {571856},
       DOI = {10.1007/bfb0120888},
       URL = {https://doi.org/10.1007/bfb0120888},
}

@article{Raymond18,
 AUTHOR = {Raymond, Annie},
     TITLE = {The {T}ur\'an polytope},
   JOURNAL = {Electron. J. Combin.},
  FJOURNAL = {Electronic Journal of Combinatorics},
    VOLUME = {25},
      YEAR = {2018},
    NUMBER = {3},
     PAGES = {Paper No. 3.43, 20},
      ISSN = {1077-8926},
   MRCLASS = {05C35 (05C65 52B11 90C10)},
  MRNUMBER = {3853895},
       DOI = {10.37236/6555},
       URL = {https://doi.org/10.37236/6555},
}

@article{Reinke26,
AUTHOR = {Reinke, Bernhard and Silversmith, Rob},
     TITLE = {Stable curves and chromatic polynomials},
   JOURNAL = {Adv. Math.},
  FJOURNAL = {Advances in Mathematics},
    VOLUME = {501},
      YEAR = {2026},
     PAGES = {Paper No. 111110, 45},
      ISSN = {0001-8708,1090-2082},
   MRCLASS = {14H10 (05C31 14N10 52C35)},
  MRNUMBER = {5092811},
       DOI = {10.1016/j.aim.2026.111110},
       URL = {https://doi.org/10.1016/j.aim.2026.111110},
}

@book{Schrijver03,
  title={Combinatorial Optimization: Polyhedra and Efficiency},
  author={Schrijver, Alexander and others},
  volume={24},
  number={2},
  year={2003},
  publisher={Springer}
}

@article{Silversmith22,
  AUTHOR = {Silversmith, Rob},
     TITLE = {Cross-ratio degrees and perfect matchings},
   JOURNAL = {Proc. Amer. Math. Soc.},
  FJOURNAL = {Proceedings of the American Mathematical Society},
    VOLUME = {150},
      YEAR = {2022},
    NUMBER = {12},
     PAGES = {5057--5072},
      ISSN = {0002-9939,1088-6826},
   MRCLASS = {14N10 (05C30 14H10 14N35 14T15)},
  MRNUMBER = {4494586},
MRREVIEWER = {Magdalena\ Zielenkiewicz},
       DOI = {10.1090/proc/16016},
       URL = {https://doi.org/10.1090/proc/16016},
}

@article{silversmith24,
AUTHOR = {Silversmith, Rob},
     TITLE = {Cross-ratio degrees and triangulations},
   JOURNAL = {Bull. Lond. Math. Soc.},
  FJOURNAL = {Bulletin of the London Mathematical Society},
    VOLUME = {56},
      YEAR = {2024},
    NUMBER = {11},
     PAGES = {3518--3529},
      ISSN = {0024-6093,1469-2120},
   MRCLASS = {14N10 (14H10 14H81)},
  MRNUMBER = {4828030},
MRREVIEWER = {Roberto\ Mu\~noz},
       DOI = {10.1112/blms.13148},
       URL = {https://doi.org/10.1112/blms.13148},
}

@article{Turan41,
AUTHOR = {Tur\'an, Paul},
     TITLE = {Eine {E}xtremalaufgabe aus der {G}raphentheorie},
   JOURNAL = {Mat. Fiz. Lapok},
  FJOURNAL = {Matematikai \'es Fizikai Lapok},
    VOLUME = {48},
      YEAR = {1941},
     PAGES = {436--452},
      ISSN = {0302-7317},
   MRCLASS = {56.0X},
  MRNUMBER = {18405},
MRREVIEWER = {P.\ Erd\H os},
}

@article{Turan54,
 AUTHOR = {Tur\'an, P.},
     TITLE = {On the theory of graphs},
   JOURNAL = {Colloq. Math.},
  FJOURNAL = {Colloquium Mathematicum},
    VOLUME = {3},
      YEAR = {1954},
     PAGES = {19--30},
      ISSN = {0010-1354,1730-6302},
   MRCLASS = {56.0X},
  MRNUMBER = {62416},
MRREVIEWER = {W.\ T.\ Tutte},
       DOI = {10.4064/cm-3-1-19-30},
       URL = {https://doi.org/10.4064/cm-3-1-19-30},
}

@book{Traub2024,
  title={Approximation Algorithms for Traveling Salesman Problems},
 AUTHOR = {Traub, Vera and Vygen, Jens},
 PUBLISHER = {Cambridge University Press, Cambridge},
      YEAR = {2025},
     PAGES = {xiv+427},
      ISBN = {978-1-009-44541-2; [9781009445436]},
   MRCLASS = {90-01 (90C27)},
  MRNUMBER = {4835767},
MRREVIEWER = {Hans-Ulrich\ Simon},
}

\end{document}